\documentclass{article}

\usepackage{amsfonts}
\usepackage{amsmath,amsthm}
\usepackage{latexsym}
\usepackage{array}
\usepackage{amssymb}
\usepackage{graphicx}
\usepackage{enumerate}
\usepackage[english]{babel}
\usepackage{color}
\definecolor{marin}{rgb}   {0.,   0.3,   0.7}
\definecolor{rouge}{rgb}   {0.8,   0.,   0.}
\definecolor{sepia}{rgb}   {0.8,   0.5,   0.}
\usepackage[colorlinks,citecolor=marin,linkcolor=rouge,
            bookmarksopen,
            bookmarksnumbered
           ]{hyperref}

\newtheorem{lemma}{Lemma}[section]

\newtheorem{theorem}[lemma]{Theorem}
\newtheorem{proposition}[lemma]{Proposition}

\newtheorem{remark}[lemma]{Remark}
\newtheorem{example}[lemma]{Example}

\newtheorem{notation}[lemma]{Notation}
\newtheorem{definition}[lemma]{Definition}
\newtheorem{conclusion}[lemma]{Conclusion}
\numberwithin{equation}{section}
\newcommand{\QED}{\mbox{}\hfill \raisebox{-0.2pt}{\rule{5.6pt}{6pt}\rule{0pt}{0pt}}
          \medskip\par}

\newcommand{\dd}{\mathrm{d}}

\newcommand{\R}{\mathbb{R}}

\newcommand{\C}{\mathbb{C}}

\newcommand{\T}{\mathbb{T}}

\newcommand{\eps}{\varepsilon}
\newcommand{\ds}{\displaystyle}
\newcommand{\beq}{\begin{equation}}
\newcommand{\eeq}{\end{equation}}

\begin{document}

\title{Nonlinear Geometric Optics method based multi-scale numerical schemes
for highly-oscillatory transport equations}
\author{Nicolas Crouseilles\thanks{Inria, IRMAR, University of Rennes 1, Rennes, France} \and   Shi Jin \thanks{Department of Mathematics, University of Wisconsin-Madison, USA;  Department of Mathematics, Institute of Natural Sciences, MOE-LSEC and SHL-MAC,  Shanghai Jiao Tong University, Shanghai 200240, China
(sjin@wisc.edu). This author's research was also supported by NSF grants 1522184, 1107291 (KI-Net) and NSFC grant 91330203.}  \and Mohammed Lemou\thanks{CNRS, IRMAR, University of Rennes 1, Rennes, France} }
\date{}
\maketitle

\tableofcontents

\newpage 
\begin{abstract}
We introduce a new numerical strategy to solve a class of oscillatory transport PDE models  which is able to capture
accurately the solutions  without numerically resolving the high frequency oscillations {\em in both
space and time}.
Such PDE models arise in semiclassical modeling of quantum dynamics with band-crossings, and other
highly oscillatory waves. Our first main idea is to use the nonlinear geometric optics ansatz, which builds the
oscillatory phase into an independent variable. We then choose suitable initial data, based on the
Chapman-Enskog expansion, for the new model. For a scalar model, we prove that so constructed model will have certain
 smoothness, and consequently, for  a first
 order approximation scheme we prove  uniform error estimates independent of  the (possibly small) wave length. The method is extended to
 systems arising from a semiclassical model for surface hopping, a non-adiabatic quantum dynamic phenomenon. Numerous numerical examples demonstrate that the method has the desired properties.
\end{abstract}

\section{Introduction}

  Many partial differential equations for high frequency waves, in particular,  semiclassical models in quantum dynamics, take the form of systems of transport or Liouville equations with oscillatory
source terms describing interband quantum transitions that are
associated with chemical reactions, quantum tunnelling, Dirac points
in graphene, etc. \cite{morandi1, morandi2, jin-morandi}.
These
terms contain important quantum information, such as Berry connection and
Berry phase, which
are associated with quantum Hall effects \cite{Niu}.  Solving such
systems are computationally daunting since one needs to numerically
resolve the small wave length (denoted by a small parameter $\eps$ in
this paper), which can be prohibitively expansive.

 To efficiently solve a quantum system, or more generally high frequency
waves, a classical method is the geometric optics (GO) or WKB method,
which approximates the amplitude by a transport equation and phase by (nonlinear)
eiconal equation  \cite{Maslov}. This method allows
the computational mesh ($\Delta x$) and time step ($\Delta t$) independent
of $\eps$ \cite{EFO}. However the approximation  is not valid
beyond caustics, since the physically relevant solutions are multi-valued, rather than
the viscosity, solutions to the eiconal equation \cite{SMM, JL, ER1, JO03}.
Even the multi-valued solutions do not
describe accurately caustics, quantum tunnelling and other important
non-adiabatic quantum phenomena. A more accurate method, called the Gaussian
beam or Gaussian wave packet methods, originated independently in
seismology \cite{Hill90, popov}  and chemistry \cite{Hel81} communities (see also
recent developments in the math community \cite{Ral82, Hag, LQB, JWY1, LY10, Lub}), are more
accurate near caustics but need to use $\Delta x, \Delta t=O(\sqrt{\eps})$,
and have difficulties to handle singular potentials \cite{JWY} and non-adiabatic
band-crossing phenomena \cite{JQ}. For recent  overviews of computational high
frequency waves and semiclassical methods for quantum dynamics, see
\cite{ER2, JMS}.

  For problems that contain small or multiple time and space scales, another
framework that has found many success in kinetic and hyperbolic
problems is the {\it asymptotic-preserving} (AP) {\it schemes} \cite{Jin-AP}.
An AP scheme mimics the transition from a microscopic model to the macroscopic
one in the discrete setting, and as a result the scheme can capture
the macroscopic behavior correctly {\it without} resolving numerically the
small, microscopic behavior, thus can be used for all range of $\eps$ with fixed $\Delta x$ and $\Delta t$.
Based on solving one model--the microscopic one, an AP scheme  undergoes the numerical transition from the microscopic to the macroscopic scales {\it automatically}
 without the need to couple two different models at different scales, which is the bottleneck of
 most multiscale or multiphysical methods \cite{EE-HMM}. See recent reviews of AP methods in
\cite{Jin-Review, Degond-Review}. For high frequency wave problems, the AP framework
has found  successes only in dealing with time oscillations,
allowing $\Delta t>\!\!> O(\eps)$ \cite{BJM, Jin-Dirac, Bao1, Bao2, clm, cclm} for a number of
physical problems. Nevertheless, for high frequency waves, the most
difficult challenge is the spatial oscillations which unfortunately demands
$\Delta x=O(\eps)$, an impossible task in high space dimensions.
One earlier work in this direction was in \cite{MR2831055}, 
by using the WKB-basis functions the method
allows $\Delta x=O(1)$, but so far this approach has only been developed for one-dimensional stationary Schr\"odinger equation  (without time oscillations).

In this paper we introduce a general AP approach  to efficiently solve a family of oscillatory waves in which
the phase oscillations depend {\it on both time and space}.
The problem under study takes  the following form
\begin{equation}
\label{eqf}
\begin{array}{l}
\ds \partial_t u + \sum_{k=1}^d A_k(x)\partial_{x_k} u + R(u) = \frac{i}{\varepsilon} E(t,x)  D u + Cu,  \ \ \  \ \ \  t\geq 0; \ \  x\in  \Omega\subset\mathbb{R}^d, \\
\ds
u(0,x)= f_{in}(x, \beta(x)/\varepsilon), \ \ \ \ x\in \Omega\subset \mathbb{R}^d,
\end{array}
\end{equation}
where $u=u(t, x)\in \mathbb{C}^n$, and $A_k$, $D$ and $C$ are given $n\times n$ real matrices.
$R(u): \mathbb{C}^n \mapsto \mathbb{C}^n$, is the source
term independent of  $\eps$, the small dimensionless wave length. The quantity $E$ 
is  a real valued scalar function. 
The initial data $f_{in}$
may have an oscillatory dependence with an initial  phase $\beta(x)/\varepsilon$, and 
in this case we will assume that the dependence of $f_{in}$ on this phase is periodic.
Many semiclassal  models for quantum dynamics may be written in this general form (see an example surface hopping \cite{jin-morandi},
graphene \cite{morandi2}, and quantum dynamics in periodic lattice \cite{morandi1}),  in which $E\geq 0$ is the gap between different energy bands. 
Some high frequency wave equations also have the form of (\ref{eqf}) \cite{ER2}.  In this paper, we assume periodic boundary condition in space such that $\Omega=[0, 1]^d$, although the method can be extended to
more general boundary conditions.

Our main idea
is to use the {\it nonlinear geometric optics} (NGO), which has been
widely studied at the theoretical level in the mathematical community last century for
nonlinear hyperbolic conservation laws \cite{DM, Joly, Majda, RK}. The NGO approach builds
the oscillatory phase as an independent variable.
Specifically, one introduces
a function  $U : (t,x,\tau) \in (0, T)\times \mathbb{R}^d\times (0, 2\pi) \rightarrow U(t, x,\tau) \in \mathbb{C}^n$,
which is $2\pi$-periodic with respect to the last variable $\tau \in (0, 2\pi)$,
and  coincides with the solution $u$  of \eqref{eqf} in the sense
\begin{equation}
\label{Ftof}
U(t, x, S(t, x)/\varepsilon) = u(t, x).
\end{equation}
We then transfer the original equation  into a {\it linear} equation
for the phase $S(t,x)$--unlike in GO where the phase equation is the nonlinear eikonal equation which
triggers caustics--coupled with an equation on the profile
$U$. The main interest of our reformulation is twofold. First the equation
for $S$  is linear and does not depend on $\eps$,  which makes its numerical approximation simple, accurate and
inexpensive. Second, thanks to the additional degree of freedom in $U$,
one can, and {\it needs to}, choose suitable initial data such that  $U(t,x,\tau)$ is
uniformly bounded in $\eps$ up to certain order of derivatives in time and space,
which can then be solved numerically efficiently: such initial
data can be generated
by  utilizing the
classical Chapman-Enskog expansion \cite{Cer} as was done
in \cite{cclm, clm} to efficiently compute the time oscillations.
As a result, our method is AP, {\it in both space and time},  which allows correct solutions even when
$\Delta x, \Delta t >\!\!> O(\eps)$.

 The paper is organized as follows. In section 2, we present in details the strategy for highly oscillatory scalar equations in one dimension. In particular we reformulate the problem into a new one with an additional dependence  on a
 well-chosen oscillation phase. We prove that this augmented problem is smooth enough in both space and time with respect to the oscillation parameter. Based on this reformulation, we construct
a numerical scheme for which we prove  that the order of accuracy is also uniform in $\eps$. Numerical results are performed to assert the efficiency of our method.
Then, in section 3, we extend the strategy to a class of oscillatory hyperbolic systems with an application to a  semiclassical surface hopping model. A conclusion is finally given in section 4.

We remark that our approach, although presented here only in one space dimension, can be generalized to higher
space dimension straightforwardly. This will be the subject of a future work.

\section{One dimensional scalar equations}
As an illustrative example, we first consider the following model satisfied by $u(t, x)\in \C, \ x\in [0, 1],\  t\geq 0$,
\begin{equation}
\label{eq_scalar}
\partial_t  u + c(x)\partial_x u + r(u) = \frac{ia(x)}{\varepsilon} u,\quad  \;\; u(0, x)=u_0(x),
\end{equation}
where the functions $u_0$, $a$, $c$ and $r$ are given. Periodic boundary conditions are also considered in space.
In some cases, we will allow the initial data to be oscillatory
$$u_0(x)=f_{in}(x, \beta(x)/\varepsilon) \equiv f_{in}(x,\tau) \qquad {\text {with}} \quad \tau=\frac{\beta(x)}{\varepsilon},$$
where $\beta$ is a given function and $f_{in}$ is supposed to be
periodic with respect to the second variable $\tau$.
More precise technical  assumptions on all these functions will be made later on.

\subsection{The linear case}
First, we focus on the linear case $r(u)=\lambda u$ where $\lambda$ is a constant .
In this case, we expand the initial data with respect to the periodic variable $\tau$:
$$
u_0(x)= \sum_{k\in \mathbb{Z}} f_k(x) e^{ik\beta(x)/\varepsilon},
$$
which allows to restrict the study of \eqref{eq_scalar} to the following equations
\begin{equation}
\label{eq_scalar_b}
\partial_t  u_k + c(x)\partial_x u_k + \lambda u_k = \frac{ia(x)}{\varepsilon} u_k, \;\; u_k(t=0, x)=f_k(x)e^{ik\beta(x)/\varepsilon}.
\end{equation}
Indeed the linearity of the equation allows the use of the superposition principle,  and the solution of \eqref{eq_scalar} can be recovered by $u(t, x)=\sum_k u_k(t, x)$.

Since \eqref{eq_scalar_b} is linear, one can apply the standard {\em Geometric Optics} (GO) by injecting the ansatz
$u_k(t, x)=\alpha_k(t, x)e^{ikS(t, x)/\varepsilon}$ into \eqref{eq_scalar_b}.
This gives
$$
\partial_t \alpha_k + c(x)\partial_x \alpha_k +\lambda \alpha_k+\frac{i}{\varepsilon}[\partial_t S+c(x)\partial_x S]\alpha_k
= \frac{ia(x)}{\varepsilon}\alpha_k.
$$
To remove the terms in $1/\eps$, one can impose the following equations on $\alpha$ and $S$
\begin{eqnarray*}
\partial_t \alpha_k + c(x)\partial_x \alpha_k+\lambda \alpha _k= 0, & \alpha(0, x)=f_k(x), \nonumber\\
\partial_t S + c(x)\partial_x S= a(x), & S(0, x)=\beta(x).
\end{eqnarray*}
This gives rise to {\it non oscillatory} solutions $S$ and $\alpha_k$ which can
be solved numerically quite efficiently without numerically resolving the
small time and wavelength scales  of size $O(\varepsilon)$.

\subsection{The nonlinear case}
When $r$ is nonlinear, the superposition principle cannot be applied anymore and the GO approach does not work.
Then, we utilize what was called in the literature the {\it nonlinear
geometric optics} (NGO) ansatz, namely, introduce a function $U(t, x, \tau)$ which depends on an additional {\it periodic} variable $\tau$,
and satisfies
\begin{equation}
\label{udiag}
U(t, x, S(t, x)/\varepsilon)=u(t, x),
\end{equation} with $u$ solution to \eqref{eq_scalar}.
The equation satisfied by $U$ writes
$$
\partial_t U+c(x)\partial_x U +\frac{1}{\varepsilon}[\partial_t S+c(x)\partial_x S]\partial_\tau U + r(U) = \frac{ia(x)}{\varepsilon}U.
$$
To get a constant period in the independent variable $\tau$,
 we should impose the following equation on $S$
\begin{equation}
\label{eq_S}
\partial_t S+c(x)\partial_x S = a(x), \;\; S(0, x)=\beta(x).
\end{equation}
Then, we deduce the equation for $U$ to
\begin{equation}
\label{eq_U}
\partial_t U+c(x)\partial_x U + r(U) = -\frac{a(x)}{\varepsilon}(\partial_\tau U-iU), \;\;\;\;\; U(0, x, \beta(x) /\varepsilon)=u_0(x).
\end{equation}

It is clear that solving (\ref{eq_S}) and (\ref{eq_U})  with any initial data  $U(0, x, \tau)$ satisfying
$U(0, x, \beta(x) /\varepsilon)=u_0(x)$ allows one to recover the desired original solution to
(\ref{eq_scalar}) through  relation  (\ref{udiag}). Due to the extra dimension introduced by $U$,  there are infinitely many possible choices of such initial data.  We will choose one--which is essential--that provides  a "smooth enough" solution with respect
to $\varepsilon$. Indeed, from  numerical point of view, this smoothness property is of paramount importance when one wants to get a numerical scheme with a {\em uniform accuracy} with respect to $\eps$.
We will consider two cases for which this choice is possible and 
a uniform smoothness with respect to $\eps$ of the phase $S$ and the profile $U$ can be obtained at any order.
The first case  is very simple  since the models on $S$ and $U$ do not depend on $\eps$, while the second case requires more care and is presented in the next subsection.

{\bf Case 1}: $a\equiv 0$ with possibly oscillatory initial data:  $u_0(x)= f_{in}(x,\beta(x)/\eps)$. The equation
on the phase $S$  is given by (\ref{eq_S}) and the equation on $U$ reduces to
$$ \partial_t U+c(x)\partial_x U + r(U) = 0, \;\;\;\;\; U(0, x, \tau)=f_{in}(x,\tau).$$
The two equations on $S$ and $U$ clearly do not depend on $\eps$ and therefore numerical schemes on $S$ and $U$ will not be restricted by the small values of $\eps$.

\medskip

{\bf Case 2}: $a\ne 0$ with non-oscillatory initial data $u_0(x)=\alpha(x)$. In this case the equation for the phase $S$ is still given by (\ref{eq_S}) with $\beta(x)=0$, and the equation on
$U$ is also  given by \eqref{eq_U}, which can be written in terms of
\[
V= e^{-i\tau} U
\]
where $V$ solves
\begin{equation}
\label{eq_V}
\partial_t V+c(x)\partial_x V + e^{-i\tau}r(e^{i\tau}V) = -\frac{a(x)}{\varepsilon}\partial_\tau V.
\end{equation}
Since the only condition one has to impose on $V$ is  $V (0, x, 0)=u_0(x)$ (recall that boundary conditions are imposed in $x$),
this gives some  freedom for the choice of the
initial data for $V$, and the strategy of this choice will be developed in the next subsection.

\begin{remark}
Our approach works with either oscillatory initial data, or oscillatory source. It does not apply to problems where oscillations are generated from both initially data and sources. See discussions in section \ref{mixed}.
\end{remark}

\subsection{A suitable initial condition}
\label{subsection_IC}
Considering the non-oscillatory initial data (Case 2),
one needs initial data $V(0,x,\tau)$  {\em for all $\tau$} to solve equation (\ref{eq_V}) .
Since the only condition we have  to ensure  is
$ V (0, x,0)=u_0(x)= \alpha(x)$, there is a degree of freedom in choosing the expression of $V(0,x,\tau)$.
{\it The central idea here is to choose it
in such a way that the solution $V$ is non-oscillatory in $\varepsilon$ (up to
certain order of time-space derivatives).}
To show this construction, we will deal in this section with the case of non-oscillatory initial data, in which case
we can construct $V(0, x, \tau)$ in such a way that the time-space derivatives of $V$ (up to second order) are
uniformly bounded with respect to $\varepsilon$.

Since the initial condition in \eqref{eq_scalar} takes the non-oscillatory form $u(0, x)= \alpha(x)$,
we have $S(0,x)=0$.
Following \cite{cclm, clm}, we will construct ''well-prepared initial data'' $V(0, x,\tau)$
which ensures that the high-order time and space derivatives of $V$ are also bounded
uniformly in $\eps$, together with  \begin{equation}
V(0, x, \tau=0 )=\alpha(x).
\label{V_init1}
\end{equation}
As a consequence, we will see that  the so-obtained  initial data for $V$ provides a  non-oscillatory  solution
$V$ and  allows the construction of  numerical schemes with a uniform accuracy with respect to $\varepsilon$.
Below, we will describe  the method,  and we refer to \cite{cclm, clm} for more details.
Note that this type of initial data is obtained by formally expanding the solution in terms of $\eps$ in the spirit of
the well-known Chapman-Enskog expansion in kinetic theory \cite{Cer}.
For this purpose, we first introduce  the notations
$$
{\cal L}g = \partial_\tau g, \quad   \Pi  g  = \frac{1}{2\pi} \int_0^{2\pi} g(\tau) d\tau,
$$
and let
$$
V^0= \Pi V, \quad V^1= ({\cal I}- \Pi) V. 
$$
The operator ${\cal L}$ is skew-symmetric  on $L^2(d\tau)$, its kernel is the space of functions
which do not depend on $\tau$, and  $\Pi$ is the $L^2(d\tau)$-orthogonal
projector onto the kernel of ${\cal L}$. The operator ${\cal L}$ is invertible on the set of functions having zero average in the variable $\tau$, and
$$
{\cal L}^{-1} g =  ({\cal I}-\Pi) \int_0^\tau g(\sigma) d\sigma = \int_0^\tau g(\sigma) d\sigma +\frac{1}{2\pi} \int_0^{2\pi} \sigma g(\sigma) d\sigma,
$$
for all $g\in L^2(d\tau)$ such that $\Pi g=0$. In particular

$$
{\cal L}^{-1} (e^{i\tau})= -ie^{i\tau} , \qquad  {\cal L}^{-1} (e^{-i\tau})= i e^{-i\tau}.
$$
In addition,  ${\cal L}^{-1}$ is a bounded operator since $\|{\cal L}^{-1} g \|_{L^\infty_\tau} \leq C \| g \|_{L^\infty_\tau}$
and $\|{\cal L}^{-1} g \|_{L^2_\tau} \leq C \| g \|_{L^2_\tau}$ for all $g\in L^2(d\tau)$ such that $\Pi g=0$.

We now apply $\Pi$ and ${\cal I}-\Pi$ to (\ref{eq_V}) to  get
\begin{equation}
\label{eq_V01}
\partial_t V^0+c(x)\partial_x V^0 + {\cal  \Pi} [e^{-i\tau}r(e^{i\tau}(V^0+V^1))] = 0,
\end{equation}
\begin{equation}
\label{eq_V11}
\partial_t V^1+c(x)\partial_x V^1 + ({\cal I}-\Pi)[e^{-i\tau}r(e^{i\tau}(V^0+V^1))] = -\frac{a(x)}{\varepsilon}\partial_\tau V^1.
\end{equation}
In particular (\ref{eq_V11}) gives
\begin{equation}
\label{res_V11}
V^1(t, x, \tau)= -\eps (a(x))^{-1}  {\cal L}^{-1} ({\cal I}-\Pi)[e^{-i\tau}r(e^{i\tau}V^0)] + O(\eps^2),
\end{equation}
which, when applied to (\ref{eq_V01}) and letting $\eps \to 0$, formally yields
\begin{equation}
\label{eq_V0_limit}
\partial_t V^0+c(x)\partial_x V^0 + \Pi [e^{-i\tau}r(e^{i\tau}(V^0))] = 0.
\end{equation}
Then, one gets the following expansion for $V$
\begin{eqnarray}
\hspace{-1cm}V(t,x,\tau)&\!\!\!=\!\!\!&V^0(t, x)+V^1(t, x, \tau)\nonumber\\
\label{ce_Vt}
&\!\!\!\!=\!\!\!& V^0(t, x)- \eps (a(x))^{-1} {\cal L}^{-1} ({\cal I}-\Pi)[e^{-i\tau}r(e^{i\tau}V^0(t, x))] + O(\eps^2).
\end{eqnarray}
To avoid oscillations in $\eps$, this expansion should be satisfied at $t=0$ as well.
Evaluating \eqref{ce_Vt} at $t=\tau=0$ and using (\ref{V_init1}), this means that
$$
\alpha(x) = V^0(t=0, x)- \eps   (a(x))^{-1}  {\cal L}^{-1} ({\cal I}-\Pi)[e^{-i\tau}r(e^{i\tau}V^0(t=0, x))] {\Big|_{\tau=0}}  + O(\eps^2),
$$
or
\begin{equation}
\label{ce_V0}
V^0(t=0, x)= \alpha(x)+ \eps  (a(x))^{-1}  {\cal L}^{-1} ({\cal I}-\Pi)[e^{-i\tau}r(e^{i\tau}\alpha(x))]{\Big|_{\tau=0}} + O(\eps^2).
\end{equation}
Evaluating \eqref{ce_Vt} at $t=0$ and using  \eqref{ce_V0}
finally yields our suitable initial data:
\begin{eqnarray}
 V(0,x,\tau)&=& \alpha(x)+ \frac{\eps}{a(x)} \Big[ G(0, \alpha) -G(\tau, \alpha)\Big],  \nonumber\\
 \label{init_prepared}
 \mbox{ with }\!\!\!\!\!\!\!  && \!\!\!\!\!\!\! G(\tau, \alpha) =  {\cal L}^{-1} ({\cal I}-\Pi)[e^{-i\tau}r(e^{i\tau}\alpha(x))].
\end{eqnarray}
We will see that this approach not only allows one to capture the main oscillations with phase $S(t,x)$ and
amplitude $O(1)$, it also allows to capture oscillations of amplitude $\eps$.

\begin{remark}
The Chapman-Enskog expansion is conducted only to generate the suitable initial data for $V$, while the equations for  $S$ and $V$ (or $U$) are not asymptotically truncated. This guarantees that our method,
unless other asymptotic methods, {\it is accurate for all $\eps$}.
\end{remark}

\begin{remark}
\label{remark_asymptotic_model}
Another interest of the augmented formulation above is the following. Because of oscillations, in general the solution to (\ref{eq_scalar}) cannot converge strongly but only weakly when $\eps\to0$. However, if $S$ is the solution to (\ref{eq_S}), then one may have (in some appropriate functional space)
$$  \exp(-iS(t,x)/\eps) u(t,x) = V(t,x,S(t,x)/\eps)\ \  \mbox{converges strongly to } \ \   \bar u =  \bar u(t,x)$$
where $\bar u$ satisfies
\beq
\label{asymptotic_model}
\partial_t \bar u+c(x)\partial_x \bar u + \Pi e^{-i\tau}r(e^{i\tau}\bar u) = 0, \ \qquad   \bar u(0,x)= u_0(x).
\eeq
We will demonstrate this numerically in section \ref{num_test}.  It will be interesting to investigate rigorously this strong convergence, but this task is beyond the scope of
this paper and is deferred to a future work.
\end{remark}

Now we will give a theorem which states that, up to the second order, time and space derivatives of $V$ are bounded uniformly in $\eps$, provided that the initial condition is given by (\ref{init_prepared}).
First, we make the following assumptions on  $r$ and $a$.
\paragraph{Assumption on $r$.}
We assume in the sequel that $r:\R\rightarrow \R$ is a twice differentiable function on $\R$
whose two first derivatives are bounded.

\paragraph{Assumption on $a$.} $a:[0, 1]\rightarrow \R$ is a ${\cal C}^2$
function satisfying   $a(x)\geq a_0>0, \forall x\in [0, 1]$.
\bigskip

Note that the assumption on $a$ excludes the case of band crossing. There is  no such a restriction for the
actual numerical scheme, as will be demonstrated numerically later.
\bigskip

 For the sake of simplicity, we will restrict ourselves in the following  theorem to the case of a constant transport coefficient $c(x)=c >0$. The extension to a non-constant $c(x)$ can easily be derived following the lines of the proof.

\bigskip
\bigskip

\begin{theorem}
\label{theorem1}
Let $V$ be the solution of \eqref{eq_V} on $[0, T]$. We consider the initial data \eqref{init_prepared} and
periodic boundary condition in $x$ and $\tau$ variables.
Then, the time and spatial derivatives of $V$
are bounded uniformly in $\eps\in ]0,1]$, that is,  $\exists C>0$ independent of $\eps$ such that $\forall t\in [0, T]$
$$
 \|\partial_t^p V(t)\|_{L^\infty_{\tau, x}} \leq C, \;\; \mbox{ and } \;\; \|\partial_x^p V(t)\|_{L^\infty_{\tau, x}} \leq C, \mbox{ for } p=0, 1, 2,
 $$
 and
 $$
 \|\partial^2_{xt} V(t)\|_{L^\infty_{\tau, x}} \leq C, \;\; \mbox{ and } \;\; \|\partial_\tau V(t)\|_{L^\infty_{\tau, x}} \leq C.
$$
\end{theorem}

\begin{proof}[Proof of Theorem \ref{theorem1}]
First, we make the following change of variable
\beq
\label{V_to_W}
V(t, x, \tau)=W(S(t, x), x, \tau)
\eeq
where $ S$ satisfies
\[
 \partial_t S + c\partial_x S = a(x),\qquad S(0, x)=0.
\]
In the one-dimensional case, one can write the exact solution for $S$
\beq
\label{sol_exact_S}
S(t, x) = \frac{1}{c}\Big[ A(x)-A(x-ct) \Big], \;\; \mbox{ with } A(x)=\int_0^x a(y) dy.
\eeq
Observe that, for all $x\in \R $, the phase  $S(t,x)$ is an increasing function in $t$, and this property remains true in higher space dimensions. This means that  the map $t\mapsto s=S(t,x)$ can be seen
as a change of variable in time. Note that $s \in [0, \bar T]$, with
$\bar T\leq (2/c)\|A\|_{L^\infty_x} \leq (2/c)\|a\|_{L^\infty_x}$.
Then, $W(s, x, \tau)$ satisfies
\begin{eqnarray}
&&\partial_s W+\frac{c}{a(x)}\partial_x W + \frac{1}{a(x)} e^{-i\tau} r(e^{i\tau}W) = -\frac{1}{\eps}\partial_\tau W, \nonumber\\
&&W(0, x, \tau)= V(0,x,\tau)= \alpha(x)+ \frac{\eps}{a(x)}  \Big[ G(0, \alpha) -G(\tau, \alpha)\Big], \nonumber\\
 &&\mbox{with } G(\tau, \alpha) = {\cal L}^{-1} ({\cal I}-\Pi)[e^{-i\tau}r(e^{i\tau}\alpha(x))].
\label{eq_W}
\end{eqnarray}
We then need to prove the following result for $W$.
\begin{proposition}
\label{theoremW}
Let $W$ be the solution of \eqref{eq_W} on $[0, \bar T]$, $\bar T>0$, with periodic boundary condition in $x$ and $\tau$.
Then, up to the second order, the time and spatial derivatives of $W$
are bounded uniformly in $\eps\in ]0,1]$, that is,  $\exists C>0$ independent of $\eps$ such that, $\forall s\in [0, \bar T]$
$$
 \|\partial_s^p W(s)\|_{L^\infty_{\tau, x}} \leq C, \;\; \mbox{ and } \;\; \|\partial_x^p W(s)\|_{L^\infty_{\tau, x}} \leq C, \mbox{ for } p=0, 1, 2,
 $$
 and
 $$
 \|\partial^2_{st} W(s)\|_{L^\infty_{\tau, x}} \leq C.
$$
\end{proposition}

We first claim that this proposition implies the result of  Theorem \ref{theorem1}.
Indeed, using the following straightforward relations
\begin{eqnarray*}
\partial_t V &=& \partial_s W \partial_t S, \\
\partial^2_t V &=& \partial^2_s W (\partial_t S)^2 + \partial_s W\partial_t^2 S,\\
\partial_xV &=& \partial_s W \partial_x S +\partial_x W, \\
\partial^2_x V &=& \partial^2_s W (\partial_x S)^2 + \partial_s W\partial_x^2 S + \partial^2_{xs}W (\partial_x S) + \partial_x^2W,\\
\partial^2_{tx} V &=& \partial_s^2 W(\partial_t S) (\partial_x S) + \partial_s W \partial^2_{tx} S+\partial^2_{sx} W\partial_t S,
\end{eqnarray*}
and the fact that $S(t, x)$ given by \eqref{sol_exact_S} has its time and space derivatives uniformly
bounded (since $S$ does not depend on $\eps$), the estimates stated in Theorem \ref{theorem1} hold.
\end{proof}

Now, we focus on the proof of Proposition \ref{theoremW}.
To this aim, we start with the following elementary lemma:
\begin{lemma}
\label{lemma}
Consider the following ordinary differential equation
$$
\frac{d y^\varepsilon}{dt} = \alpha^\varepsilon(t) y^\varepsilon + \beta^\varepsilon(t), \;\; t\in [0, T], \;\; y^\eps(0)=y^\eps_0,
$$
with $y^\varepsilon, \alpha^\varepsilon, \beta^\varepsilon:[0, T] \rightarrow {\cal H}$, ${\cal H}$
being a Banach algebra space. Assume  that  there exists a constant $C>0$
independent of $\varepsilon$ such that $\|\alpha^\varepsilon(t)\|_{\cal H}\leq C$ and $\|\beta^\varepsilon(t)\|_{\cal H}\leq C$, $\forall t\in [0, T]$.
Then there exists a constant $M>0$ independent of $\eps$ such that  $\|y^\varepsilon(t)\|_{\cal H}\leq M  \| y^\varepsilon_0\|_{\cal H}$, $\forall t\leq T$.
\end{lemma}

\begin{proof}[Proof of Lemma \ref{lemma}]
The proof of Lemma \ref{lemma} follows from the exact solution which can be written as
$$
y^\varepsilon(t) = y^\varepsilon_0 \exp\left(\int_0^t \alpha^\varepsilon(s) ds\right) + \int_0^t \left[\exp\left(\int_s^t \alpha^\varepsilon(u) du\right) \beta^\varepsilon(s) \right]ds.
$$
Then, one can straightforwardly  deduce from the assumptions on $\alpha$ and $\beta$, that  $\forall t\in [0, T]$, $\|y^\varepsilon(t) \|_{\cal H} \leq M \|y^\varepsilon_0 \|_{\cal H}$.

\end{proof}

\medskip

In the following, we will use this lemma to prove that the time and space derivatives of $W$, solution of \eqref{eq_W},
are uniformly bounded in the space ${\cal H}=L^\infty_{\tau, x}$ of bounded functions of $\tau$ and $x$.

\begin{proof}[Proof of Proposition \ref{theoremW}]
We first introduce the characteristic equations associated with \eqref{eq_W},
$$
\dot{x}(s)=\frac{c}{a(x(s))}, \;\; x(0)=x_0; \;\;   \dot{\tau}(s) =\frac{1}{\eps}, \;\; \tau(0)=\tau_0.
$$
Since  $c$ is a constant, this system can be solved analytically.
For $A$ given by \eqref{sol_exact_S}, since $a>0$, $A$ is a strictly increasing function, thus its inverse $A^{-1}$
is well defined so that the solution of the differential system is
$$
x(s)=A^{-1}(A(x_0)+cs), \;\; x(0)=x_0; \;\; \;\;   \tau(s) =\tau_0+\frac{s}{\eps}, \;\; \tau(0)=\tau_0.
$$
This motivates the following change of variables
\begin{equation}
\label{change}
\tilde{x}=A^{-1}(A(x)+cs), \;\;   \tilde{\tau} =\tau+\frac{s}{\eps},
\end{equation}
which enables to filter out the transport terms in \eqref{eq_W}.

\paragraph{Existence and estimate of $W$\\}
Using \eqref{change}, we write the equation satisfied by
$\widetilde{W}(s, x, \tau)=W(s, A^{-1}(A(x)+cs), \tau+\frac{s}{\eps})$ to get
\beq
\label{eq_Wtilde}
\partial_s \widetilde{W} =  -\frac{1}{a({\widetilde x})}e^{-i{\widetilde\tau}} r(e^{i{\widetilde\tau}} \widetilde{W}),
\eeq
where $\tilde x$ and $\tilde \tau$ are given by  \eqref{change}, and with the initial condition $\widetilde{W}(0, {x}, {\tau})=V(0, x, \tau)$ given by \eqref{eq_W}. Since $r$ is a  Lipschitz function,
according to the Cauchy-Lipschitz theorem in the Banach space $L^\infty_{\tau, x}$,
equation \eqref{eq_Wtilde} has a unique global solution. Furthermore, since
$a(\tilde x) \geq a_0>0$, we have
\begin{eqnarray*}
\|\widetilde{W}(s)\|_{L^\infty_{\tau, x}} &\leq & \|{V}(0)\|_{L^\infty_{\tau, x}} +C \Big\|\int_0^s \frac{1}{a(\tilde{x})} e^{-i\tilde\tau} r(e^{i\tilde\tau} \widetilde{W}(\sigma)) d\sigma \Big\|_{L^\infty_{\tau, x}} \nonumber\\
&\leq & \|{V}(0)\|_{L^\infty_{\tau, x}} + C \int_0^s (1+ \ \|\widetilde{W}(\sigma)\|_{L^\infty_{\tau, x}}) d\sigma,
\end{eqnarray*}
and using the Gronwall lemma,  we get
\begin{equation}
\label{estimateW}
\sup_{\varepsilon >0} \| \widetilde{W}(s)\|_{L^\infty_{\tau, x}} \leq C (1+  \|V(0) \|_{L^\infty_{\tau, x}}).
\end{equation}
Since the initial data $V(0)$, given by  \eqref{init_prepared}, is uniformly bounded with respect to $\eps$, we deduce that  $\widetilde W$, and then $W$, are also uniformly bounded.

We now  prove that  time and space derivatives, up to the second order, are bounded uniformly in $\varepsilon$.

\paragraph{Estimate of the first time derivative\\}
The first derivative $W_1=\partial_s W$ satisfies
\begin{equation}
\label{eq_V1}
\partial_s W_1 + \frac{c}{a(x)}\partial_x W_1 +\frac{1}{a(x)}r'(e^{i\tau} W) W_1 = -\frac{1}{\varepsilon}\partial_\tau W_1.
\end{equation}
As before, we consider the change of variables \eqref{change} so that
$\widetilde{W}_1(s, x, \tau)=W_1(s, A^{-1}(A(x)+cs), \tau+\frac{s}{\eps})$ solves
$$
\partial_s \widetilde{W}_1 +\frac{1}{a(\tilde x)} r'(e^{i\tilde\tau} \widetilde{W}) \widetilde{W}_1=0.
$$
This equation enters in the framework of Lemma \ref{lemma} with
$\alpha^\varepsilon(s)=(1/a(\tilde x))r'(e^{i\tilde\tau} \widetilde{W}(s))$
and $\beta^\varepsilon(s)=0$. Under the assumption on $r$ and the estimate \eqref{estimateW} on $\widetilde{W}$,
$\widetilde{W}_1(s)$ is  uniformly bounded provided that $\widetilde{W}_1(0)$ is bounded. From  equation \eqref{eq_W} on $W$ at $s=0$, one has
\begin{eqnarray*}
\widetilde{W}_1(0) &= & \partial_s W(0) \nonumber\\
&=&-\frac{1}{\varepsilon}\partial_\tau W(0) - \frac{c}{a(x)}\partial_x W(0) -\frac{1}{a(x)}e^{-i\tau}r(e^{i\tau} W(0)).
\end{eqnarray*}
The last term is bounded using the assumption on $r$ and the fact that $W(0)$ is bounded.
Considering the first term, with the choice of $W(0)$ given by \eqref{eq_W}, one gets
\begin{eqnarray*}
\Big |\frac{1}{\eps} \partial_\tau W(0) \Big | &=& \Big |\frac{1}{a(x)}(I-\Pi)(e^{-i\tau}r(e^{i\tau}\alpha(x))) \Big | \nonumber\\
&\leq & C  \| r(e^{i\tau}\alpha) \|_{L^\infty_{x, \tau}} \leq C (1+ \| \alpha \|_{L^\infty_{x}}).
\end{eqnarray*}
Then, for the second term, using the notations in \eqref{eq_W},
\begin{eqnarray}
\partial_x W(0)  & = & \alpha'(x) -\frac{\eps a'(x)}{a(x)^2} \Big[ G(0, \alpha)-G(\tau, \alpha)\Big] \nonumber\\
\label{dxV0}
&& + \frac{\eps}{a(x)} \Big[  H(0, \alpha)- H(\tau, \alpha)\Big],
\end{eqnarray}
where $H(\tau, \alpha) =:\partial_x G(\tau, \alpha)= {\cal L}^{-1}({\cal I}-\Pi) (r'(e^{i\tau}\alpha(x)) \alpha'(x))$.
Using the fact that ${\cal L}^{-1}$ is a bounded operator
on ${\cal C}^0(\T)$ and the assumptions on $r$, we get $\|G\|_{L^\infty_{x,\tau}} \leq C (1+\|\alpha\|_{L^\infty_{x}})$
and $ \|H\|_{L^\infty_{x,\tau}} \leq  C (1+\|\alpha\|_{L^\infty_{x}})\|\alpha' \|_{L^\infty_{x}}$. Finally, we have
\begin{eqnarray}
\Big | \partial_x W(0) \Big | &\leq& | \alpha'(x)| + C \Big|  \frac{\eps a'(x)}{a(x)^2} \Big |(1+  \|\alpha\|_{L^\infty_{x}})+\frac{\eps}{a}  \|\alpha' \|_{L^\infty_{x}} \nonumber\\
\label{estimate_dxV0}
&\leq & \frac{C\eps}{a_0^2} + \| \alpha \|_{W^{1, \infty}_{x}} \Big[1 +  C\frac{\eps}{a_0^2} + C\frac{\eps}{a_0} \Big]
\leq C.
\end{eqnarray}
We then conclude that $\widetilde{W}_1(0)=\partial_s W(0)$ is uniformly bounded.
As a consequence,
\begin{equation}
\label{estimateW1}
\|\widetilde{W}_1(s)\|_{L^\infty_{\tau, x}} \leq C,
\end{equation}
and then $\partial_s W$ is uniformly bounded in $\eps$.

\paragraph{Estimate of the second time derivative\\}
We proceed in a analogous way for $W_2=\partial_s^2 W=\partial_s W_1$
by taking the time derivative of equation \eqref{eq_V1}, which satisfies
\begin{equation}
\label{eq_V2}
\partial_s W_2 + \frac{c}{a(x)}\partial_x W_2 + \frac{1}{a(x)}e^{i\tau} r'' (e^{i\tau} W)W_1^2 +  \frac{1}{a(x)} r'(e^{i\tau} W) W_2 = -\frac{1}{\varepsilon} \partial_\tau W_2.
\end{equation}
Using \eqref{change}, $\widetilde{W}_2(t, x, \tau)=W_2(s, A^{-1}(A(x)+cs), \tau+\frac{s}{\eps})$ satisfies
$$
\partial_s \widetilde{W}_2 +\frac{1}{a(\tilde x)}r'(e^{i\tilde \tau} \widetilde{W}) \widetilde{W}_2+ \frac{1}{a(\tilde x)}e^{i\tilde \tau} r''(e^{i\tilde \tau} \widetilde{W}) \widetilde{W}_1^2 =0.
$$
We now use Lemma \ref{lemma} with $\alpha^\varepsilon(s)=(1/a(\tilde x))r'(e^{i\tilde\tau} \widetilde{W})$
and $\beta^{\varepsilon}(s)=(1/a(\tilde x)) e^{i\tilde\tau} r''(e^{i\tilde\tau} \widetilde{W}) \widetilde{W}_1^2$.
Using \eqref{estimateW},  \eqref{estimateW1}
and the assumption on $r$, one deduces that $\alpha^\varepsilon(s)$ and $\beta^\varepsilon(s)$
are uniformly bounded,
and one just needs to prove that the initial data $\widetilde{W}_2(0)= {W}_2(0)$ is bounded uniformly in $\eps$. Let us recall the expression of $W_2(0)$ using \eqref{eq_V1} at $s=0$
\beq
\label{eq_V20}
W_2(0) = \partial_s W_1(0) = -\frac{c}{a(x)}\partial_x W_1(0) - \frac{1}{a(x)}r'(e^{i\tau} W(0)) W_1(0) -\frac{1}{\varepsilon}\partial_\tau W_1(0).
\eeq
The second term in the right hand side is uniformly bounded since $W$ and $W_1$ are bounded.
The terms $\frac{1}{\varepsilon}\partial_\tau W_1(0)$ and $\frac{c}{a(x)}\partial_x W_1(0)$
need more care. First we write $W_1(0)$ using \eqref{eq_W}
\begin{eqnarray}
W_1(0)  &=&  \partial_s W(0) = -\frac{c}{a(x)}\partial_x W(0)-\frac{1}{a(x)}e^{-i\tau} r(e^{i\tau} W(0)) -\frac{1}{\eps}\partial_\tau W(0)  \nonumber\\
 &=&   -\frac{c}{a(x)}\partial_x W(0)-\frac{1}{a(x)}e^{-i\tau} r(e^{i\tau} W(0)) + \frac{1}{a(x)}({\cal I}-\Pi)[e^{-i\tau}r(e^{i\tau}\alpha(x))]  \nonumber\\
\label{eq_W1zero}
&=& -\frac{c}{a(x)}\partial_x W(0)+ \frac{1}{a(x)}e^{-i\tau}[ r(e^{i\tau} \alpha(x)) - r(e^{i\tau} W(0))]\nonumber \\
& &- \frac{1}{a(x)}\Pi[e^{-i\tau}r(e^{i\tau}\alpha(x))].
\end{eqnarray}
Then, we compute $(1/\eps)\partial_\tau W_1(0)$ using \eqref{eq_W}
\begin{eqnarray*}
\frac{1}{\eps}\partial_\tau W_1(0) &=& -\frac{c}{a\eps }\partial^2_{\tau x} W(0)+\frac{1}{a \eps} \partial_\tau\Big[ e^{-i\tau} [r(e^{i\tau} \alpha) -  r(e^{i\tau} W(0))]       \Big]  \nonumber\\
&=& -\frac{c}{a\eps} \Big[  \frac{\eps a'}{a^2} ({\cal I}-\Pi)[e^{-i\tau} r(e^{i\tau}\alpha)] - \frac{\eps}{a} ({\cal I}-\Pi)[\alpha' r'(e^{i\tau}\alpha)  ]  \Big] \nonumber\\
&&+\frac{1}{a\eps} \left\{ -i e^{-i\tau} [r(e^{i\tau} \alpha) - r(e^{i\tau} W(0))] + i\alpha r'( e^{i\tau} \alpha) \right. \nonumber\\
&& \left.- r'(e^{i\tau} W(0))[iW(0) + \partial_\tau W(0)]  \right\} \nonumber\\
&=& - \frac{c a'}{a^3} ({\cal I}-\Pi)[e^{-i\tau} r(e^{i\tau}\alpha)] +  \frac{c}{a^2} ({\cal I}-\Pi)[\alpha' r'(e^{i\tau}\alpha)  ]   \nonumber\\
&&- \frac{i }{a\eps} e^{-i\tau} [r(e^{i\tau} \alpha) - r(e^{i\tau} W(0))]
+\frac{1}{a\eps} r'( e^{i\tau} W(0)) \frac{\eps}{a}({\cal I}-\Pi)[e^{-i\tau}r(e^{i\tau}\alpha)]\nonumber\\
&&+\frac{i}{a\eps}\Big[ r'(e^{i\tau}\alpha)\alpha  - r'(e^{i\tau}W(0))W(0)    \Big].
\end{eqnarray*}
Using
the fact that ${\cal L}^{-1}$ is a bounded operator on ${\cal C}^0(\T)$ and the smoothness assumptions on $r$, one gets
\begin{eqnarray*}
\Big |\frac{1}{\eps}\partial_\tau W_1(0) \Big | &\leq &  \frac{C}{a_0^3}(1+ |\alpha|) +  \frac{C}{a_0^2}
+\frac{1}{a_0\eps} \; C | \alpha - W(0)|  +  \frac{C}{a_0^2} (1+|\alpha|)  + \frac{1}{a_0\eps} C | \alpha - W(0)|. \nonumber\\
\end{eqnarray*}
From \eqref{eq_W}, since $W(0)-\alpha = \frac{\eps}{a}[G(0, \alpha)-G(\tau, \alpha)]$ with $G$ uniformly bounded, one has
\begin{eqnarray*}
\Big |\frac{1}{\eps}\partial_\tau W_1(0) \Big | &\leq & C.
\end{eqnarray*}
We now focus on the term $\partial_x W_1(0)$ (which is the first term in  the rhs of \eqref{eq_V20}) and take the derivative of  \eqref{eq_W1zero}
with respect to $x$
\begin{eqnarray}
\partial_x W_1(0) &=& \frac{c a'}{a^2}\partial_x W(0) - \frac{c}{a}\partial_x^2 W(0) -\frac{a'}{a^2}e^{-i\tau}[ r(e^{i\tau}\alpha)  - r(e^{i\tau}W(0))] \nonumber\\
&& +\frac{1}{a}[ \alpha' r'(e^{i\tau}\alpha)  - \partial_x W(0)r'(e^{i\tau}W(0))] \nonumber\\
\label{dxdtV}
&&+\frac{a'}{a^2}\Pi[e^{-i\tau} r(e^{i\tau}\alpha)]
- \frac{1}{a}\Pi[\alpha' r'(e^{i\tau}\alpha)].
\end{eqnarray}
All the terms except the first one $\partial_x^2 W(0)$ have been estimated previously. Express $\partial_x^2 W(0)$
by taking the derivative of  \eqref{dxV0} with respect to $x$
\begin{eqnarray*}
\partial_x^2 W(0) & = & \alpha'' - \eps\frac{ a'' a - 2 (a')^2 }{a^3} [G(\tau=0, \alpha)-G(\tau, \alpha)] \nonumber\\
&& -2\frac{\eps a'}{a^2} \partial_x [G(\tau=0, \alpha)-G(\tau, \alpha)] + \frac{\eps }{a}\partial_x^2[G(\tau=0, \alpha)-G(\tau, \alpha)], \nonumber\\
\end{eqnarray*}
where $G(\tau, \alpha)={\cal L}^{-1}({\cal I}-\Pi) (e^{-i\tau}r(e^{i\tau}\alpha(x)))$ is used. Using the properties of ${\cal L}^{-1}$
and of $r$ and $a$, one can estimate $\partial_x^2 W(0)$
\begin{eqnarray}
|\partial_x^2 W(0)| &\leq& \|\alpha'' \|_{L^\infty_x}+ \frac{C\eps}{a_0^3}(1+ \|\alpha \|_{L^\infty_x}) + \frac{C\eps}{a_0^2}  \|\alpha' \|_{L^\infty_x} +  \frac{\eps }{a_0}  \|\alpha'' \|_{L^\infty_x} \nonumber\\
\label{estimate_dxxV0}
&\leq &  C.
\end{eqnarray}
Thus, from \eqref{eq_V20},  $W_2(0)$ is uniformly bounded and
we conclude with Lemma  \ref{lemma} that
\begin{equation}
\label{estimateW2}
\|\widetilde{W}_2(s)\|_{L^\infty_{\tau, x}} \leq C,
\end{equation}
so that $\partial_s^2 W$ is uniformly bounded.

\paragraph{First space derivative\\}
The function $Y_1=\partial_x W$ solves the following equation
\begin{equation}
\label{eq_Y1}
\partial_s Y_1 + \frac{c}{a}\partial_x Y_1 -\frac{ca'}{a^2}Y_1-\frac{a'}{a^2}e^{-i\tau }r(e^{i\tau} W)+ \frac{1}{a}r'(e^{i\tau} W)Y_1 =- \frac{1}{\varepsilon} \partial_\tau Y_1.
\end{equation}
 Again the function $\widetilde{Y}_1(t, x, \tau)=Y_1(s, A^{-1}(A(x)+cs), \tau+\frac{s}{\eps})$ satisfies
$$
\partial_s \widetilde{Y}_1 -\frac{ca'(\tilde x)}{a(\tilde x)^2}\widetilde{Y}_1-\frac{a'(\tilde x)}{a(\tilde x)^2}e^{-i\tilde\tau}r(e^{i\tilde\tau}\widetilde{W})+ \frac{1}{a(\tilde x)}r'(e^{i\tilde\tau} \widetilde{W})\widetilde{Y}_1 =0,
$$
where we used the notation in \eqref{change}.
Then, one can use Lemma \ref{lemma} with
$\alpha^\varepsilon(s) =- (1/a(\tilde x)) r'(e^{i\tilde\tau} \widetilde{W}) +(ca'(\tilde x)/a(\tilde x)^2)$ and
$\beta^\varepsilon(s)=(a'(\tilde x)/a(\tilde x)^2)e^{-i\tilde\tau} r(e^{i\tilde\tau} \widetilde{W})$
for which one gets  a uniform estimate thanks to the assumption on $r$ and the previous estimate on $\widetilde{W}$.  Moreover,
it has already been proved in \eqref{estimate_dxV0}
that the initial condition $\widetilde{Y}_1(0)=\partial_x W(0)$ (with $W(0)$ given by \eqref{init_prepared}) is uniformly bounded.
As a consequence,
\begin{equation}
\label{estimate_Z1}
\|\widetilde{Y}_1(s)\|_{L^\infty_{\tau, x}} \leq C.
\end{equation}
so that  $Y_1(s)=\partial_x W(s)$ is also uniformly bounded with repsect to $\eps$.

\paragraph{Second space derivative\\}
Considering $Y_2=\partial_x^2 W=\partial_x Y_1$, one gets  from \eqref{eq_Y1}
\begin{eqnarray}
\label{eq_Y2}
\partial_s Y_2 &+& \frac{c}{a}\partial_x Y_2- \frac{2ca'}{a^2} Y_2 - c\frac{a'' a-2(a')^2}{a^3}Y_1-\frac{2a'}{a^2} r'(e^{i\tau} W)Y_1 +\frac{1}{a}e^{i\tau} r''(e^{i\tau} W)Y_1^2 \nonumber\\
&& \hspace{-1cm}+\frac{1}{a}r'(e^{i\tau} W)Y_2
-\frac{a'' a-2(a')^2}{a^3}e^{-i\tau}r(e^{i\tau}W) =- \frac{1}{\varepsilon} \partial_\tau Y_2 .
\end{eqnarray}
Again, with the change of variable \eqref{change}, the equation for
$\widetilde{Y}_2(s, x, \tau )=Y_2(t, x+ct, \frac{1}{c\varepsilon} [A(x+ct)-A(x)]+\tau)$ writes (using the notations introduced above)
$$
\partial_t \widetilde{Y}_2 =\alpha^\eps (s) \widetilde{Y}_2 +\beta^\eps(s),
$$
where
\begin{eqnarray*}
\alpha^\eps (s)&=& \frac{2ca'(\tilde x)}{a(\tilde x)^2} -\frac{1}{a(\tilde x)}r'(e^{i\tilde\tau} \widetilde{W})\nonumber\\
\beta^\eps(s) &=&c\frac{a''(\tilde x) a(\tilde x)-2(a'(\tilde x))^2}{a(\tilde x)^3}\widetilde{Y}_1+\frac{2a'(\tilde x)}{a(\tilde x)^2} r'(e^{i\tilde\tau} \widetilde{W})\widetilde{Y}_1 -\frac{1}{a(\tilde x)}e^{i\tilde\tau} r''(e^{i\tilde\tau} \widetilde{W})\widetilde{Y}_1^2
\nonumber\\
&&+\frac{a''(\tilde x) a(\tilde x)-2(a'(\tilde x))^2}{a(\tilde x)^3}e^{-i\tilde\tau}r(e^{i\tilde\tau}\widetilde{W}),
\end{eqnarray*}
are uniformly bounded thanks to the previous estimates and the properties of $r$.
One then needs to check that the initial condition $\widetilde{Y}_2(0)=Y_2(0)= \partial_x^2 W(0)$ is uniformly bounded, but we recall  that $\partial_x^2 W(0)$ has already been estimated in \eqref{estimate_dxxV0}.
Hence, we can conclude
\begin{equation}
\label{estimate_Z2}
\| \widetilde{Y}_2(s)\|_{L^\infty_{\tau, x}} \leq C,
\end{equation}
so that $Y_2$ is uniformly bounded with respect to $\eps$.

\paragraph{Mixed space-time derivative\\}
By differentiating  \eqref{eq_Y1} with respect to $s$, one gets the equation satisfied by $Y_3=\partial_s Y_1=\partial_{sx}^2 W$
\begin{equation}
\label{eq_Y3}
\partial_t Y_3 + \frac{c}{a}\partial_x Y_3 - \frac{ca'}{a^2} Y_3 -\frac{a'}{a^2}r'(e^{i\tau}W)W_1
+\frac{1}{a} e^{i\tau} r''(e^{i\tau} W) W_1 Y_1 +\frac{1}{a}r'(e^{i\tau}W)Y_3 = -\frac{1}{\varepsilon}\partial_\tau Y_3.
\end{equation}
Using \eqref{change}, $\widetilde{Y}_3(t, x, \tau)=Y_3(s, A^{-1}(A(x)+cs), \tau+\frac{s}{\eps})$ satisfies
\begin{equation}
\label{eq_Y3tilde}
\partial_t \widetilde{Y}_3  - \frac{ca'(\tilde x)}{a(\tilde x)^2} \widetilde{Y}_3
-\frac{a'(\tilde x)}{a(\tilde x)^2}r'(e^{i\tilde\tau}\widetilde{W})\widetilde{W}_1
+\frac{1}{a(\tilde x)} e^{i\tilde\tau} r''(e^{i\tilde\tau} \widetilde{W}) \widetilde{W}_1 \widetilde{Y}_1 +\frac{1}{a(\tilde x)}r'(e^{i\tilde\tau}\widetilde{W})\widetilde{Y}_3= 0.
\end{equation}
We use Lemma \ref{lemma} with $\alpha^\varepsilon(s)=(ca'(\tilde x)/a(\tilde x)^2)-(1/a(\tilde x)) r'(e^{i\tilde\tau}\widetilde{W})$
and $\beta^{\varepsilon}(s)=(a'(\tilde x)/a(\tilde x)^2) r'(e^{i\tilde\tau} \widetilde{W}) \widetilde{W}_1-(1/a(\tilde x)) e^{i\tilde\tau} r''(e^{i\tilde\tau} \widetilde{W}) \widetilde{W}_1 \widetilde{Y}_1$ which are uniformly bounded thanks to the previous estimates and the smoothness  of $r$.
One now needs to check that the initial condition $\partial_{sx} W(0)$ is uniformly bounded. To do so, apply \eqref{eq_Y1} at $s=0$ to get
\begin{eqnarray}
\partial^2_{sx}W(0) &=& -\frac{c}{a}\partial_x Y_1(0) + \frac{ca'}{a^2}Y_1(0) + \frac{a'}{a^2}e^{-i\tau}r(e^{i\tau}W(0))
-\frac{1}{a}r'(e^{i\tau}W(0)) Y_1(0) - \frac{1}{\eps} \partial_\tau Y_1(0)\nonumber\\
&=& -\frac{c}{a}\partial^2_x W(0) + \frac{ca'}{a^2}\partial_x W(0) + \frac{a'}{a^2}e^{-i\tau}r(e^{i\tau}W(0))
\nonumber\\
&&-\frac{1}{a}r'(e^{i\tau}W(0)) \partial_x W(0) - \frac{1}{\eps} \partial_{x \tau} W(0)\nonumber\\
\end{eqnarray}
All the terms except the last one have already been estimated previously. Let us focus on
$\frac{1}{\eps} \partial^2_{x \tau} W(0)$
\begin{eqnarray*}
\Big| \frac{1}{\eps} \partial^2_{x \tau} W(0) \Big| &=&\Big|-  \frac{1}{\eps}  \partial_x \Big[ \frac{\eps}{a} ({\cal I}-\Pi)[e^{-i\tau} r(e^{i\tau}\alpha) ] \Big]\Big| \nonumber\\
&=&  \Big| \frac{a'}{a^2} ({\cal I}-\Pi)[e^{-i\tau} r(e^{i\tau}\alpha) ]  - \frac{1}{a} ({\cal I}-\Pi)[ r'(e^{i\tau}\alpha)\alpha' ]\Big| \nonumber\\
&\leq &   \frac{|a'|}{a_0^2} (1 +\|\alpha\|_{L^\infty_x})  + \frac{C}{a_0}\|  \alpha' \|_{L^\infty_x}, \nonumber\\
\end{eqnarray*}
which is uniformly bounded. Then, we conclude that
\begin{equation}
\label{estimateZ3}
\|\widetilde{Y}_3(t)\|_{L^\infty_{\tau, x}} \leq C,
\end{equation}
so that $\partial_{sx}^2 W$ is bounded uniformly  with respect to $\eps$.

\end{proof}

\subsection{A numerical scheme for the equation of $V$}
In this section, we focus on the numerical analysis of a first order (in time and space) numerical scheme
for the equation of $V$:  \eqref{eq_V}, with initial condition \eqref{init_prepared}, in which the variable $\tau$ is kept continuous. For the sake of simplicity, our convergence analysis will be restricted to constant and nonnegative convection term $c(x)=c$. We also assume that $a(x)\geq a_0$ for all $x\in [0,1]$, where $a_0$ is a positive constant.  The numerical tests, however, will not be restricted by these assumptions.

We define a uniform grid in time $t_n=n\Delta t$
 in a time interval $[0, T]$, $n=0, 1, \dots, N$, $N\Delta t=T$ and in space $x_j=j\Delta x$, $j=0, 1, \dots, N_x$, $\Delta x=1/N_x$
in the spatial interval $[0, 1]$ (recall that periodic boundary conditions are considered in space).
Denoting $V^n_j(\tau)\approx V(t_n, x_j, \tau)$, $\tau\in \T=[0, 2\pi]$,
the numerical scheme for \eqref{eq_V} advances the solution from $t_n$ to $t_{n+1}$ through
\begin{equation}
\label{scheme_V}
\frac{V_j^{n+1} - V_j^n}{\Delta t} + c\frac{V^n_j - V^n_{j-1}}{\Delta x} +e^{-i\tau}r(e^{i\tau} V_j^n)
= -\frac{a(x_j)}{\varepsilon}\partial_\tau V_j^{n+1},
\end{equation}
with $V^0_j(\tau)$ given by \eqref{init_prepared}. In this scheme, the spatial discretization is the upwind scheme.
In the following theorem, we prove that the numerical scheme \eqref{scheme_V}, with initial data \eqref{init_prepared},  is not only a first order approximation of \eqref{eq_V}, but more importantly this first order approximation is {\em uniform } in $\eps$.

\bigskip

\begin{theorem}
\label{theorem2}
Assume that $a:[0, 1]\rightarrow \R$ is a ${\cal C}^2$ function satisfying $a(x)\geq a_0 >0, \forall x\in [0, 1]$,
and that the CFL condition $c\Delta t/\Delta x<1$
is satisfied.
Then $\exists C>0$ independent of $ \Delta t, \Delta x$ and $\eps$, such  that
\begin{equation}
\ds \sup_{\varepsilon\in ]0, 1]} \| V(t_n, x_j) - V^n_j \|_{L^\infty_\tau} \leq C(\Delta t+\Delta x),
\end{equation}
for all $n=0, \dots, N$, $ n\Delta t\leq T$ and all $j=0, \dots, N_x$.
\end{theorem}

\begin{proof}
First, we check that the scheme is well defined. Assuming $V^n_j$ is periodic in $\tau$ with period $2\pi$,
then it is easy to see that $V^{n+1}_j$ is also periodic with period $2\pi$. We proceed in an analogous way as in \cite{cclm} and introduce the operator $Q_j$ defined from ${\cal C}^1(\T)$ into ${\cal C}^0(\T)$ by
\beq
\label{qj}
\forall \tau \in \T, \;\; (Q_j g)(\tau) = g(\tau) + \frac{a(x_j)}{\varepsilon}(\partial_\tau g)(\tau).
\eeq
This operator is invertible and its inverse can be written as
$$
\forall \tau \in \T, \;\; (Q_j^{-1} g)(\tau) = \frac{\varepsilon/a(x_j)}{\exp(\varepsilon/a(x_j) 2\pi)-1} \int_\tau^{\tau+2\pi}
\exp(\varepsilon/a(x_j) (\theta-\tau))g(\theta) d\theta.
$$
Moreover, since
$$
 \frac{\varepsilon/a(x_j)}{\exp(\varepsilon/a(x_j) 2\pi)-1} \int_0^{2\pi}
\exp(\varepsilon/a(x_j) \theta) d\theta = 1,
$$
one gets
\begin{equation}
\label{qim1}
\forall g\in {\cal C}^0(\T), \;\; \|Q_j^{-1} g\|_{L^{\infty}_\tau } \leq \| g \|_{L^{\infty}_\tau}, \;\; \forall j=0, \cdots, N_x.
\end{equation}

Let us now study the error of the scheme \eqref{scheme_V}.
First, we perform Taylor expansions in time and space. On the one side, one gets
\begin{eqnarray*}
\nonumber
&& \frac{V(t_{n+1}, x_j)-V(t_{n}, x_j)  }{\Delta t}\\
 &=&
\partial_t V(t_{n+1}, x_j) - \frac{1}{\Delta t}\int_{t_n}^{t_{n+1}} (t-t_n) \partial_t^2 V(t, x_j) dt  \nonumber\\
&=& -c\partial_x V(t_{n+1}, x_j) - e^{-i\tau}r(e^{i\tau} V(t_{n+1}, x_j)) - \frac{a(x_j)}{\varepsilon}\partial_\tau V(t_{n+1}, x_j)\nonumber\\
&& - \frac{1}{\Delta t}\int_{t_n}^{t_{n+1}} (t-t_n) \partial_t^2 V(t, x_j) dt. 
\end{eqnarray*}
And on the other side,
\begin{eqnarray*}
\frac{V(t_{n}, x_j)-V(t_{n}, x_{j-1})  }{\Delta x} &=& \partial_x V(t_n, x_j)
-\frac{1}{\Delta x} \int_{x_{j-1}}^{x_j} (x-x_{j-1})\partial_x^2 V(t_n, x) dx.
\end{eqnarray*}
Gathering both equalities gives
\begin{align}
&\hspace{-2cm} \frac{V(t_{n+1}, x_j)-V(t_{n}, x_j)  }{\Delta t} +c \frac{V(t_{n}, x_j)-V(t_{n}, x_{j-1})  }{\Delta x} \nonumber\\
&=
-c\partial_x V(t_{n+1}, x_j)
- e^{-i\tau}r(e^{i\tau} V(t_{n+1}, x_j)) + (R_t)^n_j\nonumber\\
&\hspace{0.4cm}- \frac{a(x_j)}{\varepsilon}\partial_\tau V(t_{n+1}, x_j)  + c\partial_x V(t_n, x_j) +(R_x)^n_j,  \nonumber\\
&= -c\partial_x [V(t_{n+1}, x_j) - V(t_n, x_j)]  + (R_t)^n_j + (R_x)^n_j \nonumber\\
& \hspace{0.4cm}- e^{-i\tau}r(e^{i\tau} V(t_{n+1}, x_j))- \frac{a(x_j)}{\varepsilon}\partial_\tau V(t_{n+1}, x_j)\nonumber\\
&= -c \int_{t_n}^{t_{n+1}} \partial_{xt}^2 V(t, x_j) dt  + (R_t)^n_j + (R_x)^n_j \nonumber\\
\label{eq_Uex}
& \hspace{0.4cm}- e^{-i\tau}r(e^{i\tau} V(t_{n+1}, x_j))- \frac{a(x_j)}{\varepsilon}\partial_\tau V(t_{n+1}, x_j),
\end{align}
where $(R_x)^n_j=-\frac{c}{\Delta x} \int_{x_{j-1}}^{x_j} (x-x_{j-1})\partial_x^2 V(t_n, x) dx$ and
$(R_t)^n_j=- \frac{1}{\Delta t}\int_{t_n}^{t_{n+1}} (t-t_n) \partial_t^2 V(t, x_j) dt $ denote the integral remainders of the previous Taylor expansions.

Denoting by ${\cal E}^n_j =V(t_n, x_j) - V^n_j$ the error, the difference between \eqref{eq_Uex}
and \eqref{scheme_V} gives
\begin{eqnarray}
\frac{{\cal E}^{n+1}_j-{\cal E}^n_j}{\Delta t} + c\frac{{\cal E}^{n}_j-{\cal E}^n_{j-1}}{\Delta x} \hspace{-0.5cm}&&
+ e^{-i\tau}\Big[ r(e^{i\tau} V(t_{n+1}, x_j))-r(e^{i\tau} V^n_j)\Big] +c \int_{t_n}^{t_{n+1}} \partial_{xt}^2 V(t, x_j) dt \nonumber\\
\label{eq_E}
&&= (R_x)^n_j + (R_t)^n_j -\frac{a(x_j)}{\varepsilon} \partial_\tau {\cal E}^{n+1}_j.
\end{eqnarray}
Here, we focus on the third term of the left hand side of \eqref{eq_E}
\begin{eqnarray*}
&& \hspace{-1.cm}e^{-i\tau}\Big[r(e^{i\tau} V(t_{n+1}, x_j))-r(e^{i\tau} V^n_j)\Big] \nonumber\\
&=& e^{-i\tau}\Big[r(e^{i\tau} V(t_{n+1}, x_j)) - r(e^{i\tau}V(t_{n}, x_j))\Big] + e^{-i\tau}\Big[r(e^{i\tau} V(t_{n}, x_j))- r(e^{i\tau}V^n_j)\Big] \nonumber\\
&=&  \int_{t_n}^{t_{n+1}}  r'(e^{i\tau} V(t, x_j)) \partial_t V(t, x_j) dt +{\cal E}^n_j \int_0^1 r'(e^{i\tau} V^n_j + t e^{i\tau}{\cal E}^n_j) dt  \nonumber\\
&=:& (R_{1})^n_j + {\cal E}^n_j  (R_{2})^n_j.
\end{eqnarray*}
Hence, from \eqref{eq_E}, one can express the error  ${\cal E}^{n+1}_j$
with respect to ${\cal E}^{n}_j$ and ${\cal E}^{n}_{j-1}$
\begin{eqnarray}
\label{enp1}
{\cal E}^{n+1}_j &=&Q_j^{-1}\left[  \Big(1-\frac{c\Delta t}{\Delta x} - (R_{2})^n_j \Delta t\Big) {\cal E}^n_j
+ \frac{c\Delta t }{\Delta x} {\cal E}^n_{j-1} + \Delta t g^n_j\right],
\end{eqnarray}
where $Q_j$ is given by \eqref{qj}
and $g^n_j = - (R_{1})^n_j + (R_x)^n_j + (R_t)^n_j -c\int_{t_n}^{t_{n+1}} \partial_{xt}^2 V(t, x_j) dt$.
First, using Theorem 2.1 and the assumption on $r$, one has
\begin{eqnarray}
\| g^n_j\|_{L^\infty_\tau} &\leq & \| (R_{1})^n_j \|_{L^\infty_\tau} + \| (R_{x})^n_j \|_{L^\infty_\tau} + \| (R_t)^n_j \|_{L^\infty_\tau} +C\Delta t \nonumber\\
\label{estimate_g}
&\leq & C\Delta t +C\Delta x +C\Delta t + C\Delta t,
\end{eqnarray}
where $C$ is some positive constant which does not depend on $\varepsilon$. Second,
we now consider the $L^\infty_\tau$ norm of \eqref{enp1}
and use \eqref{qim1} to get (under the CFL condition $c\Delta t<\Delta x$)
\begin{eqnarray*}
\|{\cal E}^{n+1}_j  \|_{L^\infty_\tau} &\leq & \Big\|  \Big(1-\frac{c\Delta t}{\Delta x} +(R_{2})^n_j \Delta t\Big)  {\cal E}^n_j
+ \frac{c\Delta t }{\Delta x} {\cal E}^n_{j-1} + \Delta t g^n_j \Big\|_{L^\infty_\tau}\nonumber\\
&\leq & \max_j  \|   {\cal E}^n_j \|_{L^\infty_\tau} +\Delta t \| (R_{2})^n_j  \|_{L^\infty_\tau}\|   {\cal E}^n_j \|_{L^\infty_\tau}  + C\Delta t(\Delta t+\Delta x)\nonumber\\
&\leq & \| {\cal E}^n \|_{L^\infty_\tau}(1+C\Delta t) + C\Delta t(\Delta t+\Delta x),
\end{eqnarray*}
where we denote by ${\cal E}^n = \max_{j=0, \dots, N_x} |{\cal E}^n_j|$.
A discrete Gronwall lemma enables one to get the required uniform estimate
$$
\| {\cal E}^{n} \|_{L^\infty_\tau} \leq C (\Delta t +\Delta x) \exp(C T). 
$$

\end{proof}

\begin{remark}
Higher order methods can be constructed by expanding to higher power in $\eps$ in the Chapman-Enskog expansion
presented in subsection \ref{subsection_IC} and using higher order approximation scheme in time and space.
We will not elaborate on this further in this paper.
\end{remark}

\subsection{The cases where the oscillations come from both initial data and sources}\label{mixed}

In this part, we discuss the case where we may have high-oscillations in both the non linear PDE model and in the initial data. This corresponds to cases where we have a model of type \eqref{eq_scalar} with $a\ne 0$ and  $u(0, x)=f_{in}(x,\beta(x)/\varepsilon)$, the function $f_{in}=f_{in}(x,\tau)$ being periodic in $\tau$.  The previous strategy
cannot be applied in this general case and the uniform boundness of time and space derivatives of $V$ at arbitrary order is no more garanteed.
However, we can  ensure that the first time and space derivatives are uniformly bounded.

\paragraph{The case of  one-mode initial data}
In this case, the initial condition in \eqref{eq_scalar} takes the form $u(0, x)=f_{in}(x,\beta(x)/\varepsilon)= \alpha(x)e^{i\beta(x)/\varepsilon}$. If we follow the analysis above and try to transform the equation on $u$ into  equations on the profile $V$ and the oscillation phase $S$,  then the only possibility to ensure some minimal smoothness on the augmented problem (\ref{eq_V}) and (\ref{eq_S}) is the following choice of the initial data
\begin{equation}
V(0, x, \tau )=\alpha(x),  \qquad S(0, x )=\beta(x).
\end{equation}
In particular, the initial data for $V$ belong to the kernel of $\partial_\tau$, and, according to the previous analysis, this only ensures that the first time and space derivatives
are bounded. However, it is not possible to construct ''well-prepared initial data'' $V(0, x,\tau)$ so that high order time-space derivatives  are also uniformly bounded.

\paragraph{The case of multi-modes initial data}
In this case,  we  expand the initial data as $f_{in}(x, \beta(x)/\varepsilon)=\sum_k f_k(x)e^{ik\beta(x)/\varepsilon}$
and one may decompose  the solution as $u(t, x)=\sum_k u_k(t, x)$, where each component $u_k$
satisfies
\begin{eqnarray*}
\partial_t u_k + c(x)\partial_x u_k  +r(u)\delta_{1k}= \frac{ia(x)}{\varepsilon}u_k, && u_k(0, x)=f_k(x)e^{ik\beta(x)/\varepsilon},
\end{eqnarray*}
and $\delta_{ij}$ is the  usual Kronecker symbol.

One can then apply the previous strategy, by considering the augmented functions $U_k(t, x, \tau)$
satisfying $U_k(t, x, k S(t, x)/\varepsilon )= u_k(t, x)$. This gives
$$
\partial_t U_k + c(x)\partial_x U_k + \frac{1}{\varepsilon}[\partial_tS+c(x)\partial_x S]\partial_\tau U_k+r(U)\delta_{1k}
= \frac{ia(x)}{\varepsilon} U_k, \;\; U_k(0, x, \tau)=f_k(x)e^{i\tau},
$$
from which we deduce the equation for the phase $S$
$$
\partial_t S+ c(x)\partial_x S=a(x), \;\; S(0, x)=\beta(x),
$$
and for $U_k$
$$
\partial_t U_k + c(x)\partial_x U_k + r(U)\delta_{1k} = -\frac{a(x)}{\varepsilon}[\partial_\tau U_k-iU_k], \;\; U_k(0, x, \tau)=f_k(x)e^{i\tau}.
$$

\subsection{The full numerical algorithms}
\label{algo}
\bigskip

In this section, details of the algorithm for solving $V$ and $S$  are given.
Periodic boundary conditions are considered in the $x$ and $\tau$ directions.
The uniform grids in time and space are defined as previously.   In addition, we also use a uniform mesh
for the $\tau$ direction: $\tau_\ell=\ell\Delta \tau$, for $\ell=0, \dots, N_\tau, \Delta \tau=2\pi/N_\tau$.
In the following description, the variable $\tau$ is kept continuous for simplicity. In our numerical
experiments, the pseudo spectral method is used for this variable. We denote by $V^n_j(\tau) \approx V(t^n, x_j, \tau)$ and $S^n_j \approx S(t^n, x_j)$
 the discrete unknowns.

We start with $V^0_j$ given by  \eqref{init_prepared} and $S^0_j=0$.
Then, for all $n\geq 0$, the scheme reads (assuming $c>0$)
\begin{eqnarray*}
&&\frac{V_j^{n+1} - V_j^n}{\Delta t} + c\frac{V^n_j - V^n_{j-1}}{\Delta x} +e^{-i\tau}r(e^{i\tau} V_j^n)
= -\frac{a(x_j)}{\varepsilon}\partial_\tau V_j^{n+1}, \nonumber\\
&&\frac{S_j^{n+1} - S_j^n}{\Delta t} + c\frac{S^n_j - S^n_{j-1}}{\Delta x} = a(x_j).
\end{eqnarray*}
At the final time $t_{f}=N\Delta t$ of the simulation, we come back to the original solution $u$ through
\beq
\label{u_numeric}
u(t_f, x_j) = V^N_j\Big(\tau=\frac{S^N_j}{\eps}\Big).
\eeq
Since $S^N_j/\eps$ does not coincide with a grid point $\tau_\ell$, a trigonometric interpolation
is performed. Note that higher order numerical schemes can be used and are necessary since
one needs to obtain $S^N_j/\varepsilon$ which may lead to large error in \eqref{u_numeric}
if $S$ is not computed accurately. {\it In practice we will use the pseudo-spectral method in $x$
to solve the equation for $S$}.

\subsection{Numerical tests}\label{num_test}
\bigskip

We present some tests solving \eqref{eq_scalar} with $r(u)=u^2/(u^2+2|u|^2)$, $c(x)=\cos^2(x)$,
$a(x)=3/2+\cos(2x)>0$ and the following non-oscillatory initial data
$$
u(0, x)=1+\frac{1}{2} \cos(2x) + i \Big[1+\frac{1}{2} \sin(2x)\Big],  \;\; x\in I= [-\pi/2, \pi/2].
$$
We compare the solution obtained by a direct method with resolved numerical parameters (smaller than $\varepsilon$)
and the solution obtained by the new approach presented previously.
The numerical parameters are as follows:
$\Delta t=\Delta x/2=|I|/N$ ($|I|$ being the length of the interval $I$) with $N=N_{ts}$ for the new approach
and $N=N_d$ for the direct approach. We choose $N_\tau=64$.

In Figures \ref{fig1apos}, \ref{fig2apos}, \ref{fig2apos_S} and \ref{fig3apos},
we plot the $\ell^\infty$ error in space (as a function of a range of $N_{ts}$)
between a reference solution obtained by a direct method with resolved numerical parameters
and the solution obtained by the new method. The error is computed
for different values of $\eps$, at the final time $t_f=0.1$, in different configurations.

In Figure \ref{fig1apos}, the initial data is well-prepared (given by \eqref{init_prepared})
and an exact solution for $S$ is considered. We observe on the left part of Figure \ref{fig1apos}
that for different values of $\eps$ ($\eps=1, \dots, 10^{-3}$), the new method is uniformly
first order accurate both in space and time. On the right part of Figure \ref{fig1apos},
the error is plotted as a function of $\eps$, for different values of $N_{ts}$
($N_{ts}= 20, 40, 100, 200, 1000$); each curve, corresponding to a given $N_{ts}$,
is almost constant, indicating that the error is independent from $\eps$.

In Figure \ref{fig2apos}, the initial data is well-prepared (given by \eqref{init_prepared})
but we now consider a numerical calculation of the phase $S$.
We used a first order upwind scheme together with a first order time integrator to compute $S$.
Hence, a numerical error ${\cal O}(\Delta x +\Delta t)$ is generated on $S$, which is divided by $\eps$
to construct the approximation of $u(t_f, x_j)$. This explains the behavior
of the curve associated to $\eps=10^{-3}$ for instance, in the left part of Figure \ref{fig2apos}.
This is also emphasized on the right part of Figure \ref{fig2apos}:
even if the curves do not cross each other, the error increases as $\eps$ decreases.
To improve this, the numerical scheme for $S$ is changed to a pseudo-spectral method
in space with a fourth-order Runge-Kutta time integrator. The corresponding results are displayed
in Figure \ref{fig2apos_S}. We observe that, since the error on $S$ is now very small, the uniform accuracy
is recovered.

In Figure \ref{fig3apos}, an exact calculation of the phase $S$ is considered
but now the initial data  $V^0_j=u(0, x_j)$ is not a corrected one .  As expected (see \cite{clm, cclm}), the uniform
accuracy is lost since the error depends on $\eps$, but we can observe that
the numerical error is still small even for small $\varepsilon$.

In Figure \ref{fig1}, the same diagnostics as before are displayed,
but we explore the possibility for $a$ to {\it vanish} at isolated points by considering $a(x)=1+\cos(2x)$.
With this choice of $a$, Theorem \ref{theorem1} does not apply directly.
We consider the case with a corrected initial data and an exact calculation for $S$.
The same results as before are obtained in the case of vanishing $a$: the new method
is first order uniformly accurate in $\eps$.

Then, we consider the asymptotic model given by  \eqref{asymptotic_model} for which
a standard numerical approximation (first order upwind scheme in space and first order explicit
time integrator)  is used to get $\bar{u}(t_f, x)$.
Then, the quantity $e^{-iS(t_f, x)/\eps} \bar{u}(t_f, x)$ is computed where the phase $S$ is solved  exactly.
In Figure \ref{fig4}, the error ($\ell^\infty$ in space) between $e^{-iS(t_f, x)/\eps} \bar{u}(t_f, x)$ and the
solution of the new method is displayed as a function of $\eps$ (logarithmic scale).
The error between the two models is ${\cal O}(\eps)$. This numerically justifies Remark
\ref{remark_asymptotic_model}.

Finally, in Figures \ref{fig5} and \ref{fig6}, we illustrate the space-time oscillations
arising in the solution, with $a(x)=3/2+\cos(2x)>0$. In Figure \ref{fig5}, the space dependence of the real part of the solution $u(t_f=1, x)$
is displayed for $\eps=5\cdot 10^{-3}$. A reference solution (obtained by a direct method
with resolved numerical parameters $N_d=4000$ and $\Delta t=5\cdot 10^{-4}$)
and the solution obtained by the new approach (with $N_{ts}=100$ and
$\Delta t=\pi/200\approx 0.0157$, $N_\tau=16$, well-prepared initial data and an exact $S$)
are plotted in the left part of Figure \ref{fig5} (the right being a zoom of the left part).
We can observe that the new method is able to capture very well high oscillations in space.

In Figure \ref{fig6}, we focus on time oscillations by considering the following time dependent quantity (root mean
square type)
$$
{\cal R}(t) = \Big | \int_I u(t, x) x \, \dd x\Big |.
$$
The numerical quadrature for the reference solution is performed on the mesh used for the new method.
Using the same parameters as before, one can observe that the solution of the new method  fits very well with the reference
solution even when the oscillations are not resolved by the time step $\Delta t\approx 0.0157$. The right part
of Figure \ref{fig6} is a zoom of the left part.

\begin{figure}
\begin{center}
\begin{tabular}{ccll}
\includegraphics[width=0.6\linewidth]{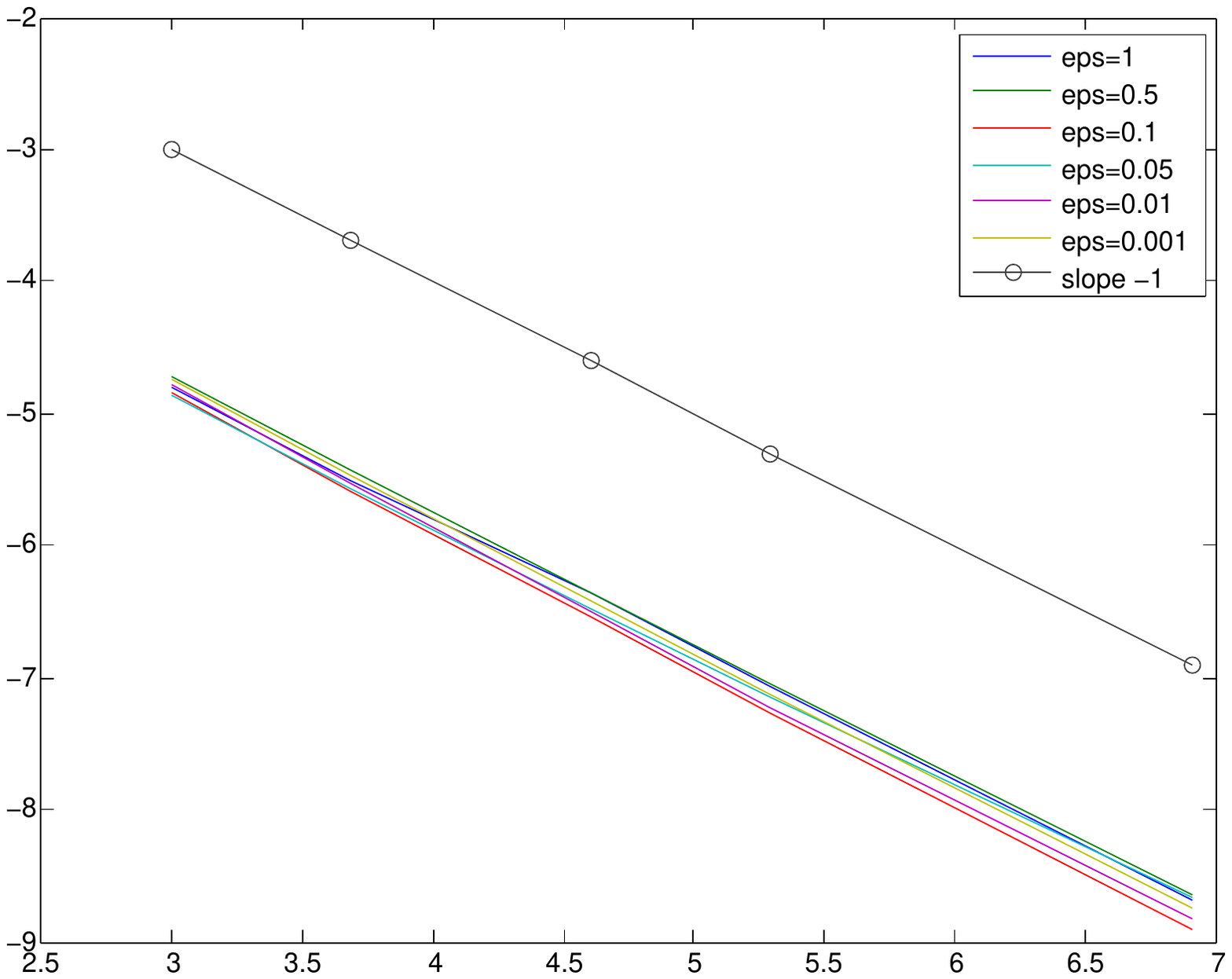}&
\hspace{-2cm}\vspace{-3cm}
\includegraphics[width=0.6\linewidth]{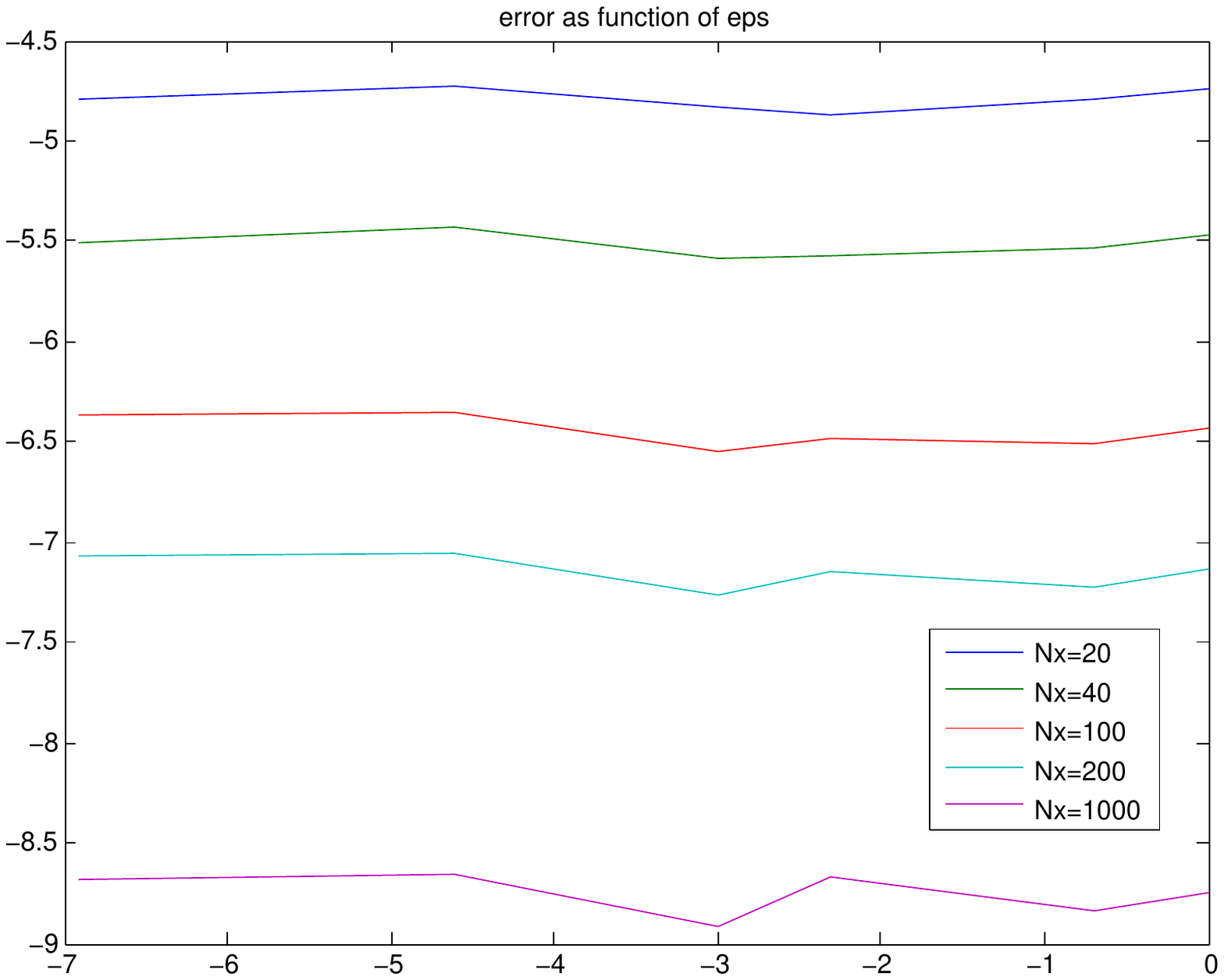}\\
\end{tabular}
\end{center}
\caption{Plot of the $\ell^\infty$ error for the new method {\it with} corrected  initial condition
and exact computation for $S$. Left: error ($\log$-$\log$ scale) as a function of $N_{ts}$
($N_{ts}= 20, 40, 100, 200, 1000$)
for different values of $\eps$ ($\eps=1, \dots, 10^{-3}$). Right: error ($\log$-$\log$ scale) as a function of $\eps$
for different $N_{ts}$.}
\label{fig1apos}
\end{figure}

\begin{figure}
\vspace{-4cm}
\begin{center}
\begin{tabular}{ccll}
\includegraphics[width=0.6\linewidth]{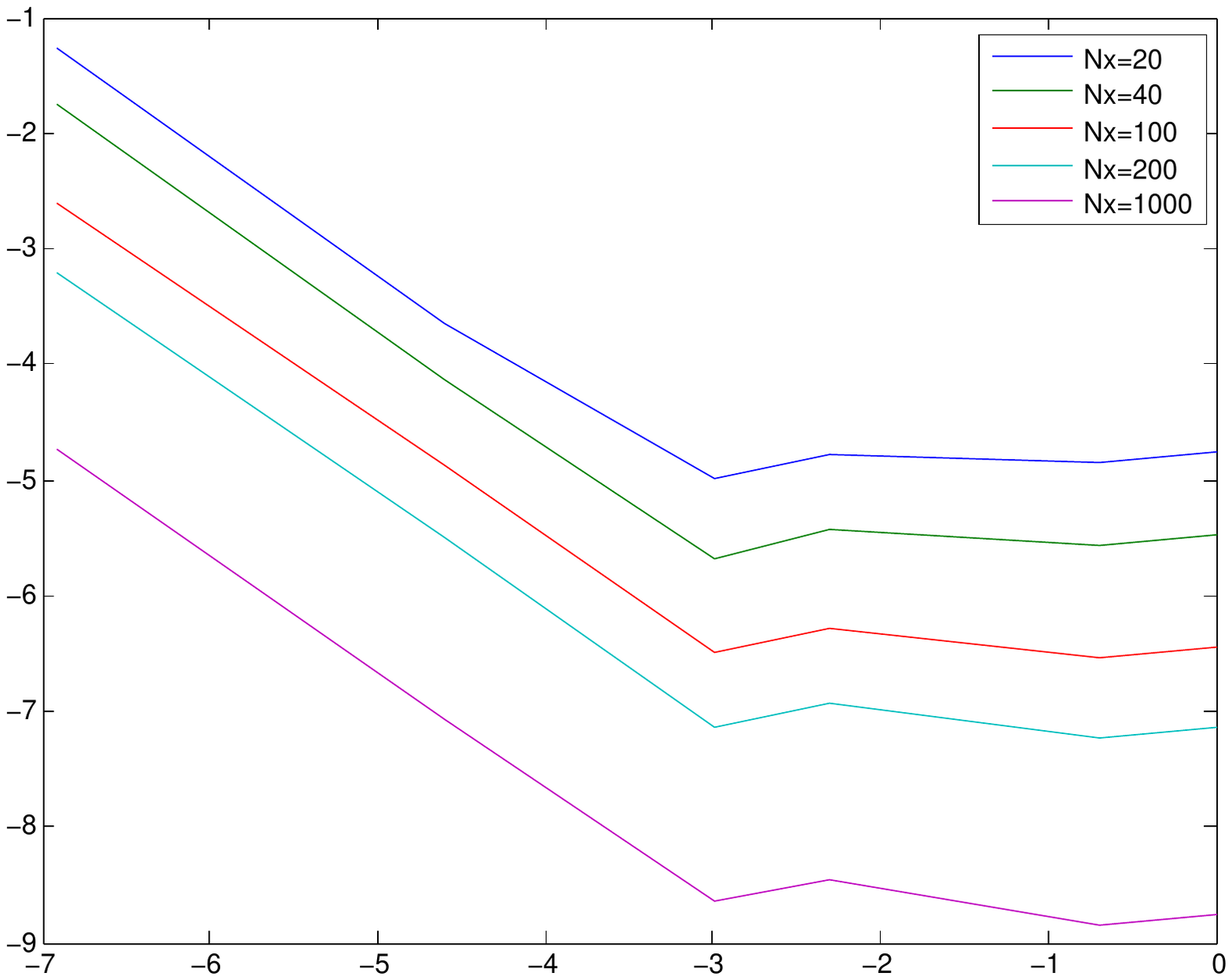}&
\hspace{-2cm}\vspace{-3cm}
\includegraphics[width=0.6\linewidth]{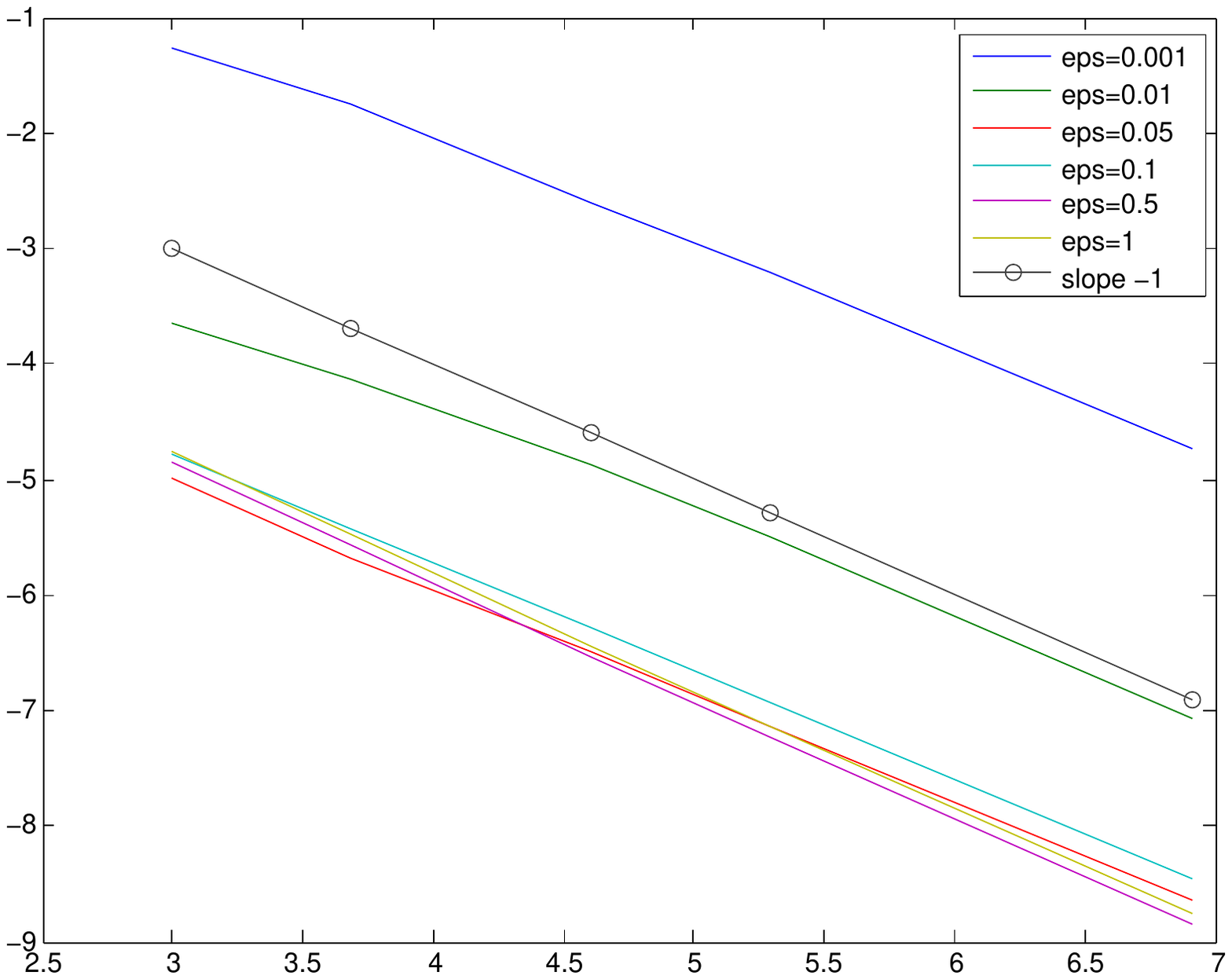}\\
\end{tabular}
\end{center}
\caption{Plot of the $\ell^\infty$ error for the new method {\it with} corrected initial condition
and numerical approximation for $S$. Left: error ($\log$-$\log$ scale) as a function of $N_{ts}$
($N_{ts}= 20, 40, 100, 200, 1000$)
for different values of $\eps$ ($\eps=1, \dots, 10^{-3}$). Right: error ($\log$-$\log$ scale) as a function of $\eps$
for different $N_{ts}$.}
\label{fig2apos}
\end{figure}

\begin{figure}
\begin{center}
\begin{tabular}{ccll}
\includegraphics[width=0.6\linewidth]{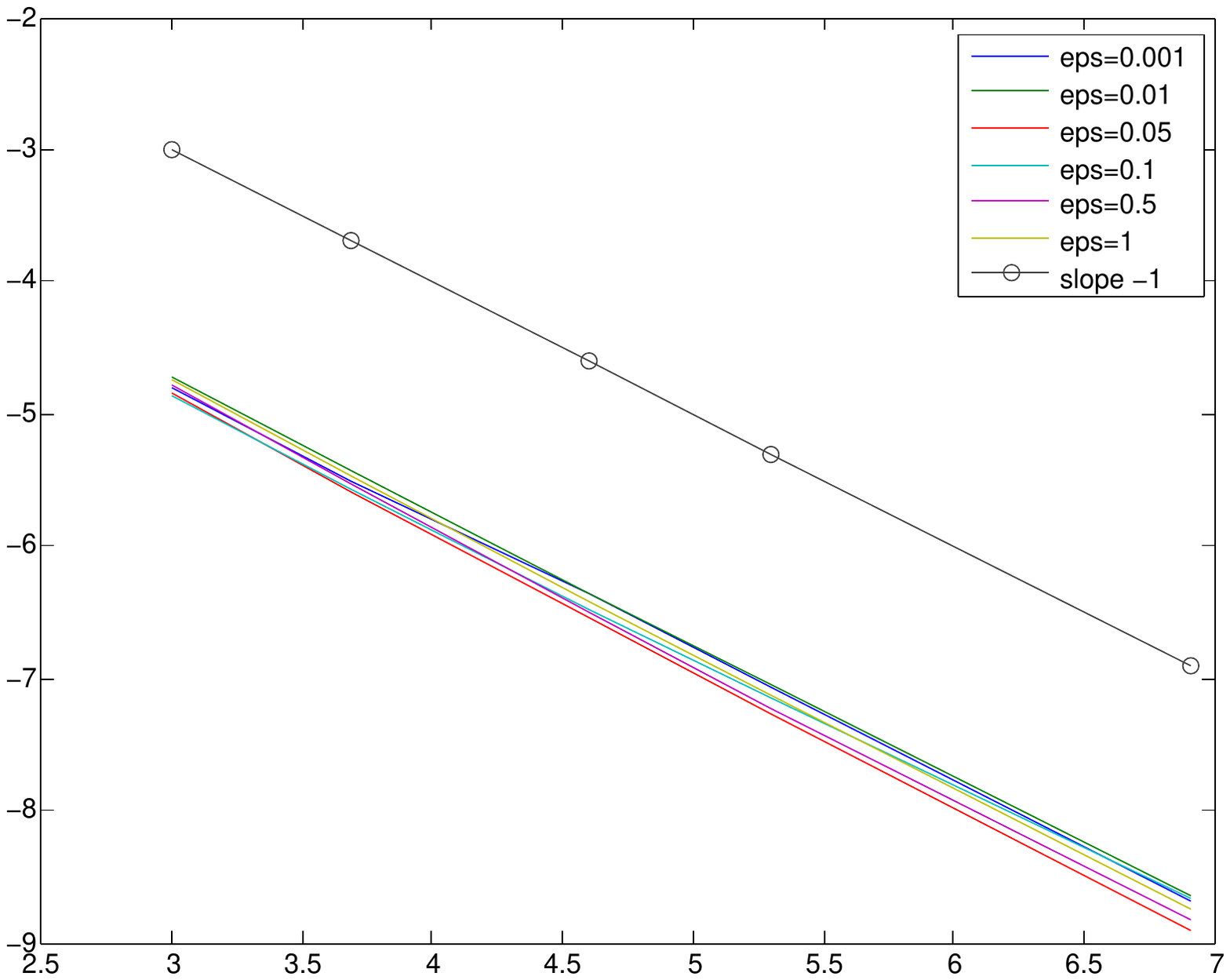}&
\hspace{-2cm}\vspace{-3cm}
\includegraphics[width=0.6\linewidth]{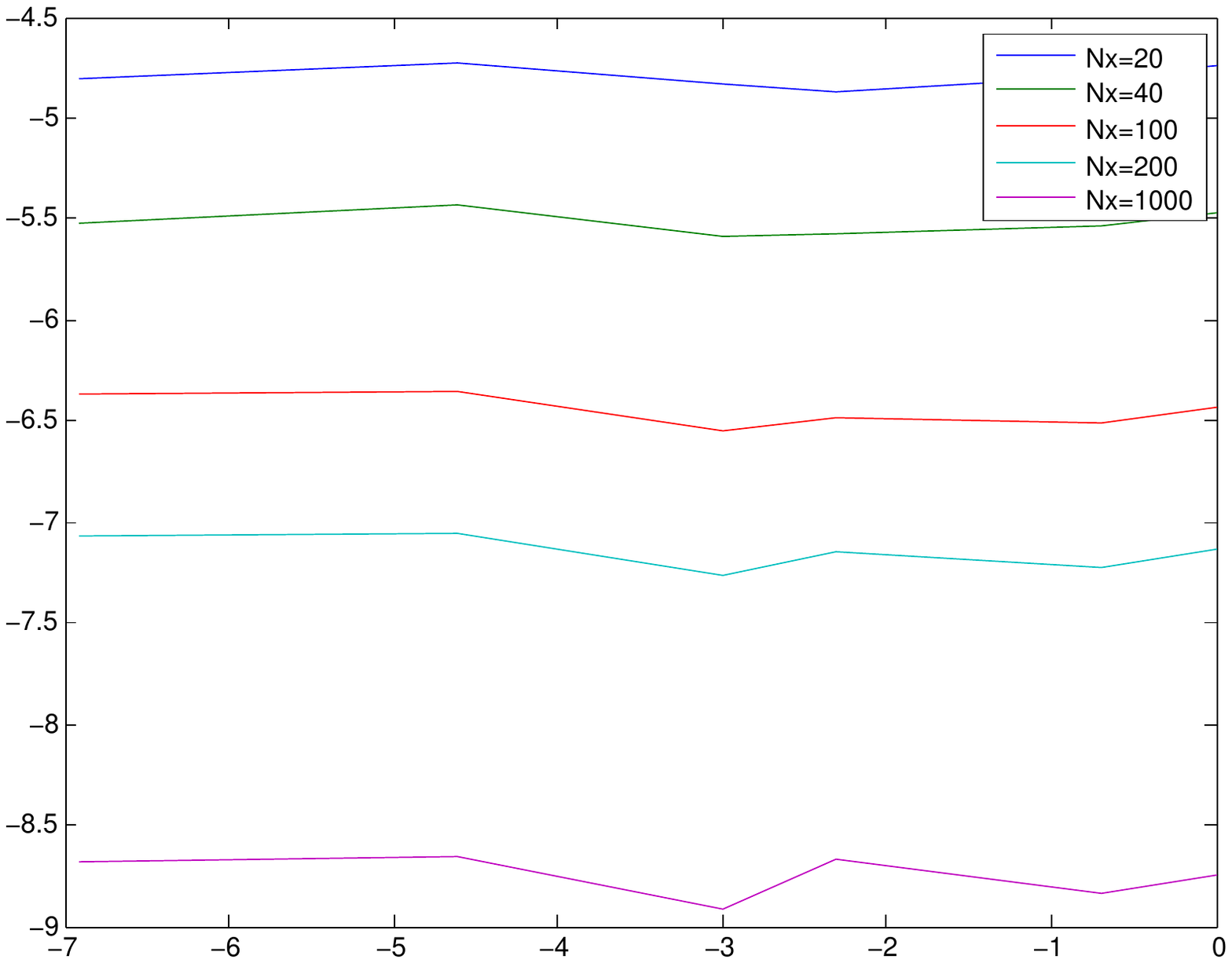}\\
\end{tabular}
\end{center}
\caption{Plot of the $\ell^\infty$ error for the new method {\it with} corrected  initial condition
and an improved numerical approximation for $S$ (pseudo-spectral in space and $4$th order Runge-Kutta).
Left: error ($\log$-$\log$ scale) as a function of $N_{ts}$
($N_{ts}= 20, 40, 100, 200, 1000$)
for different values of $\eps$ ($\eps=1, \dots, 10^{-3}$). Right: error ($\log$-$\log$ scale) as a function of $\eps$
for different $N_{ts}$.}
\label{fig2apos_S}
\end{figure}

\begin{figure}
\vspace{-4cm}
\begin{center}
\begin{tabular}{ccll}
\includegraphics[width=0.6\linewidth]{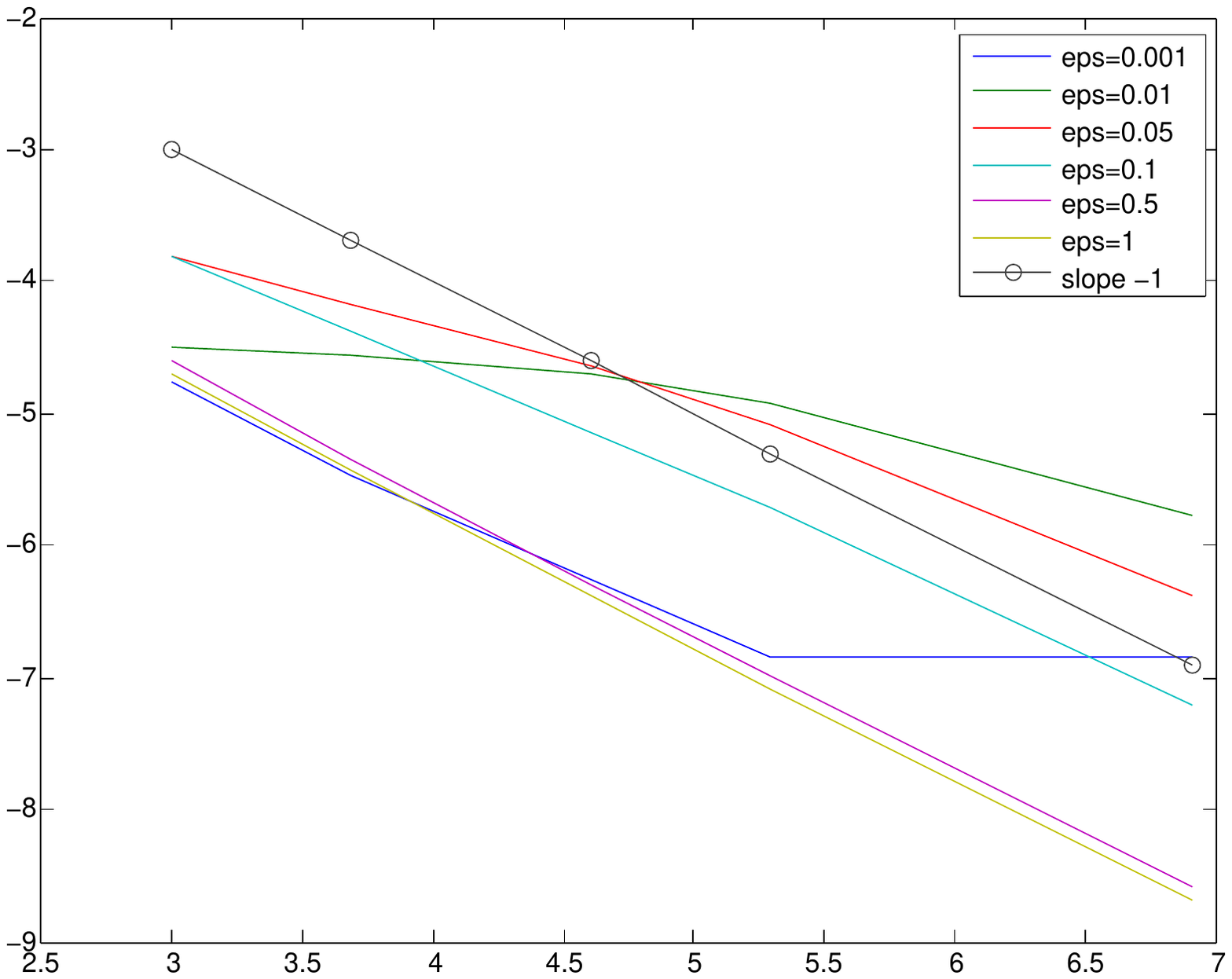}&
\hspace{-2cm}\vspace{-3cm}
\includegraphics[width=0.6\linewidth]{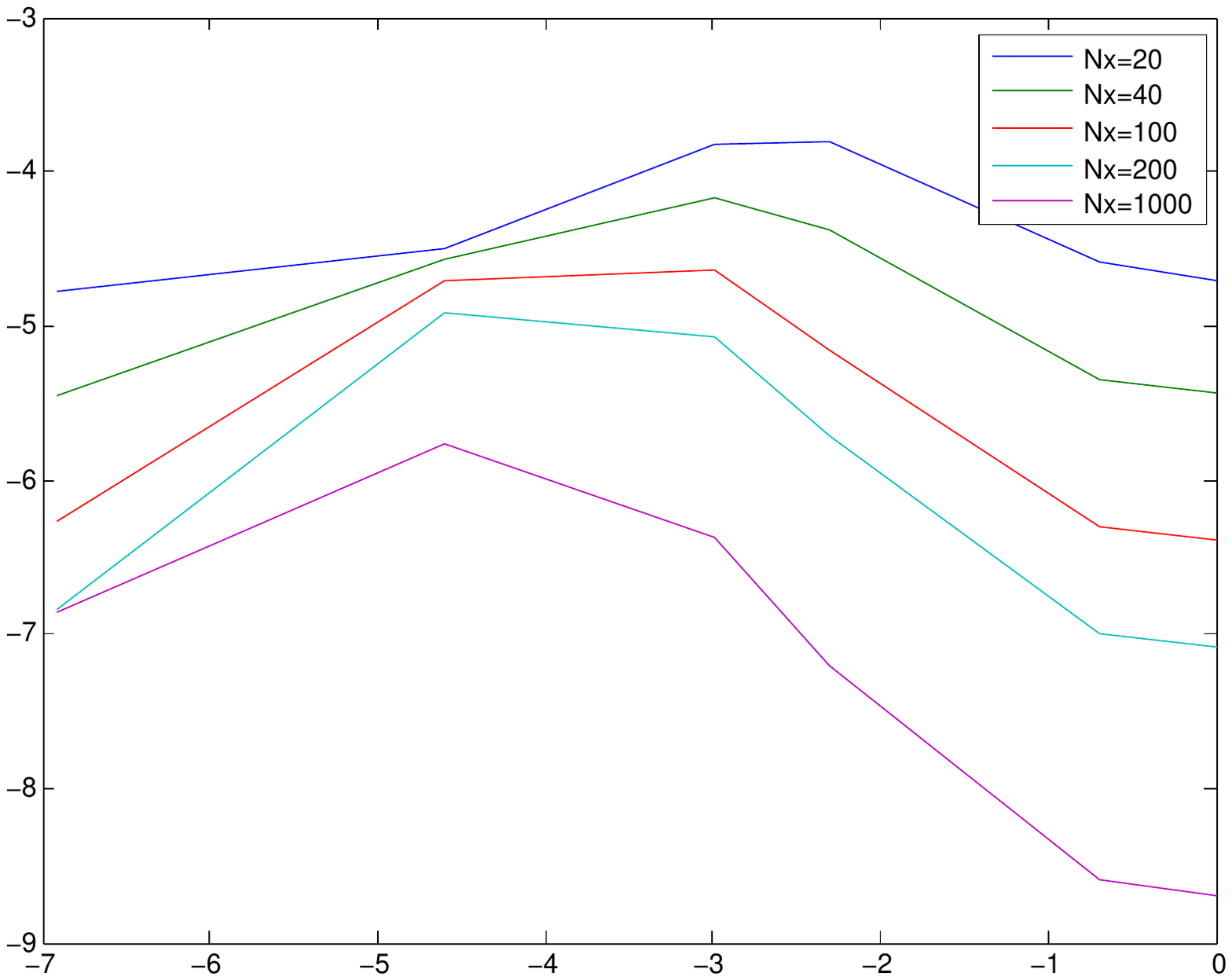}\\
\end{tabular}
\end{center}
\caption{Plot of the $\ell^\infty$ error for the new method {\it without} corrected initial condition
and exact computation for $S$. Left: error ($\log$-$\log$ scale) as a function of $N_{ts}$
($N_{ts}= 20, 40, 100, 200, 1000$)
for different values of $\eps$ ($\eps=1, \dots, 10^{-3}$). Right: error ($\log$-$\log$ scale) as a function of $\eps$
for different $N_{ts}$.}
\label{fig3apos}
\end{figure}

\begin{figure}
\begin{center}
\begin{tabular}{ccll}
\includegraphics[width=0.6\linewidth]{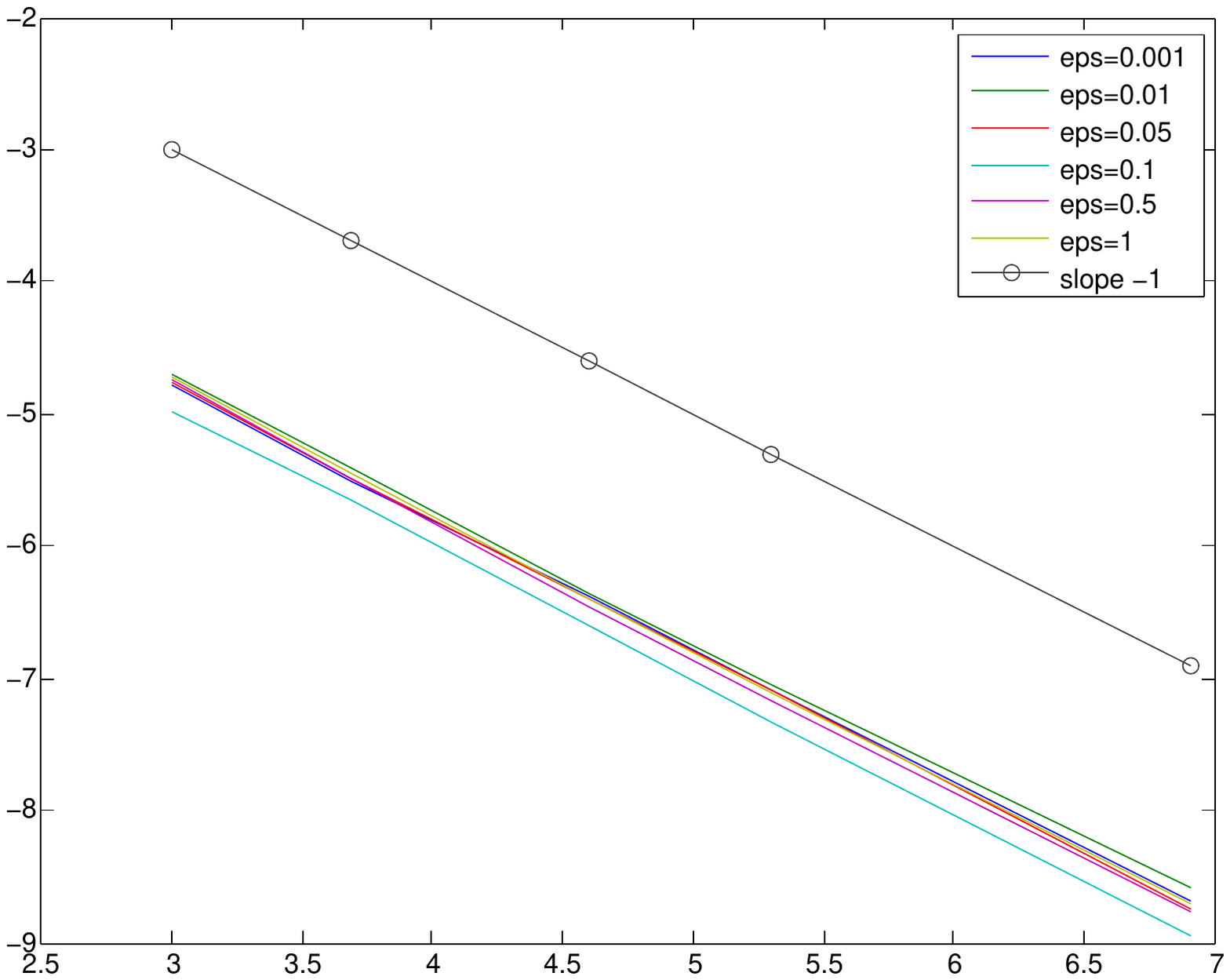}&
\hspace{-2cm}\vspace{-3cm}
\includegraphics[width=0.6\linewidth]{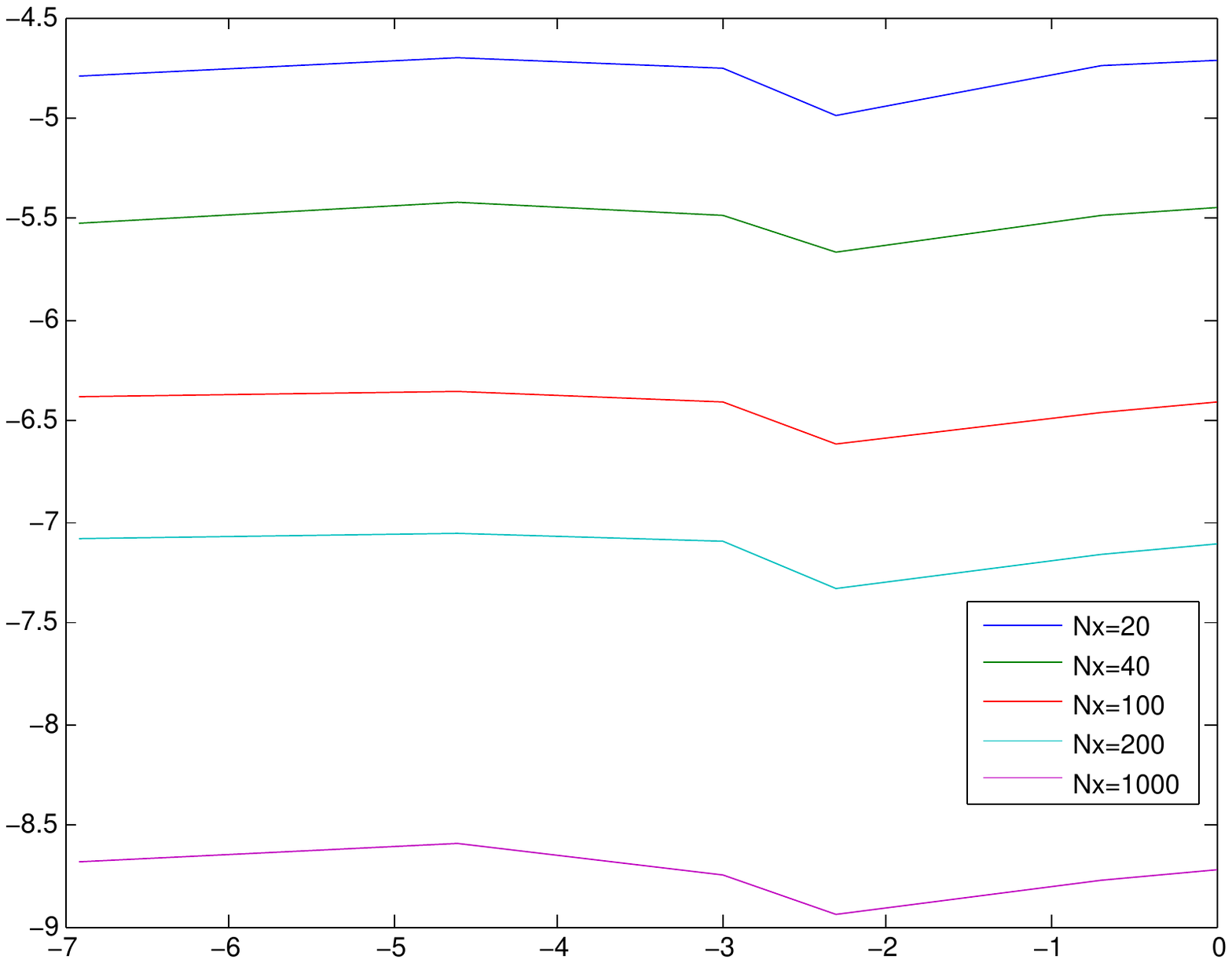}\\
\end{tabular}
\end{center}
\caption{Plot of the $\ell^\infty$ error for the new method {\it with} corrected initial condition
and exact computation for $S$, in the case where $a$ vanishes ($a(x)=1+\cos(2x)$) at isolated points.
Left: error ($\log$-$\log$ scale) as a function of $N_{ts}$
($N_{ts}= 20, 40, 100, 200, 1000$)
for different values of $\eps$ ($\eps=1, \dots, 10^{-3}$). Right: error ($\log$-$\log$ scale) as a function of $\eps$
for different $N_{ts}$. }
\label{fig1}
\end{figure}

\begin{figure}
\vspace{-4cm}
\centering
\includegraphics[scale=0.4]{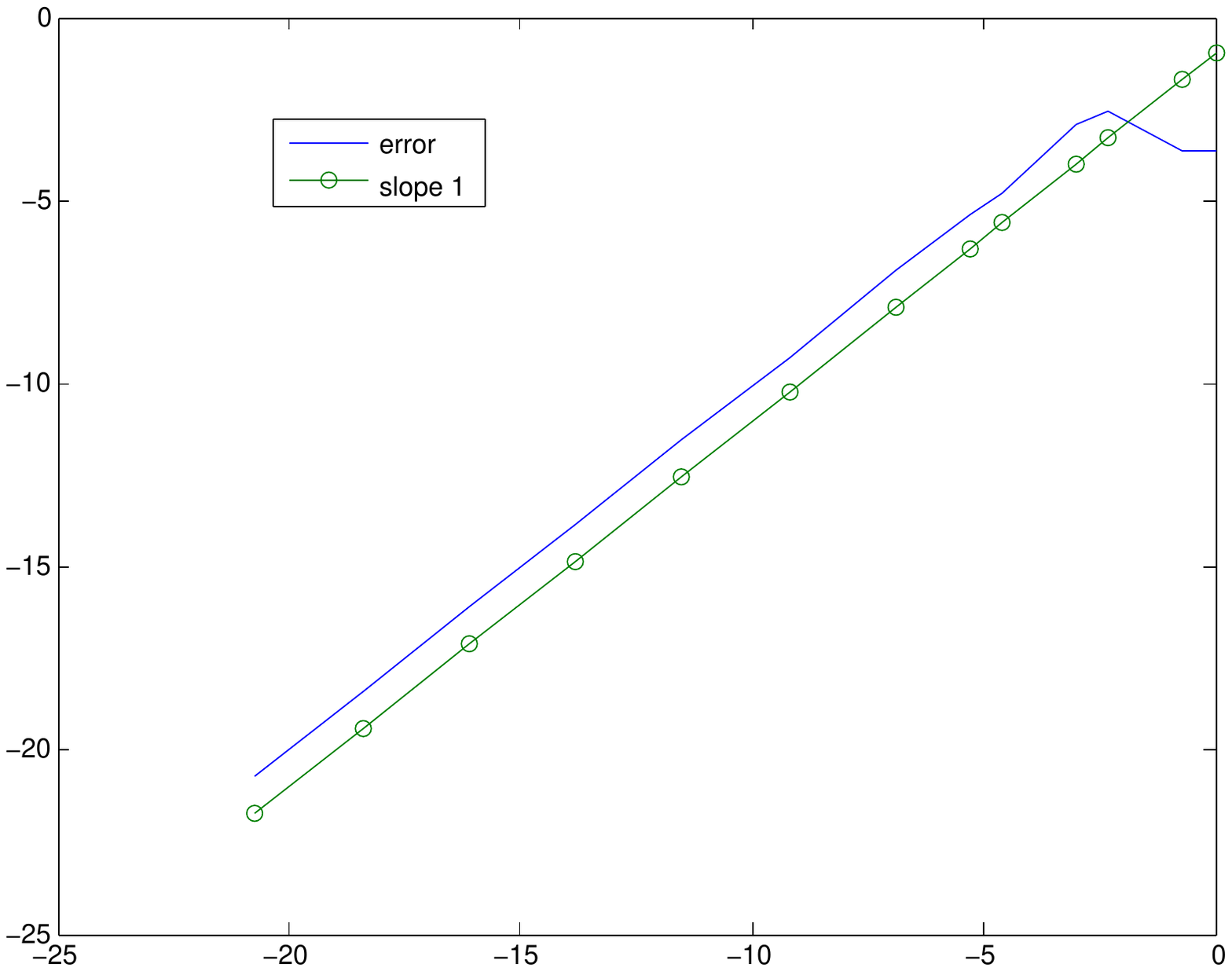}
\vspace{-3cm}
\caption{Plot of the $\ell^\infty$ error ($\log$-$\log$ scale) between the solution of the asymptotic model (\ref{asymptotic_model}) and the one obtained by the new method, as
function of $\eps$. }
\label{fig4}
\end{figure}

\begin{figure}
\begin{center}
\begin{tabular}{ccll}
\includegraphics[width=0.6\linewidth]{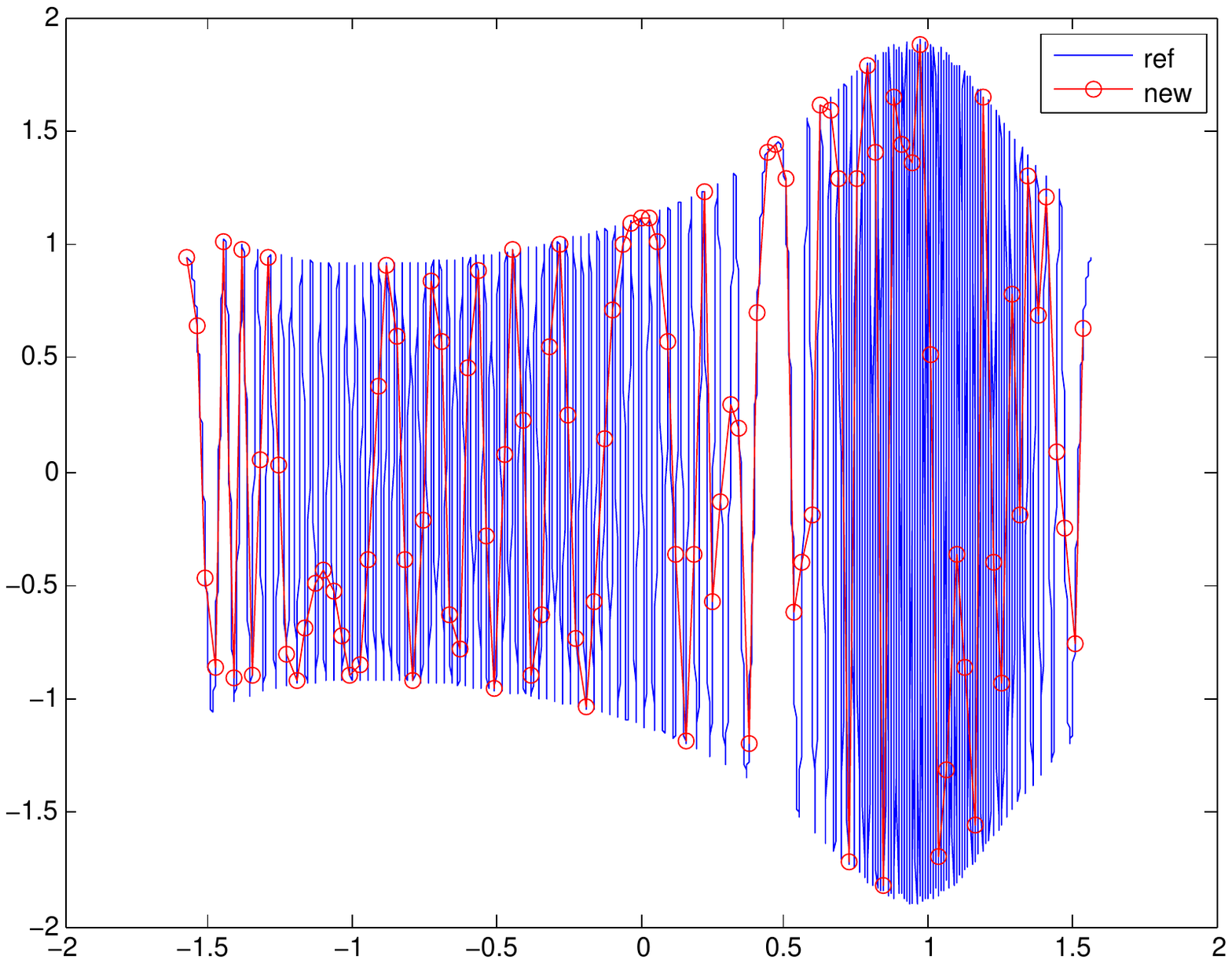}&
\hspace{-2cm}\vspace{-3cm}
\includegraphics[width=0.6\linewidth]{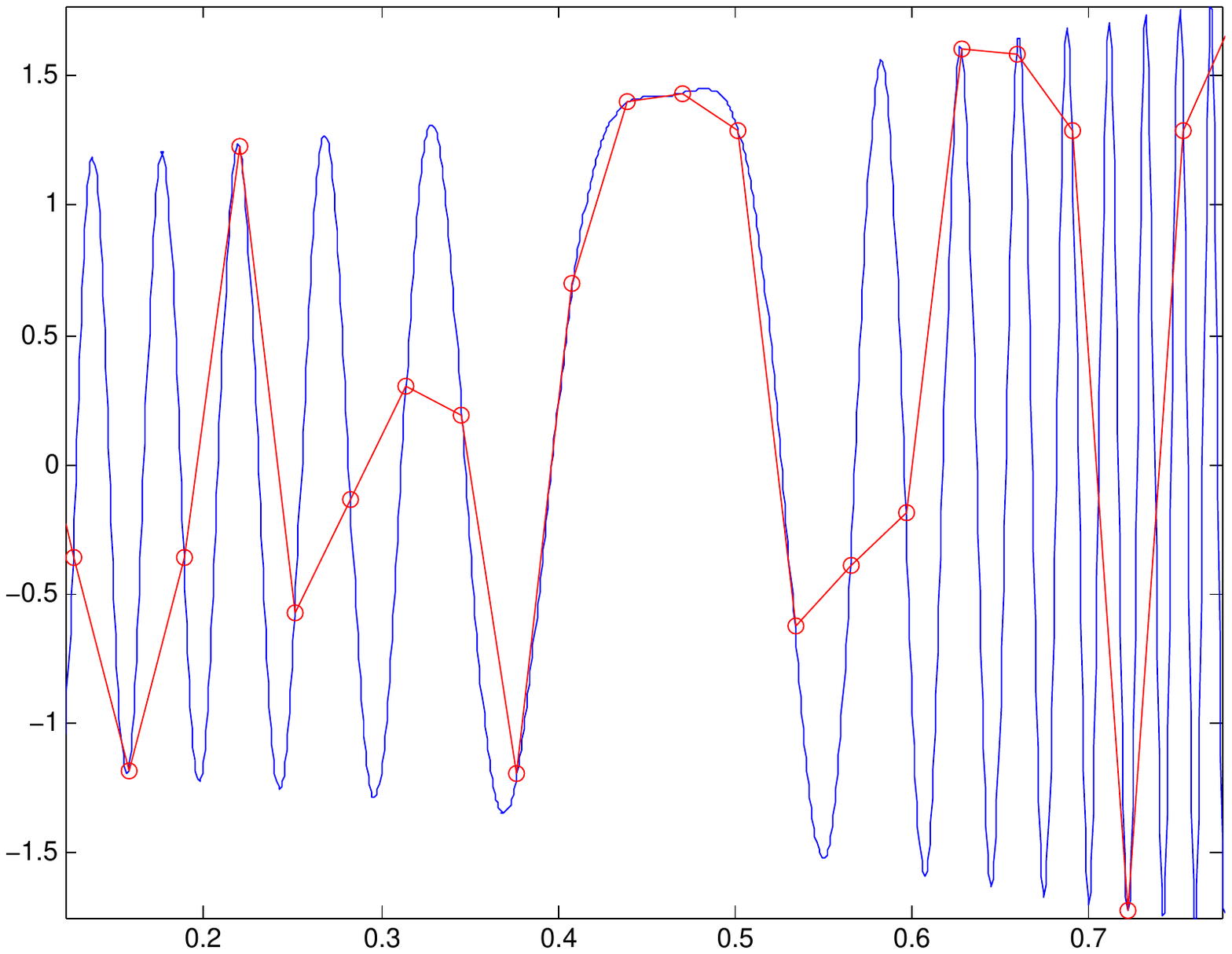}\\
\end{tabular}
\end{center}
\caption{Comparison between a reference solution and the  solution of the new method
({\it with} initial correction and exact computation for $S$),
for $\eps=5\cdot 10^{-3}$, $t_f=1$.
Left: space dependence of the real part of the unknown.
The right part is a zoom of the left part.}
\label{fig5}
\end{figure}

\begin{figure}
\vspace{-4cm}
\begin{center}
\begin{tabular}{ccll}
\includegraphics[width=0.6\linewidth]{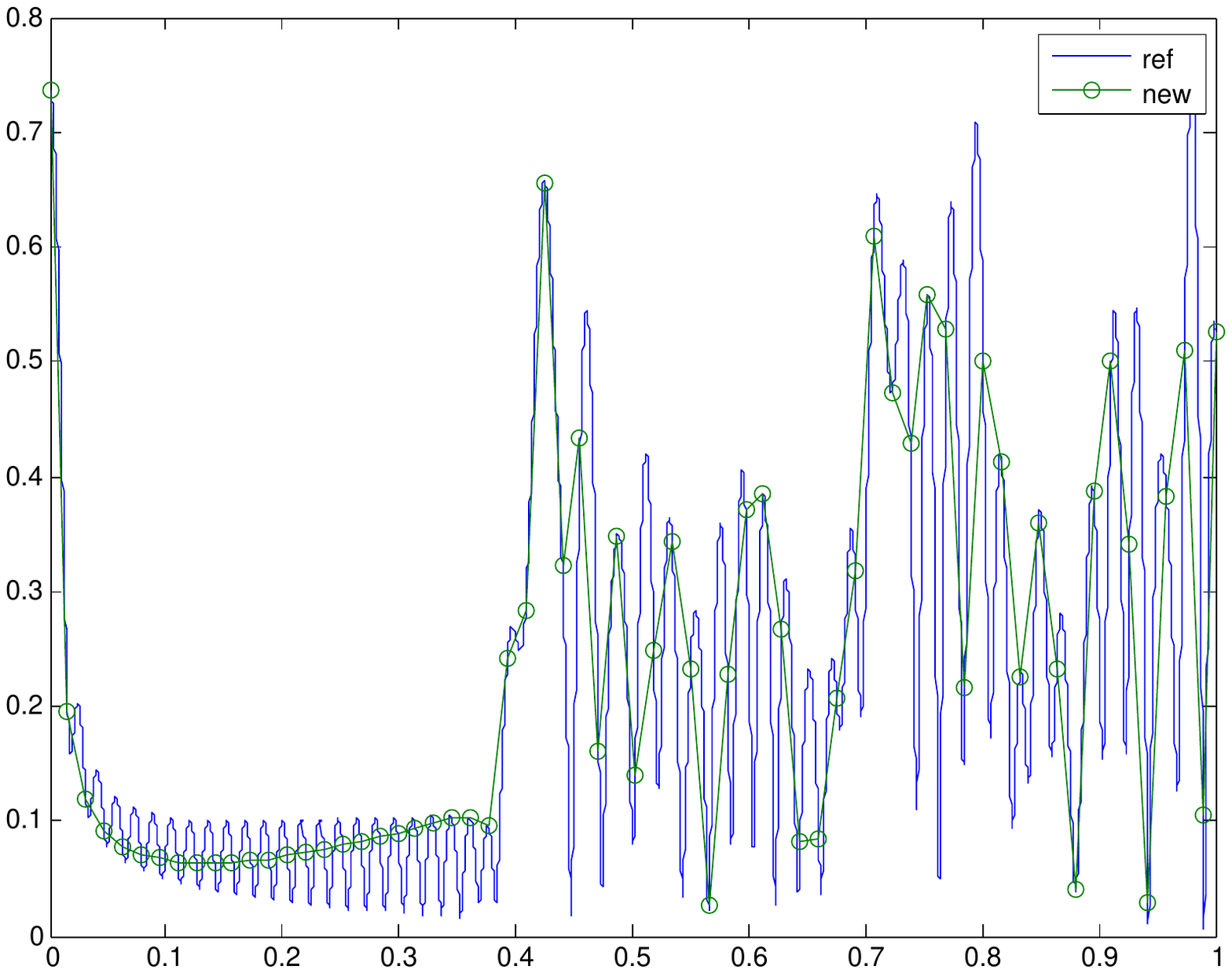}&
\hspace{-2cm}\vspace{-3cm}
\includegraphics[width=0.6\linewidth]{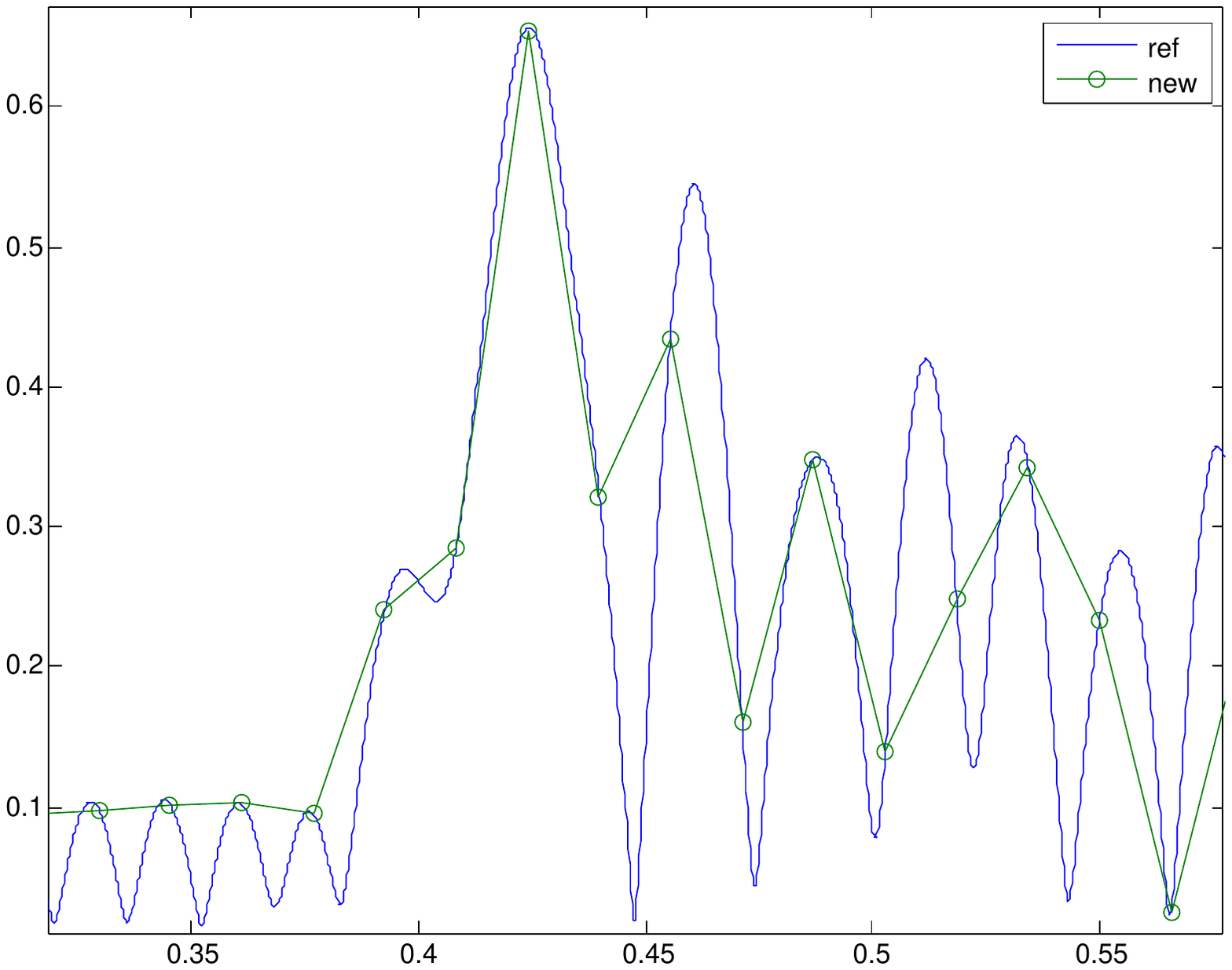}\\
\end{tabular}
\end{center}
\caption{Time history of ${\cal R}$.
Comparison between a reference solution and the  solution of the new method
({\it with} initial correction and exact computation for $S$),
for $\eps=5\cdot 10^{-3}$, $t_f=1$.
The right part is a zoom of the left part.}
\label{fig6}
\end{figure}

\newpage


\section{Extension to a class of PDE systems}
In this section, we focus on systems of equations and consider the  case where  $u(t, x)\in \mathbb{C}^2$
satisfies a hyperbolic system of the following form (with $t\geq 0, x\in \mathbb{R}$)
\begin{equation}
\label{eq_syst}
\partial_t u + A(x)\partial_x u + R(u) = \frac{i E(t, x)}{\varepsilon} Du + Cu, \;\; \  u(t=0, x)=u_0(x),
\end{equation}
where $R(u)=(R_1(u), R_2(u))\in \mathbb{C}^2$ is a reaction term,  $E$ is a real scalar function,
$C$ is a $2\times 2$ constant matrix and
$$
 A (x)=
\left(
\begin{array}{ccll}
a_1(x) & 0\\
0 & a_2(x)
\end{array}
\right),
\quad
D=
\left(
\begin{array}{ccll}
0 & 0\\
0 & -1
\end{array}
\right),
\quad
C=
\left(
\begin{array}{ccll}
C_{11} & C_{12}\\
C_{21} & C_{22}
\end{array}
\right).
$$
This model is a simplified version of a more physical model to be studied  in section 3.5.

\subsection{GO versus NGO}

 We first show that the GO approach  does not work for systems  like \eqref{eq_syst},
 even in the linear case $R=0$ and a non-oscillatory initial data $u(0, x)=u_0(x)$.
This is due to the non-commutativeness of the matrices $C$ and $D$ in general. Indeed, let
$u_k(t, x)=\alpha_k(t, x)e^{iS_k(t, x)/\varepsilon}, k=1, 2$. Inserting this ansatz in \eqref{eq_syst}, one gets
\begin{eqnarray*}
\partial_t \alpha_1 + a_1\partial_x \alpha_1 + \frac{i}{\varepsilon} [\partial_t S_1 + a_1\partial_x S_1]\alpha_1
\!\!&\!\!=\!\!&\!\! C_{11}\alpha_1 +C_{12}\alpha_2 e^{i (S_2-S_1)/\varepsilon}\,,\nonumber\\
\partial_t \alpha_2 + a_2\partial_x \alpha_2 + \frac{i}{\varepsilon} [\partial_t S_2 + a_2\partial_x S_2]\alpha_2
\!\!&\!\!=\!\!&\!\! -\frac{iE}{\varepsilon}\alpha_2 + C_{21}\alpha_1  e^{i (S_1-S_2)/\varepsilon}+C_{21}\alpha_2\,.
\end{eqnarray*}
Set
\begin{eqnarray*}
\partial_t S_1 + a_1\partial_x S_1 = 0, && S_1(0, x)=0,\nonumber\\
\partial_t S_2 + a_2\partial_x S_2 = -E, && S_2(0, x)=0,
\end{eqnarray*}
thus $S_1(t, x)\equiv 0$, while  $\alpha=(\alpha_1, \alpha_2)^T$ is governed by
$$
\partial_t \alpha + A(x)\partial_x \alpha = B\alpha,
$$
with
$$
B=
\left(
\begin{array}{ccll}
C_{11} & C_{12}e^{i (S_2-S_1)/\varepsilon} \\
 C_{21}e^{i (S_1-S_2)/\varepsilon} & C_{22}
\end{array}
\right).
$$
Clearly a solution $\alpha$  of this equation will be highly-oscillatory for $\eps$ small, whereas  the GO ansatz assumes that $\alpha$ and $S$ are smooth.   Therefore the GO approach does not serve our goal.
This motivates the NGO approach.

To do so, we consider the augmented function $U(t, x, \tau)$
such that
\[U(t, x, S(t, x)/\varepsilon)=u(t, x).
\]
Then, $U=(U_1, U_2)$ satisfies
\begin{eqnarray*}
\partial_t U_1 +a_1\partial_x U_1+\frac{1}{\varepsilon}[\partial_t S + a_1\partial_x S]\partial_\tau U_1 + R_1(U_1, U_2)=C_{11}U_1+C_{12}U_2, \nonumber\\
\partial_t U_2 +a_2\partial_x U_2+\frac{1}{\varepsilon}[\partial_t S + a_2\partial_x S]\partial_\tau U_2 + R_2(U_1, U_2)=-\frac{iE}{\varepsilon}U_2 + C_{21}U_1+C_{22}U_2. \nonumber\\
\end{eqnarray*}
The equation for the phase $S$ writes
$$
\partial_t S+a_2\partial_x S = -E, \;\; S(0, x)=0,
$$
and the equations for $(U_1, U_2)$ become
\begin{eqnarray*}
\partial_t U_1 \!\!&\!\!+\!\!&\!\! a_1\partial_x U_1+\frac{1}{\varepsilon}[(a_1-a_2)\partial_x S-E]\partial_\tau U_1 + R_1(U_1, U_2)=C_{11}U_1+C_{12}U_2, \nonumber\\
\partial_t U_2 \!\!&\!\!+\!\!&\!\! a_2\partial_x U_2 + R_2(U_1, U_2)=-\frac{E}{\varepsilon}[\partial_\tau U_2 +iU_2] + C_{21}U_1+C_{22}U_2.
\end{eqnarray*}
Setting $V_2=e^{i\tau}U_2$, we finally obtain
\beq
\label{eqU1V2}
\left\{\begin{array}{l}
\ds \partial_t U_1+a_1\partial_x U_1 \!+ R_1(U_1, e^{-i\tau}V_2)-C_{11}U_1-C_{12}e^{-i\tau}V_2
=\frac{1}{\varepsilon}[E-(a_1-a_2)\partial_x S]\partial_\tau U_1, \\ \\
\ds \partial_t V_2 +a_2\partial_x V_2 + e^{i\tau} R_2(U_1, e^{-i\tau}V_2)-C_{21}e^{i\tau}U_1-C_{22}V_2=-\frac{E}{\varepsilon}\partial_\tau V_2.
\end{array}\right.
\eeq

\subsection{A suitable initial data for system \eqref{eqU1V2}}
\label{syst_initcond}
Equation (\ref{eqU1V2}) needs initial data $U_1(0,x,\tau)$ and $V_2(0,x,\tau)$.
This initial data $U(0,x,\tau)$ will be chosen such that
the two following conditions are satisfied:
\begin{itemize}
\item  $U(0,x,  0)= u_0(x)= (f_1^{in}(x), f_2^{in}(x))$.
\item  The solution to (\ref{eqU1V2}) is smooth
with respect to $\eps$: the successive derivatives in {\em time and space} (up to some order $p\geq 1$)
are bounded uniformly in $\eps$.
\end{itemize}
The approach is similar to the scalar case, and similar notations will be used in the following analysis.

We decompose the solutions $U_1$ and $V_2$ as $U_1=U_1^0+U_1^1$ and $V_2=V_2^0+V_2^1$, where
$ U_1^0=\Pi U_1$ and $V_2^0= \Pi V_2$.
Injecting the decomposition into (\ref{eqU1V2})  and applying $(I-\Pi)$ to (\ref{eqU1V2}), one gets
\beq
\label{ImPieqF1G2}
\left\{\begin{array}{l}\ds \partial_t U_1^1 +a_1 \partial_x U_1^1 +({\cal I}-\Pi)[R_1(U_1, e^{-i\tau}V_2)]-C_{11}U_1^1- C_{12}({\cal I}-\Pi)[e^{-i\tau} V_2] \\\\
\ds\hspace{5cm}= -\frac{1}{\varepsilon}\left[E+(a_1-a_2) \partial_x S\right]\partial_\tau U_1^1,
 \\  \\
\ds \partial_t V_2^1 +a_2 \partial_x V_2^1 + ({\cal I}-\Pi)[e^{i\tau}R_1(U_1, e^{-i\tau}V_2)]-C_{21}({\cal I}-\Pi)[e^{i\tau} U_1]-C_{22}V_2^1 \\ \\
\ds \hspace{5cm}= -\frac{E}{\varepsilon}\partial_\tau V_2^1.
\end{array}\right.
\eeq
Assuming for simplicity that $R_1=R_2=0$, this implies that
$$
U_1^1= \frac{\eps C_{12}}{E +(a_1-a_2) \partial_x S }{\cal L}^{-1} \left( e^{-i\tau}\right)V_2^0  + O(\eps^2),
$$
$$
V_2^1= - \frac{\eps C_{21}}{E}{\cal L}^{-1} \left( e^{i\tau}\right)U_1^0  + O(\eps^2),
$$
which explicitly gives (using ${\cal L}^{-1}(e^{\pm i\tau}) = \mp ie^{\pm i\tau}$)
$$
U_1^1= \frac{i\eps C_{12} e^{-i\tau}}{E +(a_1-a_2) \partial_x S }V_2^0  + O(\eps^2),   \;\;\;\;
V_2^1=  -\frac{i\eps C_{21} e^{i\tau}}{E}U_1^0  + O(\eps^2).
$$
This yields
$$
U_1(t,x,\tau)= U_1^0(t, x) + \frac{i\eps C_{12} e^{-i\tau}}{E(t, x) +(a_1(x)-a_2(x)) \partial_x S(t, x) }V_2^0(t, x)  + O(\eps^2),
$$
$$
U_2(t,x,\tau)= V_2^0(t, x) e^{-i\tau}- \frac{i\eps C_{21}}{E(t, x)}U_1^0(t, x)  + O(\eps^2).
$$
To find the suitable initial condition for $U_1$ and $V_2$, one uses
$U_k(t=0, x, 0)=f_k^{in}(x), k=1, 2$,
so that  one needs  to solve the following
system in $(U_1^0, V_2^0)(0, x)$
$$
\begin{array}{l}
\ds U_1^0(0, x)+ \frac{i\eps C_{12} }{E(0,x)  }V_2^0(0, x) = f_1^{in}(x), \;\;\;  V_2^0(0, x) -  \frac{i\eps C_{21} }{E(0,x)}U_1^0(0, x) = f_2^{in}(x).
\end{array}
$$
The solutions are
$$
U_1^0=
\frac{1}{E^2-\eps^2 C_{21}C_{12}}\left(
 E^2 f^{in}_1 - iC_{12} \eps E f^{in}_2\right)
 $$
 $$
V_2^0 = \frac{1}{E^2-\eps^2C_{21}C_{12}}\left( i\eps C_{21} E f^{in}_1 + E^2 f^{in}_2\right).
$$
Thus, the initial conditions with first order correction writes
\beq
\label{init_cond_syst}
\left\{\begin{array}{l}
\ds U_1(0,x,\tau)=  f^{in}_1 + \frac{i\eps E C_{12}}{E^2-\eps^2 C_{12}C_{21}} \left( e^{-i\tau}-1\right)f^{in}_2,\\\\
\ds U_2(0,x,\tau)= \frac{i\eps  E C_{21} }{E^2 - \eps^2 C_{12}C_{21}} \left( e^{-i\tau} -1\right)f^{in}_1
+e^{-i\tau} f^{in}_2.
\end{array}
\right.
\eeq

\subsection{A numerical scheme for the $2\times 2$ system \eqref{eqU1V2}}
 Denoting $U_1^n(x, \tau)\approx U_1(t^n, x, \tau)$
and $V_2^n(x, \tau)\approx V_2(t^n,x,\tau)$ the approximations of the solution to \eqref{eqU1V2}
which satisfy the following numerical (semi-discrete in time) scheme
\beq
\label{eqF1G2-num}
\left\{\begin{array}{l}\ds \frac{U_1^{n+1} - U_1^{n}}{\Delta t}+a_1 \partial_x U_1^n -C_{11}U_1^n
-C_{12} e^{-i\tau} V_2^n \\\\
\ds \hspace{5cm}= -\frac{1}{\varepsilon}\left[E^n+(a_1-a_2) \partial_x S^n\right]\partial_\tau U_1^{n+1},
 \\  \\
\ds  \frac{V_2^{n+1} - V_2^{n}}{\Delta t}+a_2\partial_x V_2^n -C_{21}e^{i\tau}U_1^n-C_{22}V_2^n = -\frac{E^n}{\varepsilon}\partial_\tau V_2^{n+1},
\end{array}\right.
\eeq
whereas for the phase $S$, we use
$$
\frac{S^{n+1}-S^n}{\Delta t} + a_2\partial_x S^n = -E^n.
$$
At initial time $n=0$, we use the corrected initial condition \eqref{init_cond_syst}.
For the space approximation, we use the psuedo-spectral scheme in the periodic variable $\tau$
and a first order upwind scheme for the transport terms in $x$ (high order methods will be used for the approximation
of $S$ (as discussed in the scalar case), as well as semi-Lagrangian method).
Then, from $(U_1^n(x,\tau), V_2^n(x, \tau), S^n(x))$, we can construct
an approximation of $u(t^n,x)=(u_1(t^n,x), u_2(t^n, x))$ solution to \eqref{eq_syst} through the relation
$$
u_1(t^n, x) = U_1^n(x,  \tau=S^n(x)/\eps), \;\;\; u_2(t^n, x) = e^{-iS^n(x)/\eps} V_2^n(x, \tau=S^n(x)/\eps),
$$
where the evaluation at $\tau=S^n(x)/\eps$ is performed by trigonometric interpolation since the solution are
periodic with respect to the $\tau$ variable.

\subsection{Numerical results}
This section is devoted to numerical illustration of the new approach for the case of $2$x$2$ systems.
We solve \eqref{eq_syst} with $a_1(x)=1, a_2(x)=4, R(u)=0, E(t, x)=3/2+\cos(x)$, and
$$
C=
\left(
\begin{array}{ccll}
0 & 1\\
-1 & 0
\end{array}
\right).
$$
We consider the following initial condition
$$
u(t=0, x)=\Big(1+\frac{1}{2}\cos(x)+i\sin(x), 1+\frac{1}{2}\cos(x)+i\sin(x)\Big), \;\; x\in [0, 2\pi].
$$
As in the scalar case, we compare the solution obtained by a direct method (time splitting with
exact (in time) integration of each substep) and by the new approach. The direct method uses
resolved parameters so its solution provides a reference which will be compared to the
solution of the new method. For this latter method, the following numerical parameters
are used: $\Delta x=2\pi/N_{ts}, \Delta t=\Delta x/(2\max(a_1, a_2)), N_\tau=64$, where $N_{ts}$ is the number
of (uniform) grid points in the spatial direction.

In the following figures, we are interested in the $\ell^\infty$ error in space
(for different values of $N_{ts}$) at the final time $t_f=0.1$,
between the new method and the reference solution,  for different values of $\eps$.


In Figure \ref{fig1syst}, the solution obtained with the new method
is computed with the corrected initial condition and with an exact solution for  the phase $S$.
We plot the $\ell^\infty$ error for different values of $\eps$ as a function of $N_{ts}$ (left part)
and the $\ell^\infty$ error as a function of $\eps$ for different $N_{ts}$ (right part). 
As in the scalar case, the uniform accuracy is observed: the order of accuracy is independent of
$\varepsilon$ and the error is constant with respect to $\varepsilon$.

In Figure \ref{fig2syst}, we study the influence of the numerical approximation of $S$ on the error.
We used for the approximation of $S$ a first order upwind scheme in space with a first order time integrator.
We plot the same diagnostics as before.
As in the scalar case, we observe a bad behavior when $\eps$ becomes small. Then, in Figure \ref{fig3syst},
we consider an improved numerical approximation of $S$ by using a pseudo-spectral method in
space with a $4$th-order Runge-Kutta time integrator.
We then observe that the uniform accuracy is recovered.

In Figure \ref{fig4syst}, an exact computation of $S$ is used but the initial data is not corrected.
Again, we plot the $\ell^\infty$ error. As expected the uniform accuracy is lost, in particular in the intermediate regime.

\begin{figure}
\begin{center}
\begin{tabular}{ccll}
\includegraphics[width=0.6\linewidth]{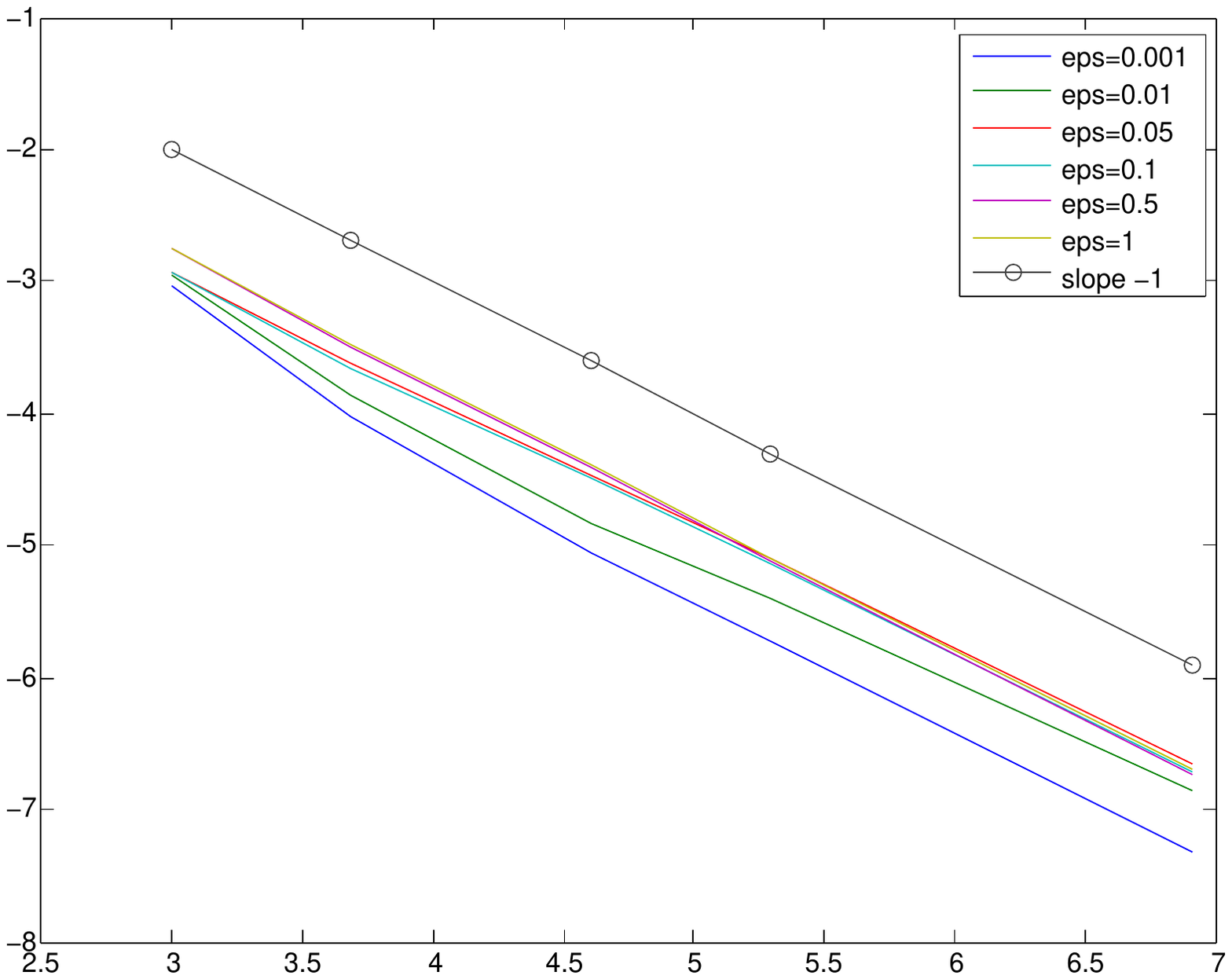}&
\hspace{-2cm}\vspace{-3cm}
\includegraphics[width=0.6\linewidth]{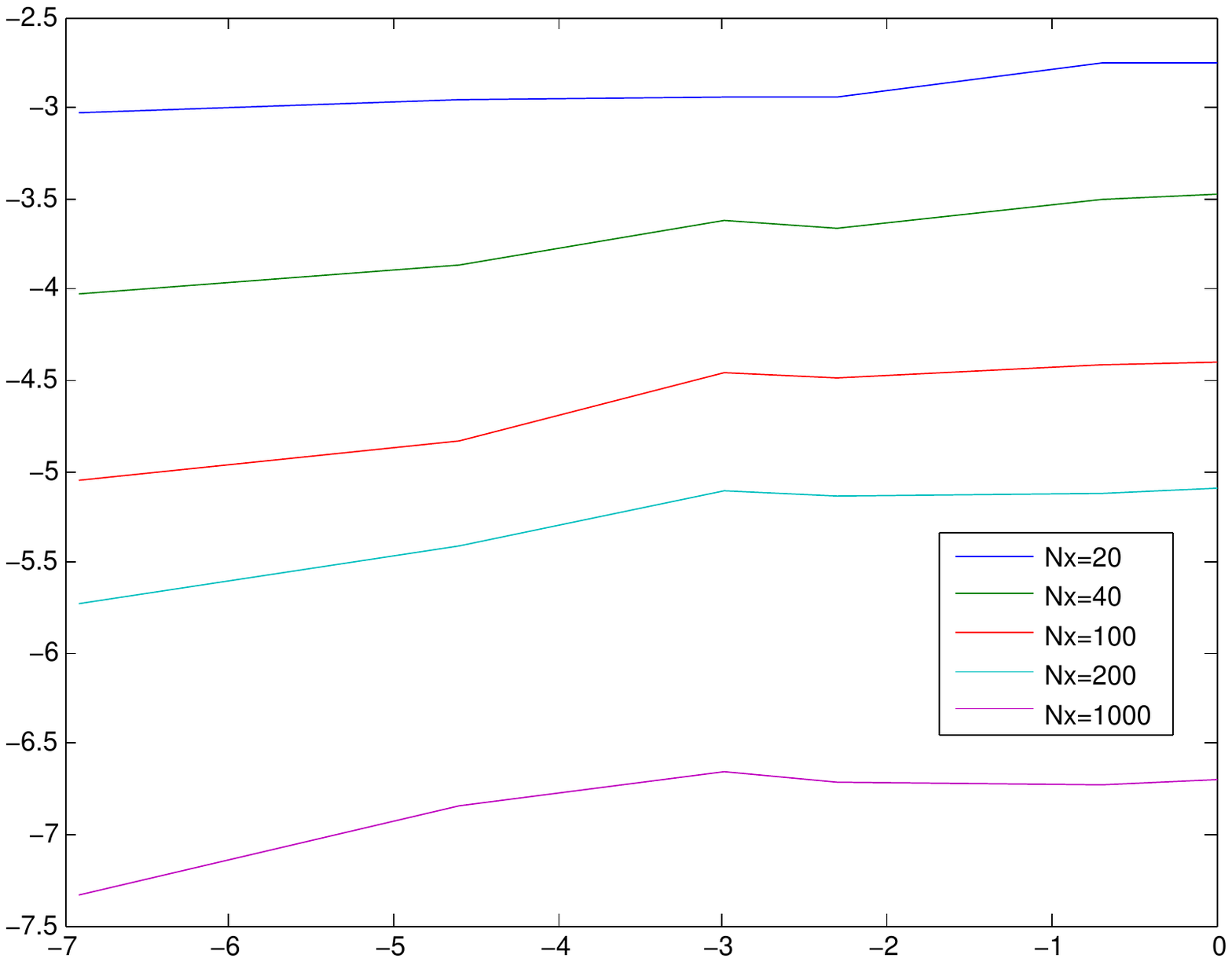}\\
\end{tabular}
\end{center}
\caption{Plot of the $\ell^\infty$ error for the new method {\it with} corrected  initial condition
and exact computation for $S$. Left: error ($\log$-$\log$ scale) as a function of $N_{ts}$
($N_{ts}= 20, 40, 100, 200, 1000$)
for different values of $\eps$ ($\eps=1, \dots, 10^{-3}$). Right: error ($\log$-$\log$ scale) as a function of $\eps$
for different $N_{ts}$.}
\label{fig1syst}
\end{figure}

\begin{figure}
\vspace{-4cm}
\begin{center}
\begin{tabular}{ccll}
\includegraphics[width=0.6\linewidth]{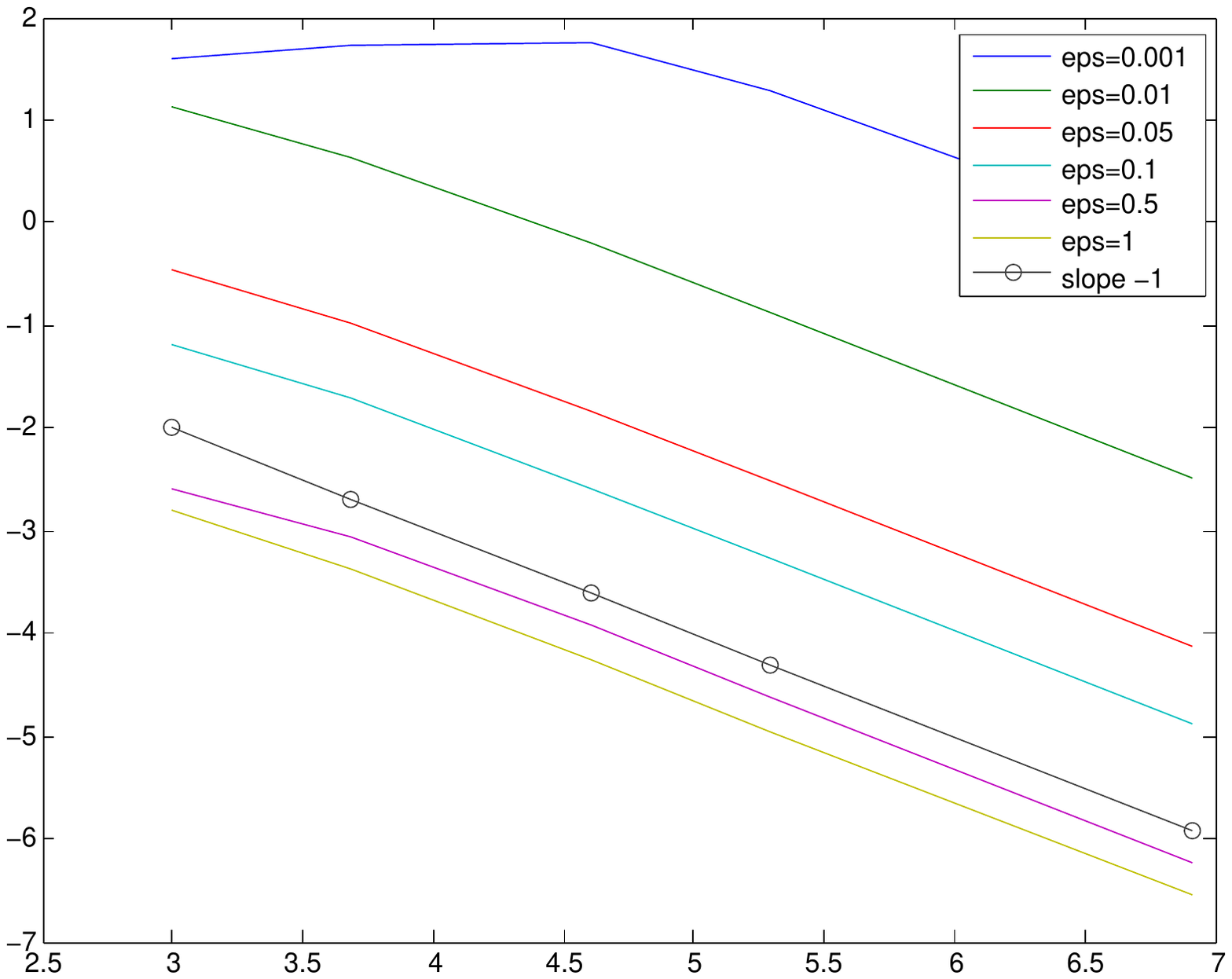}&
\hspace{-2cm}\vspace{-3cm}
\includegraphics[width=0.6\linewidth]{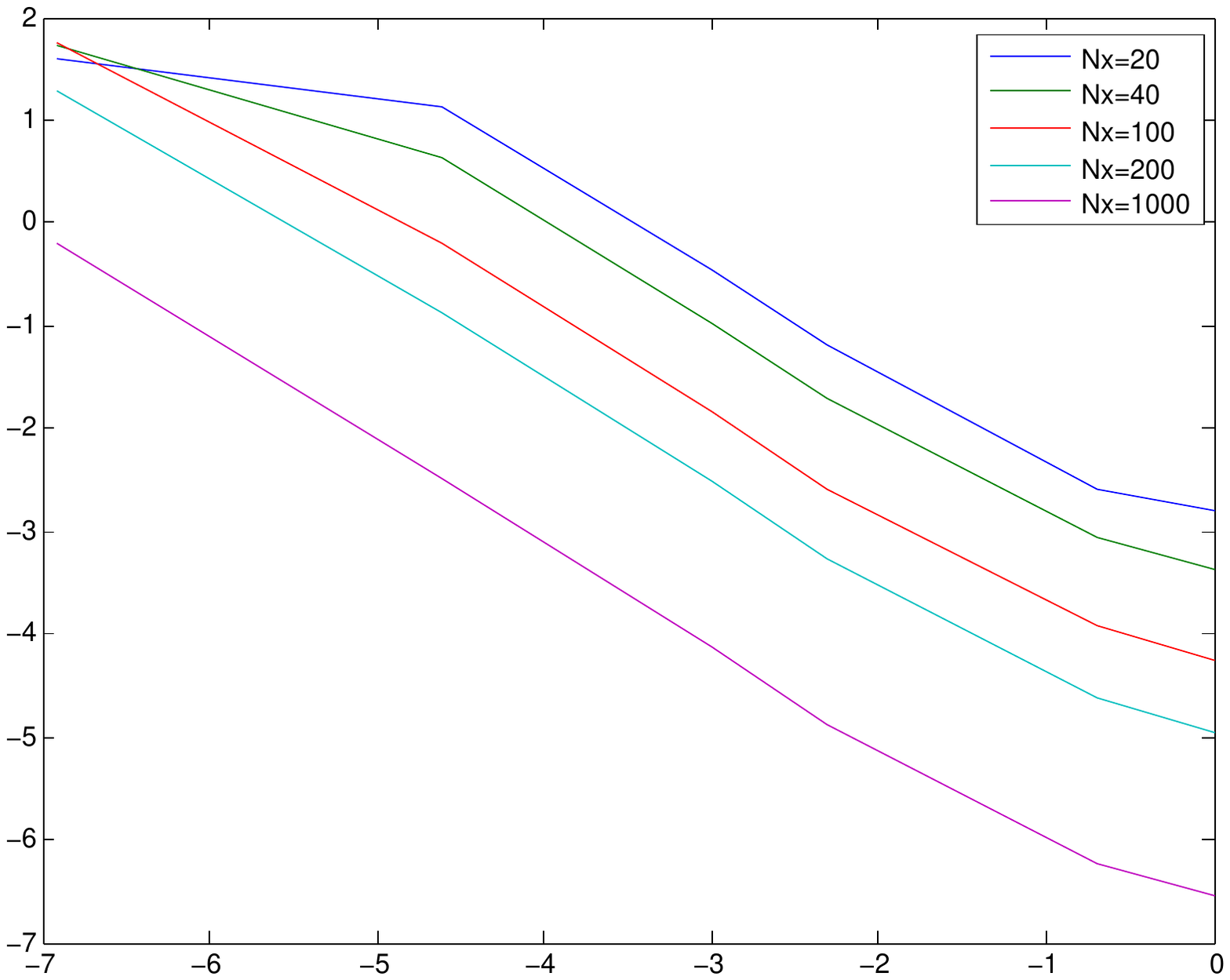}\\
\end{tabular}
\end{center}
\caption{Plot of the $\ell^\infty$ error for the new method {\it with} corrected initial condition
and numerical approximation for $S$. Left: error ($\log$-$\log$ scale) as a function of $N_{ts}$
($N_{ts}= 20, 40, 100, 200, 1000$)
for different values of $\eps$ ($\eps=1, \dots, 10^{-3}$). Right: error ($\log$-$\log$ scale) as a function of $\eps$
for different $N_{ts}$.}
\label{fig2syst}
\end{figure}

\begin{figure}
\begin{center}
\begin{tabular}{ccll}
\includegraphics[width=0.6\linewidth]{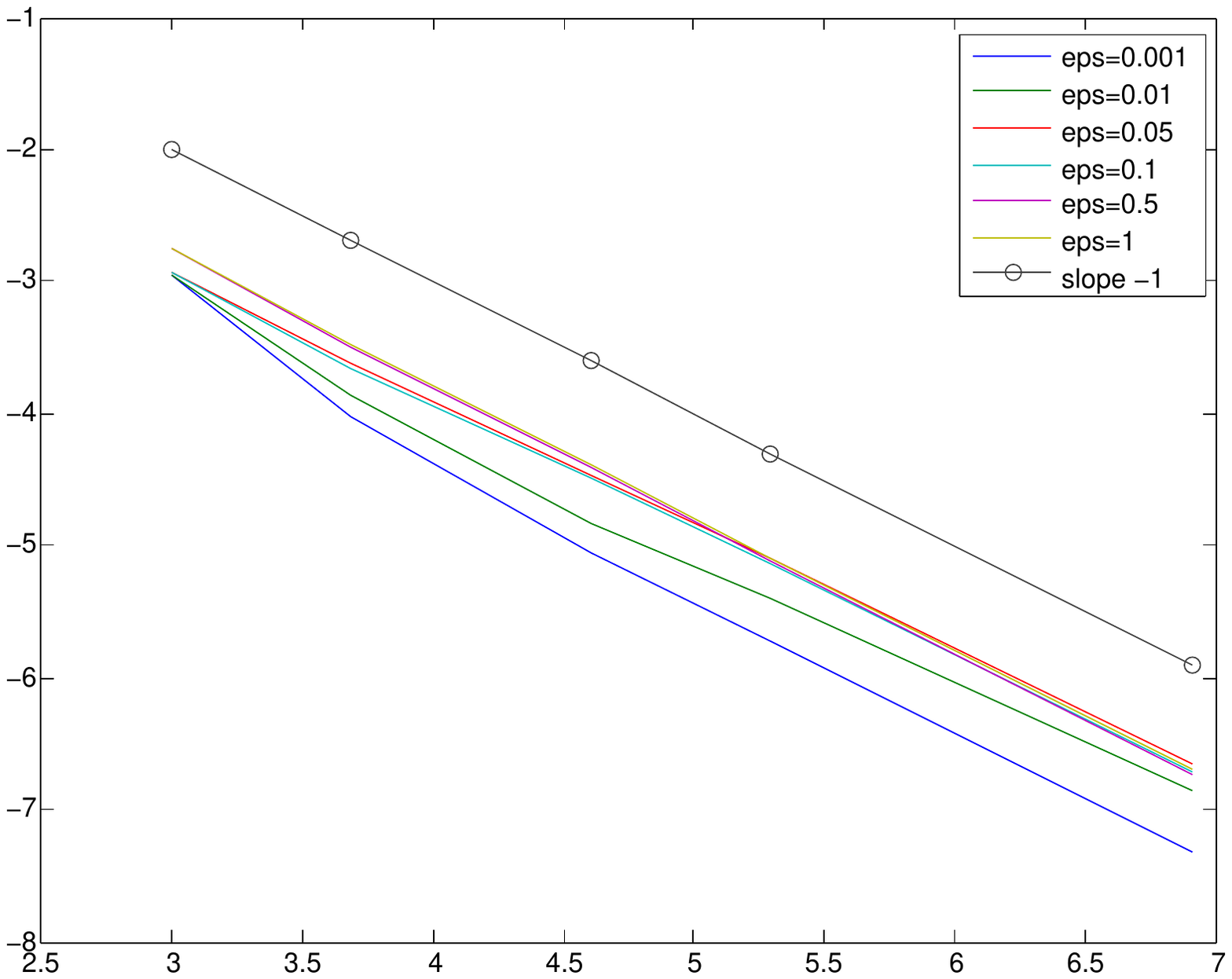}&
\hspace{-2cm}\vspace{-3cm}
\includegraphics[width=0.6\linewidth]{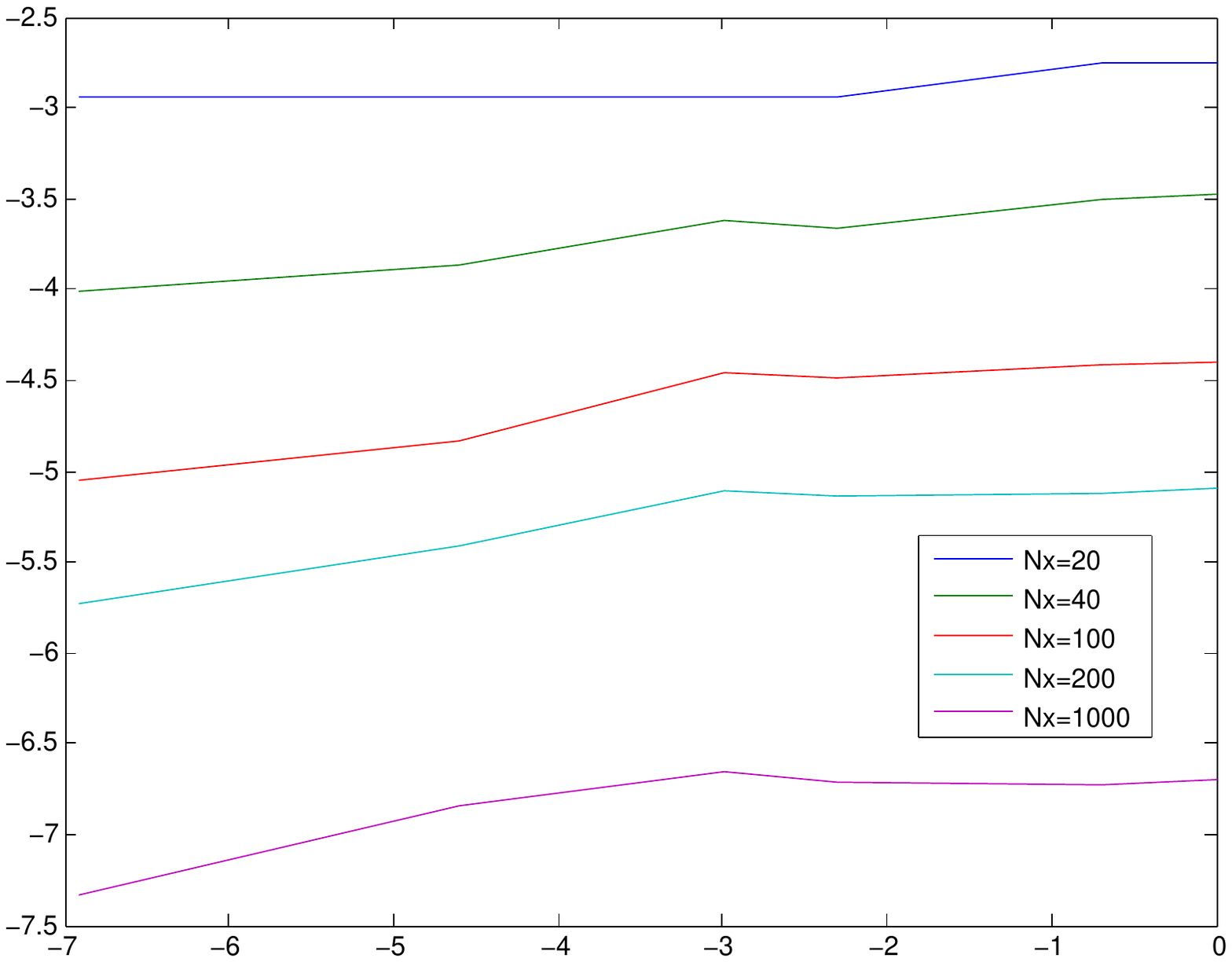}\\
\end{tabular}
\end{center}
\caption{Plot of the $\ell^\infty$ error for the new method {\it with} corrected initial condition
and an improved numerical approximation for $S$ (pseudo-spectral in space and $4$th order Runge-Kutta). Left: error ($\log$-$\log$ scale) as a function of $N_{ts}$
($N_{ts}= 20, 40, 100, 200, 1000$)
for different values of $\eps$ ($\eps=1, \dots, 10^{-3}$). Right: error ($\log$-$\log$ scale) as a function of $\eps$
for different $N_{ts}$.}
\label{fig3syst}
\end{figure}

\begin{figure}
\vspace{-4cm}
\begin{center}
\begin{tabular}{ccll}
\includegraphics[width=0.6\linewidth]{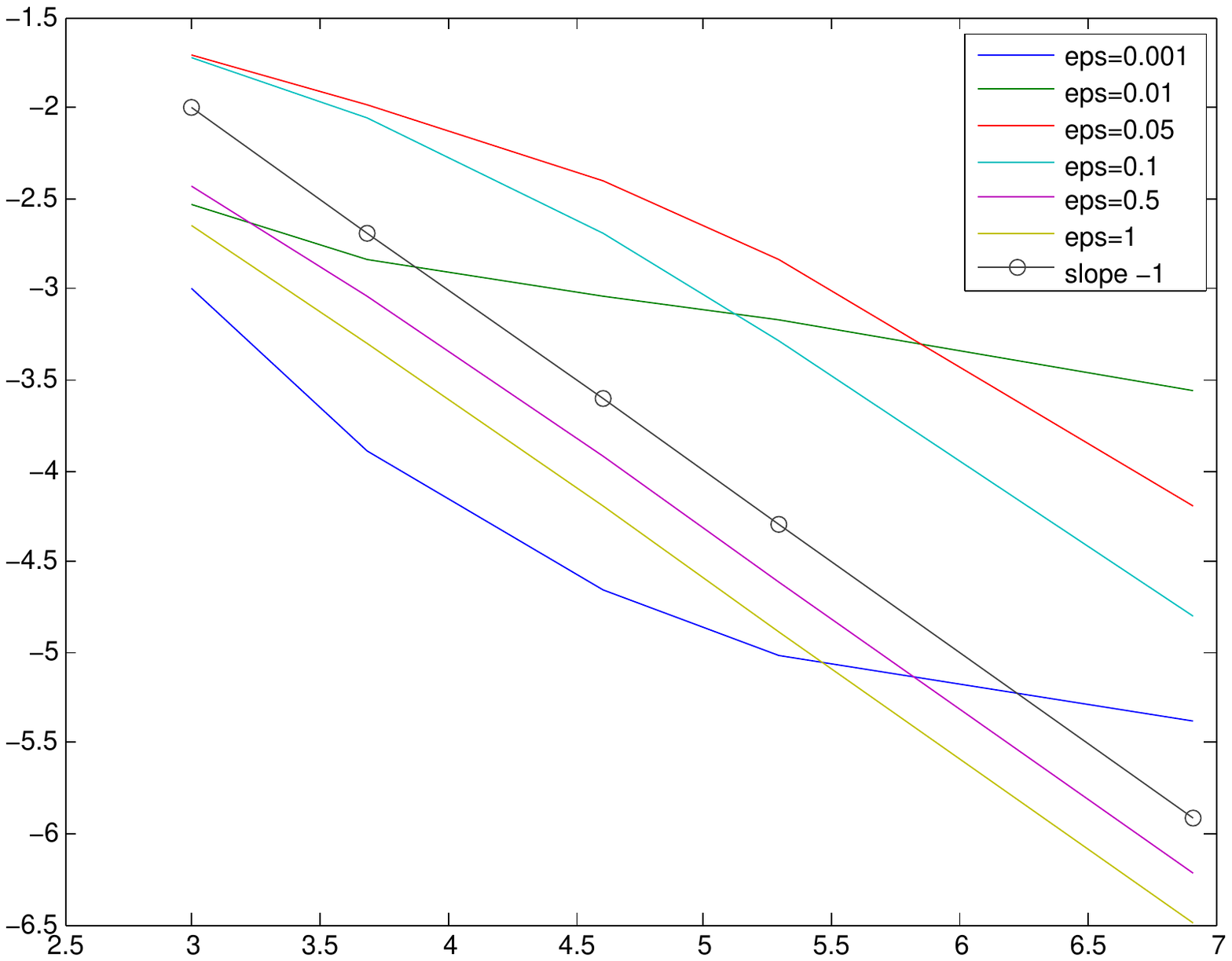}&
\hspace{-2cm}\vspace{-3cm}
\includegraphics[width=0.6\linewidth]{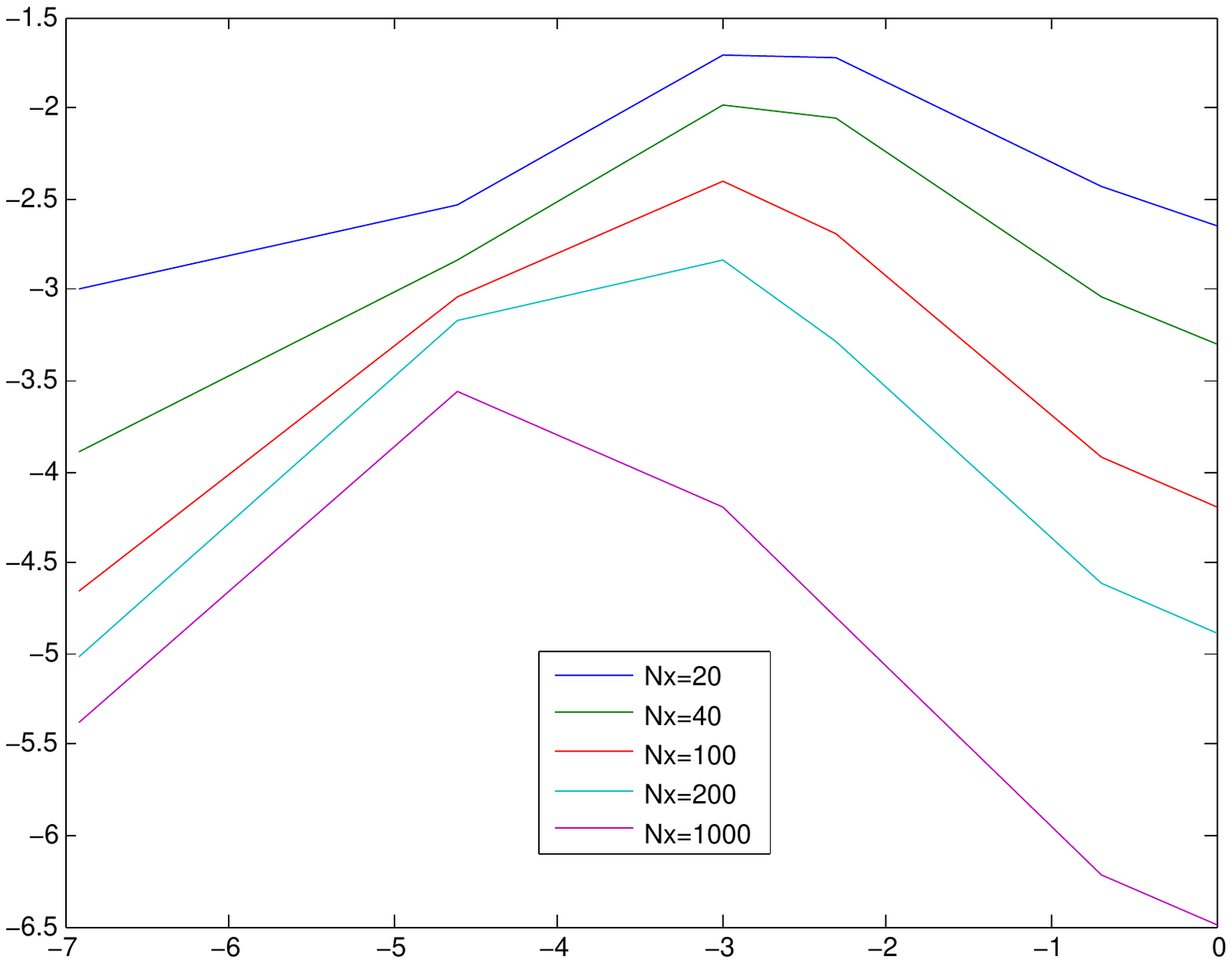}\\
\end{tabular}
\end{center}
\caption{Plot of the $\ell^\infty$ error for the new method {\it without} corrected initial condition
and an exact computation for $S$. Left: error ($\log$-$\log$ scale) as a function of $N_{ts}$
($N_{ts}= 20, 40, 100, 200, 1000$)
for different values of $\eps$ ($\eps=1, \dots, 10^{-3}$). Right: error ($\log$-$\log$ scale) as a function of $\eps$
for different $N_{ts}$.}
\label{fig4syst}
\end{figure}

Finally, in the following figures, we illustrate the performances of the new method
using the same data as before except the initial condition
$$
u(t=0, x)=\Big(1+\frac{1}{2}\cos(x), 1+\frac{1}{2}\cos(x)\Big), \;\; x\in [0, 2\pi].
$$
We compare a reference solution (computed with a direct method using resolved
numerical parameters $N_d=4000, \Delta t=10^{-4}$)
and the  solution of the new method at $t_f=1$, for $\eps=0.01$.
For the new method, we choose $N_\tau=8$, a well-prepared initial condition, an exact phase $S$
and different values of $N_{ts}$ are considered.
In Figure \ref{fig6syst}, we plot the real and imaginary part of the first component of the solution
as a function of space for $N_{ts}=20, 40, 100$ (and $\Delta t=2\pi/(2N_{ts}\max(a_1, a_2))$).
Even with a very coarse mesh, we observe that the new method is able to capture the high
space oscillations of the solution. In Figure \ref{fig7syst}, the real part of the second component
of the solution is displayed as a function of $x$, for  $N_{ts}=20, 40, 100$. On the right column
(which is a zoom of the left one), we see that the new solution almost coincides
with the reference one even if the spatial mesh $\Delta x\approx 0.31, 0.15, 0.06$ is large compared
to the size of the smallest oscillations (of order $\eps=0.01$).

 \begin{figure}
\begin{center}
\begin{tabular}{cclll}
\includegraphics[width=0.6\linewidth]{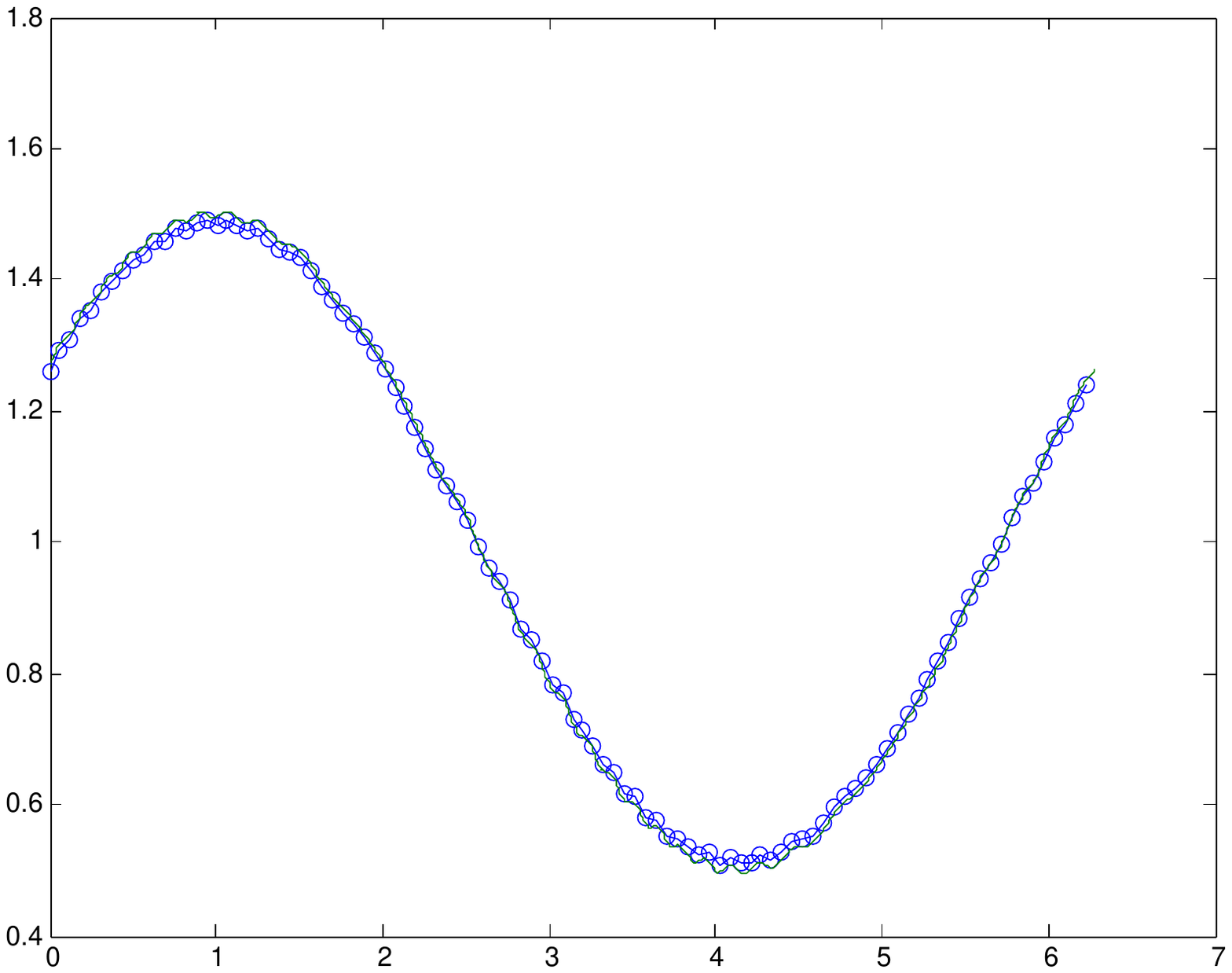}&
\hspace{-2cm}\vspace{-5.5cm}
\includegraphics[width=0.6\linewidth]{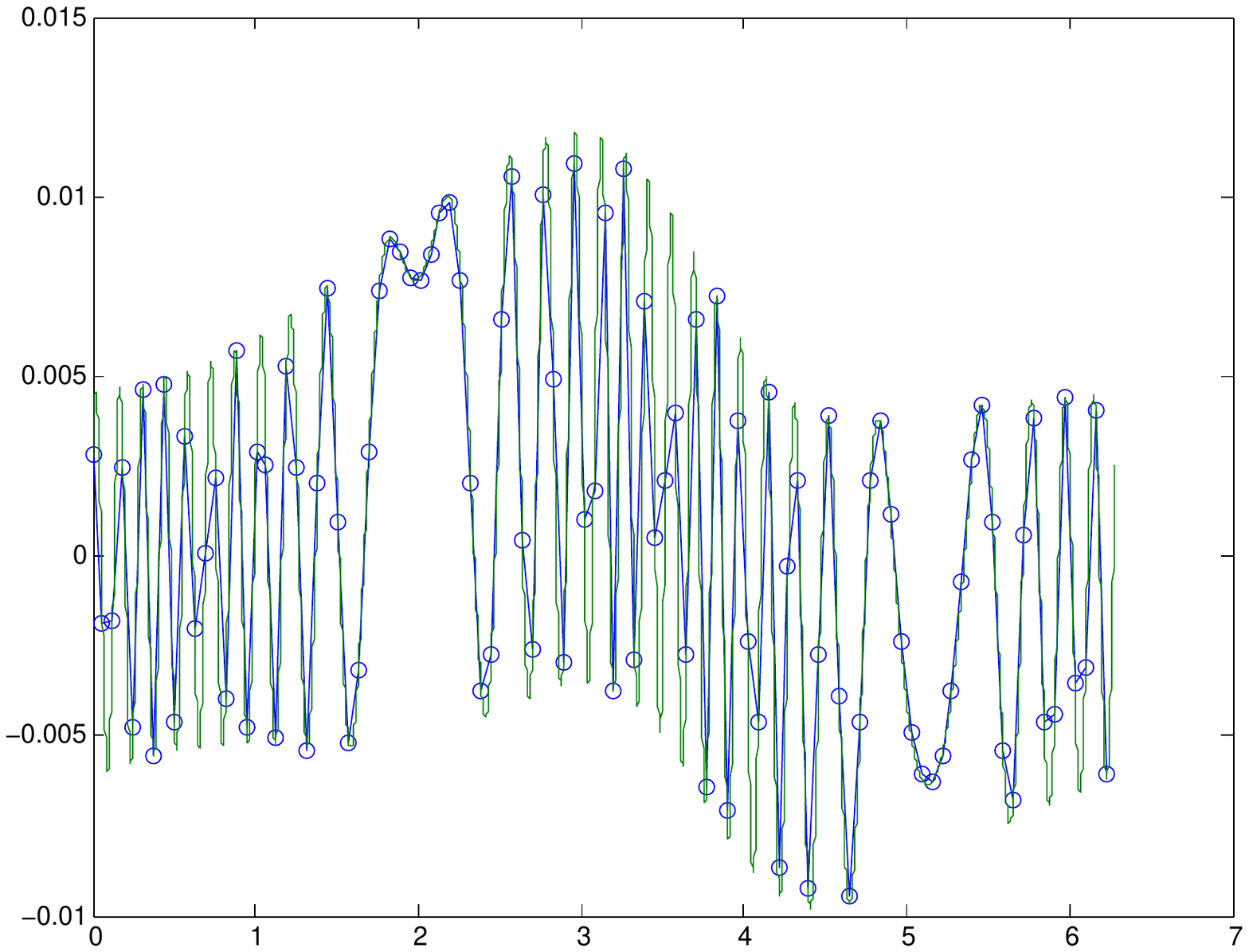}\\
\includegraphics[width=0.6\linewidth]{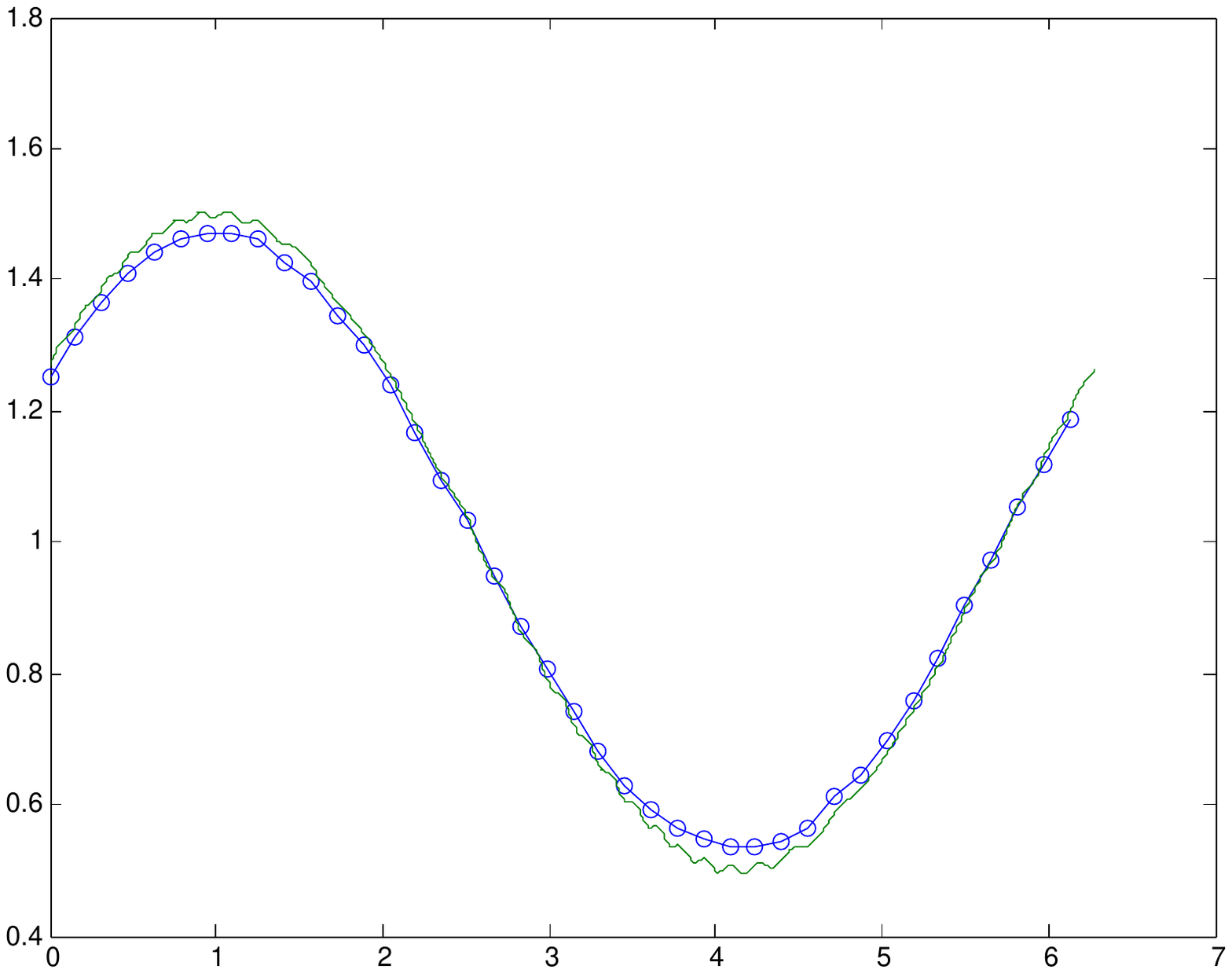}&
\hspace{-2cm}\vspace{-5.5cm}
\includegraphics[width=0.6\linewidth]{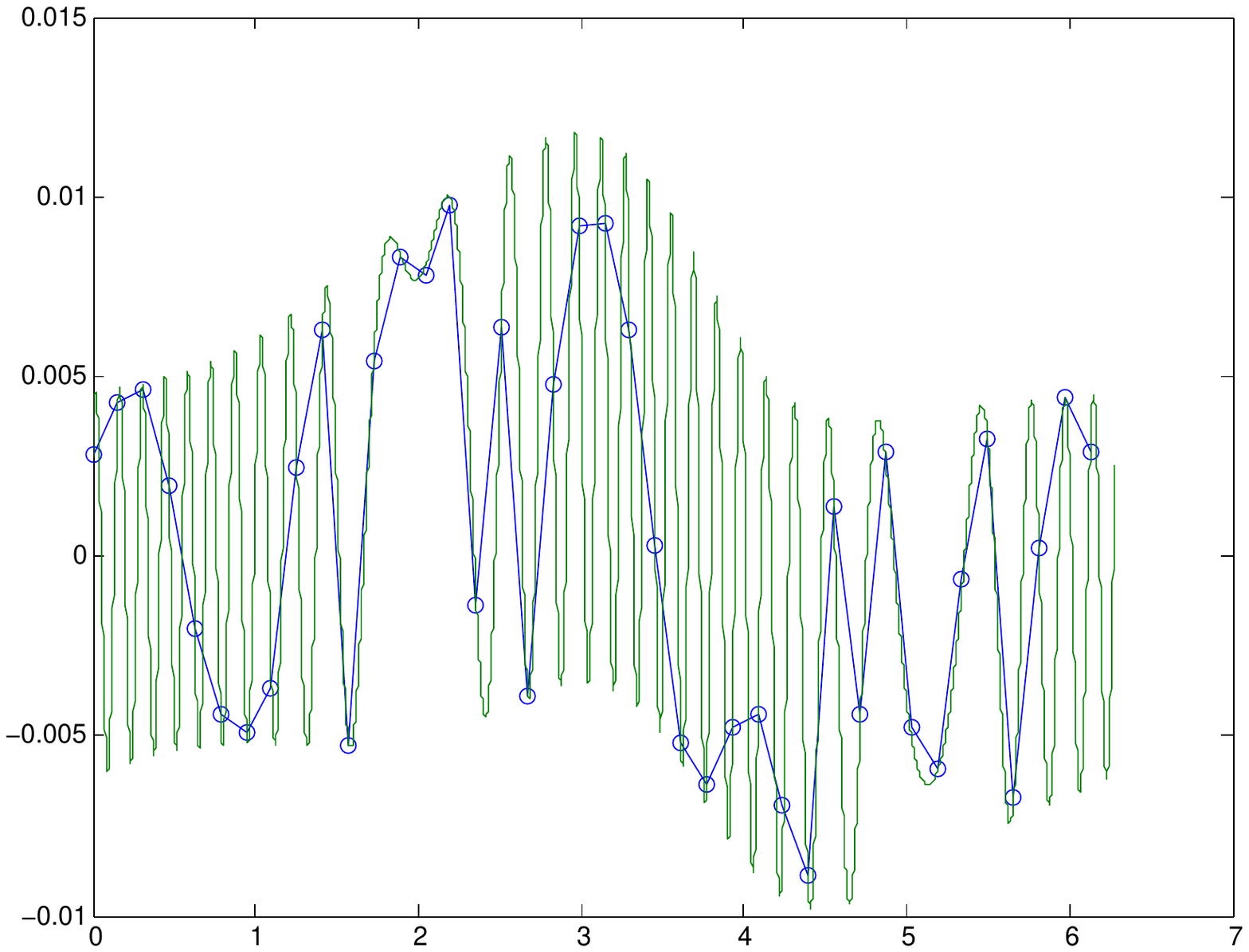}\\
\includegraphics[width=0.6\linewidth]{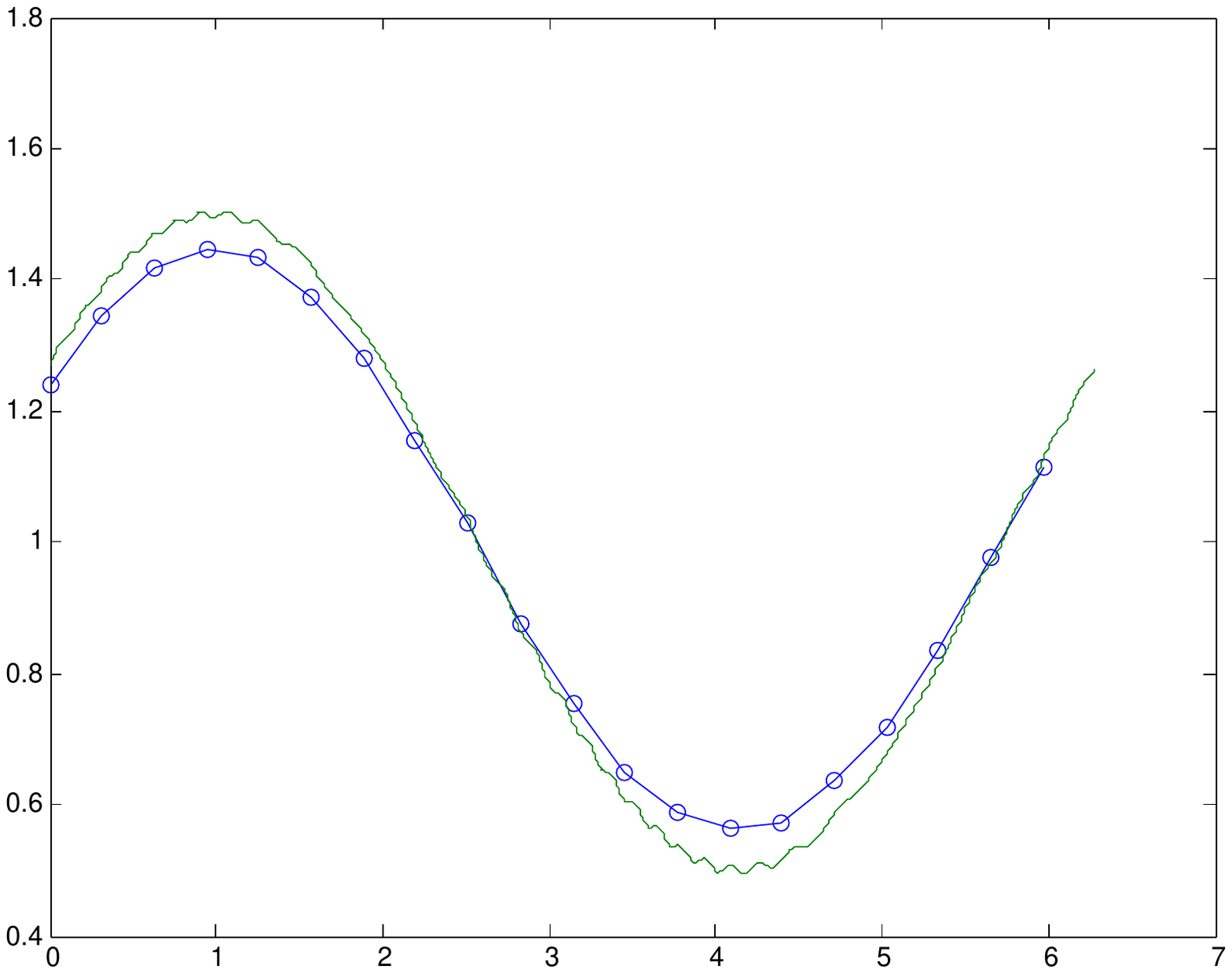}&
\hspace{-2cm}\vspace{-3cm}
\includegraphics[width=0.6\linewidth]{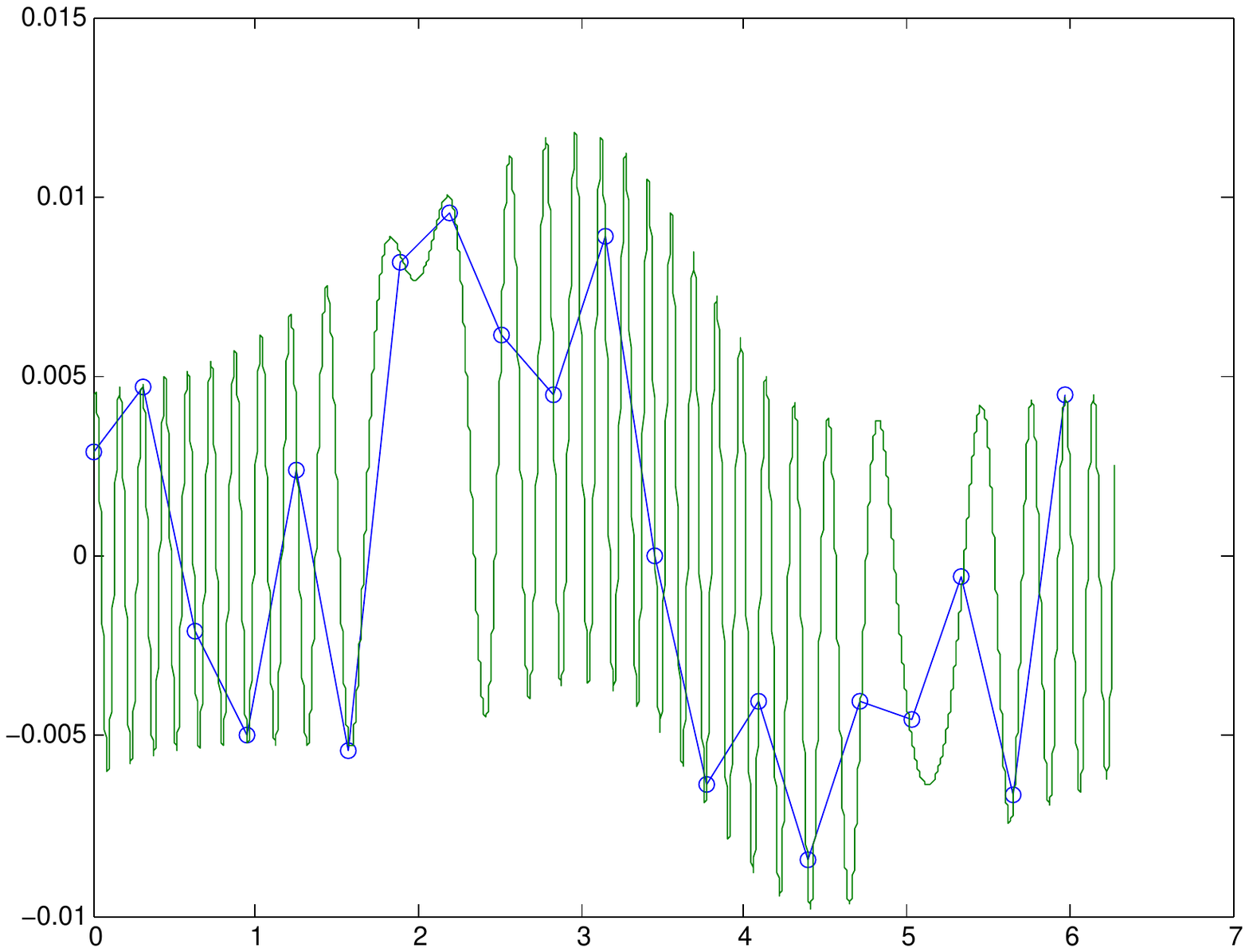}
\end{tabular}
\end{center}
\caption{The first component of the solution (left column: real part, right column: imaginary part)
as a function of $x$ at time $t_f=1$, $\eps=0.01$.
Comparison between the reference solution (with $N_d=4000, \Delta t=10^{-4}$)
and the new method with $N_{ts}=100, 40, 20$ (from top to bottom). }
\label{fig6syst}
\end{figure}

 \begin{figure}
\begin{center}
\begin{tabular}{cclll}
\includegraphics[width=0.6\linewidth]{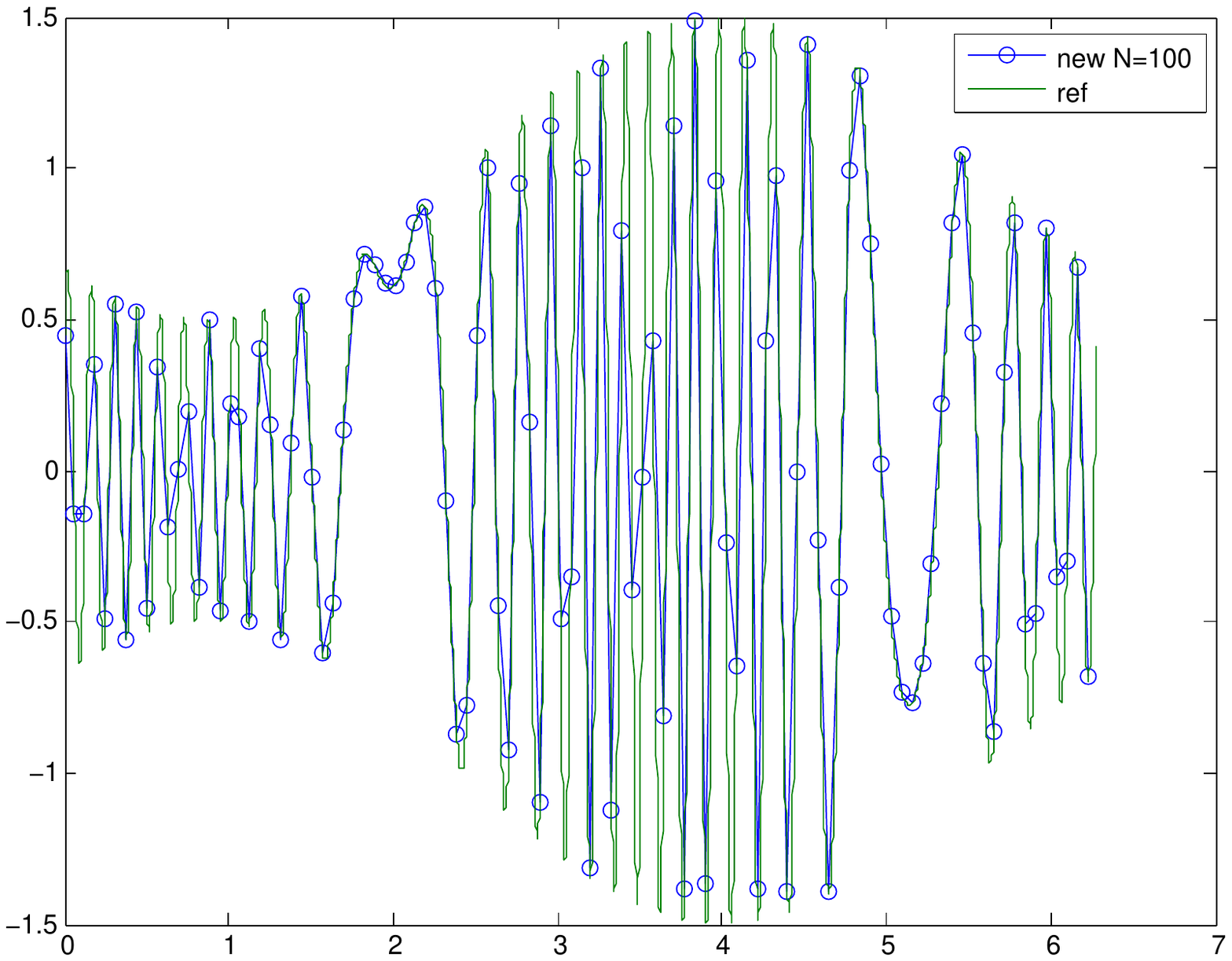}&
\hspace{-2cm}\vspace{-5.5cm}
\includegraphics[width=0.6\linewidth]{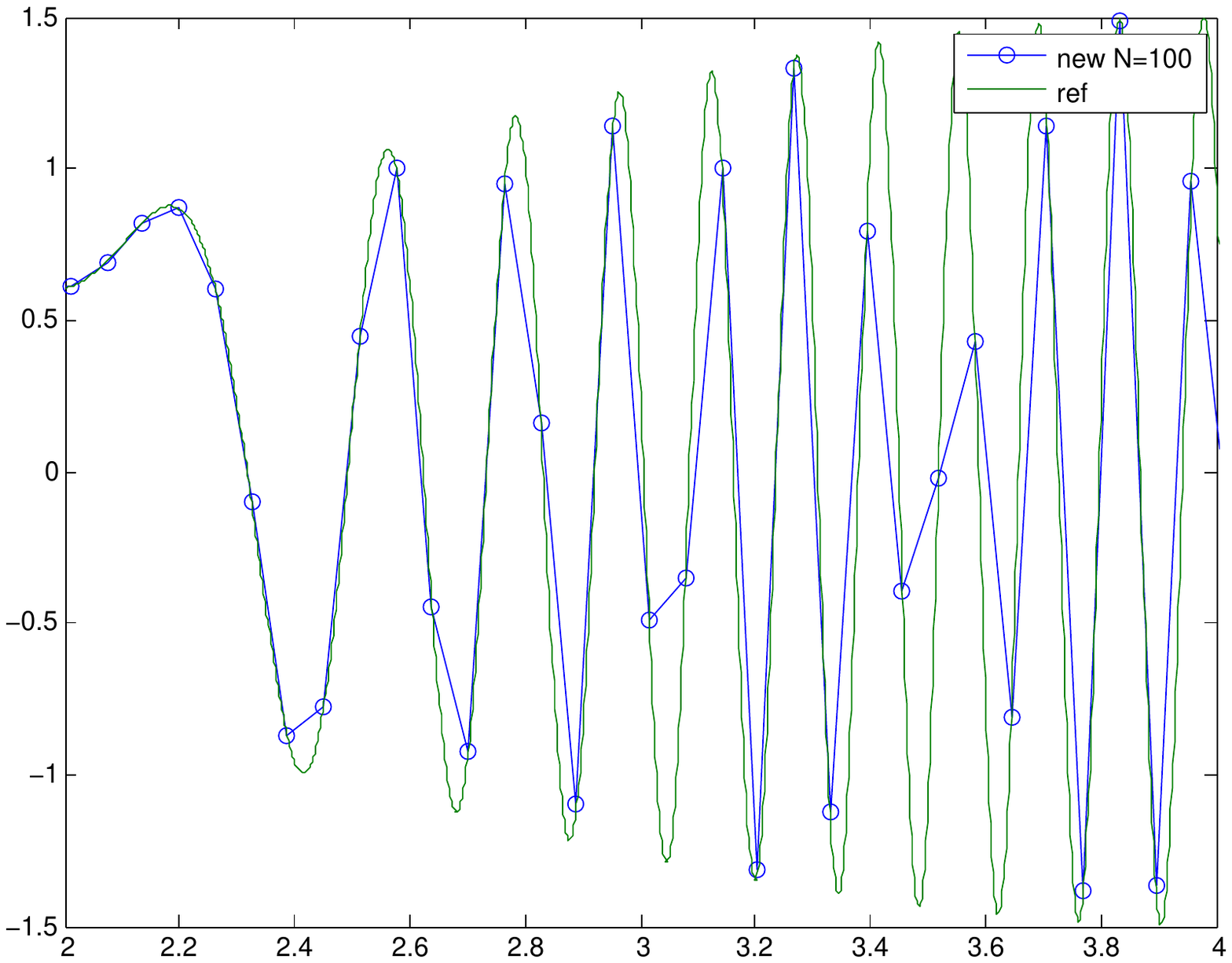}\\
\includegraphics[width=0.6\linewidth]{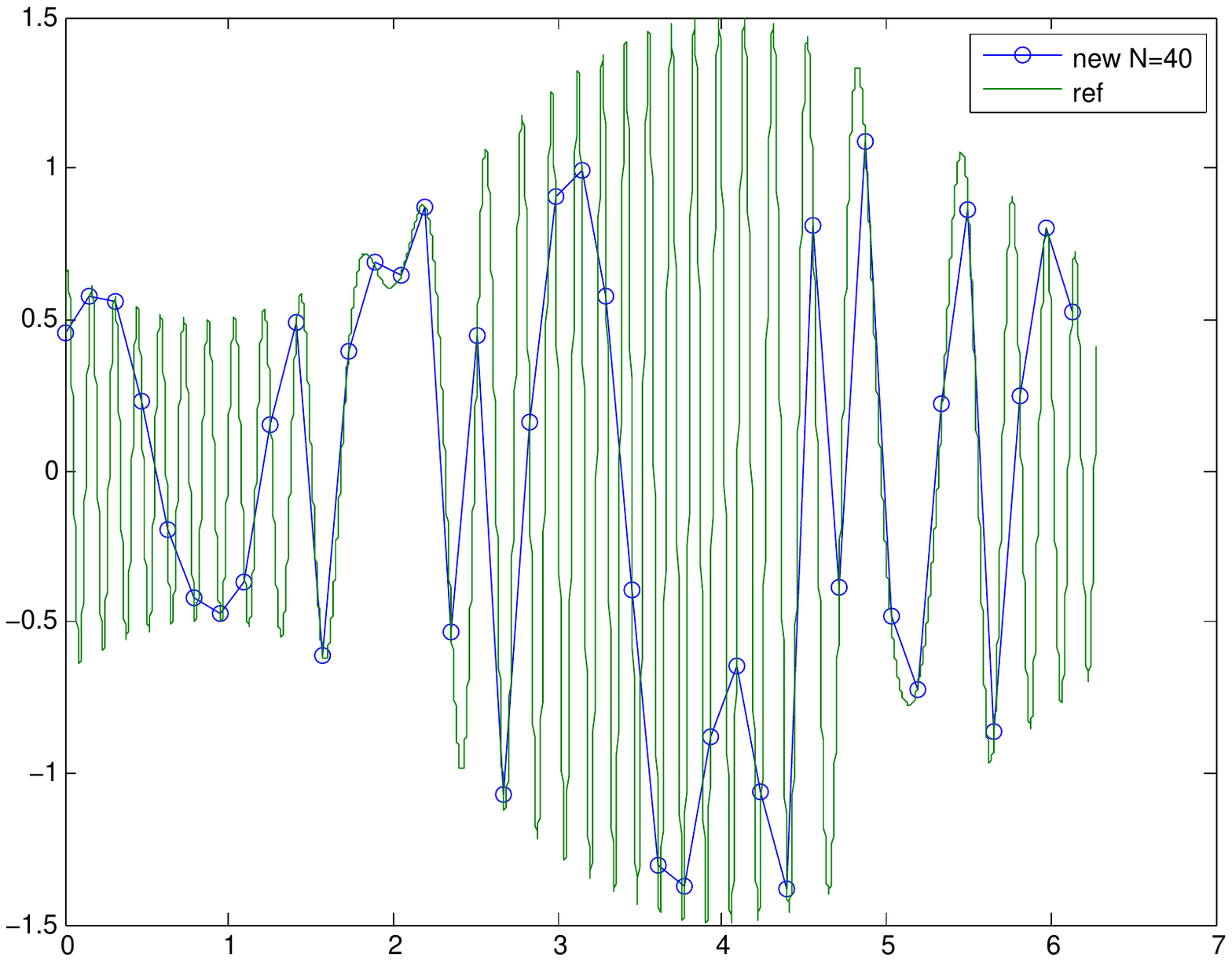}&
\hspace{-2cm}\vspace{-5.5cm}
\includegraphics[width=0.6\linewidth]{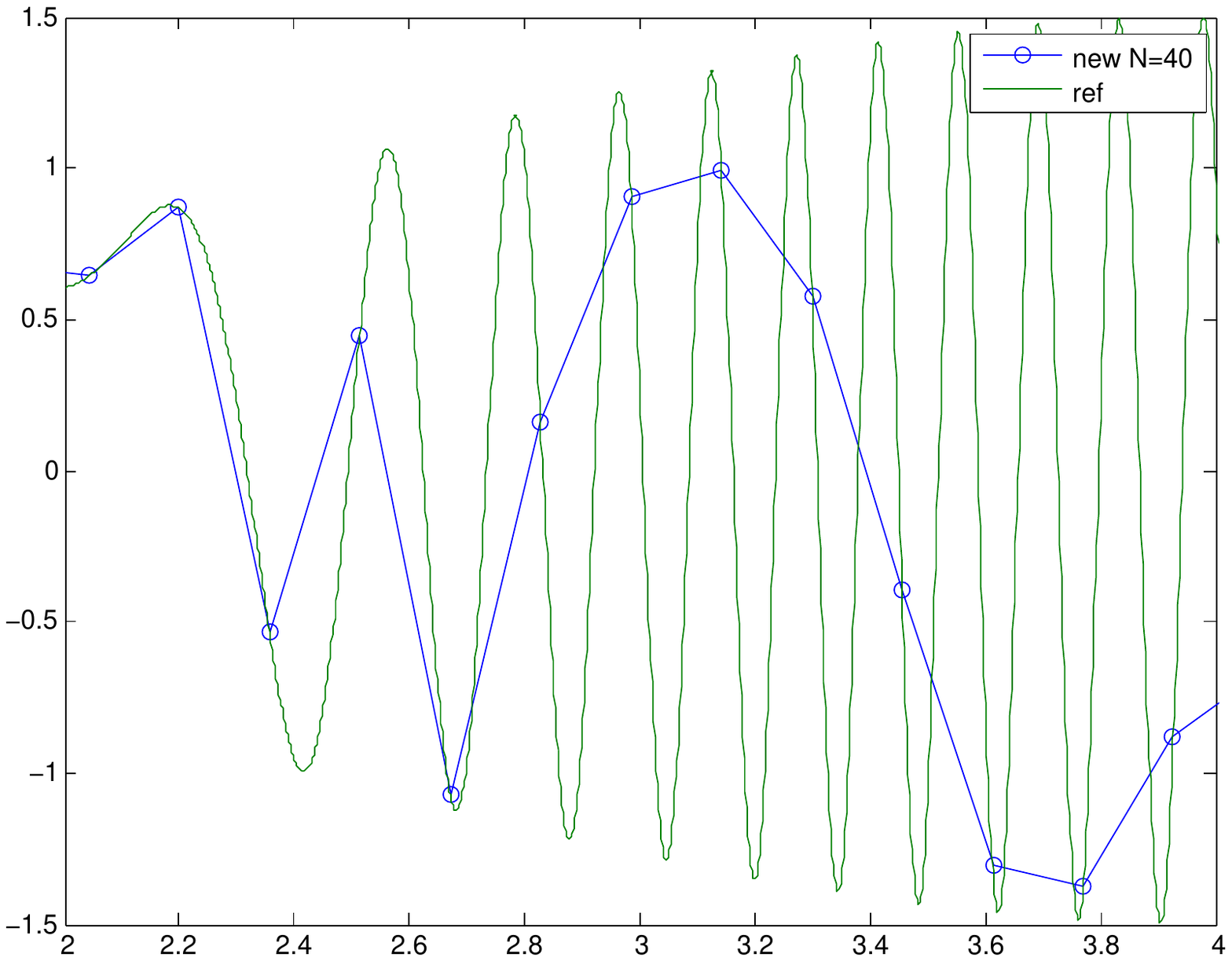}\\
\includegraphics[width=0.6\linewidth]{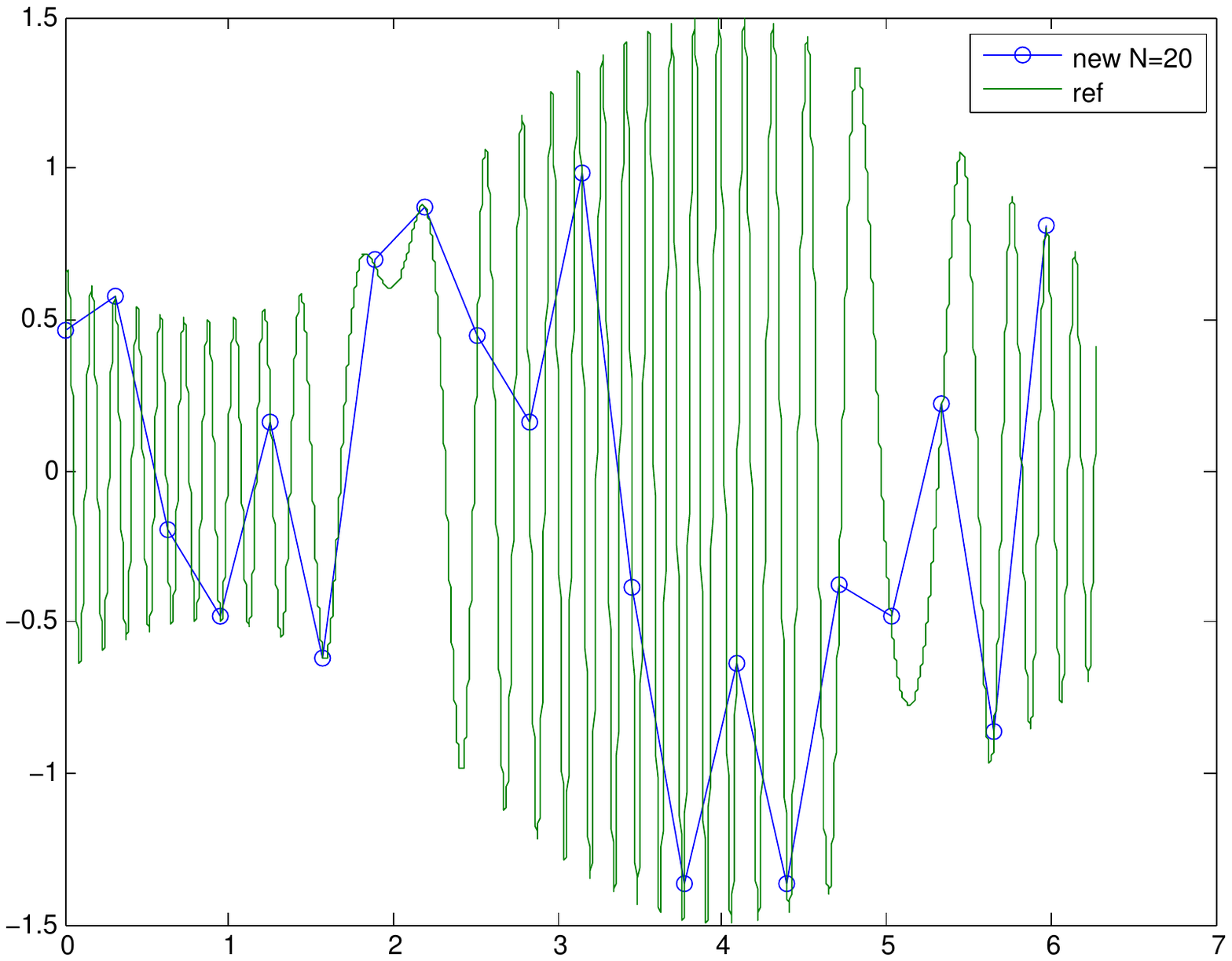}&
\hspace{-2cm}\vspace{-3cm}
\includegraphics[width=0.6\linewidth]{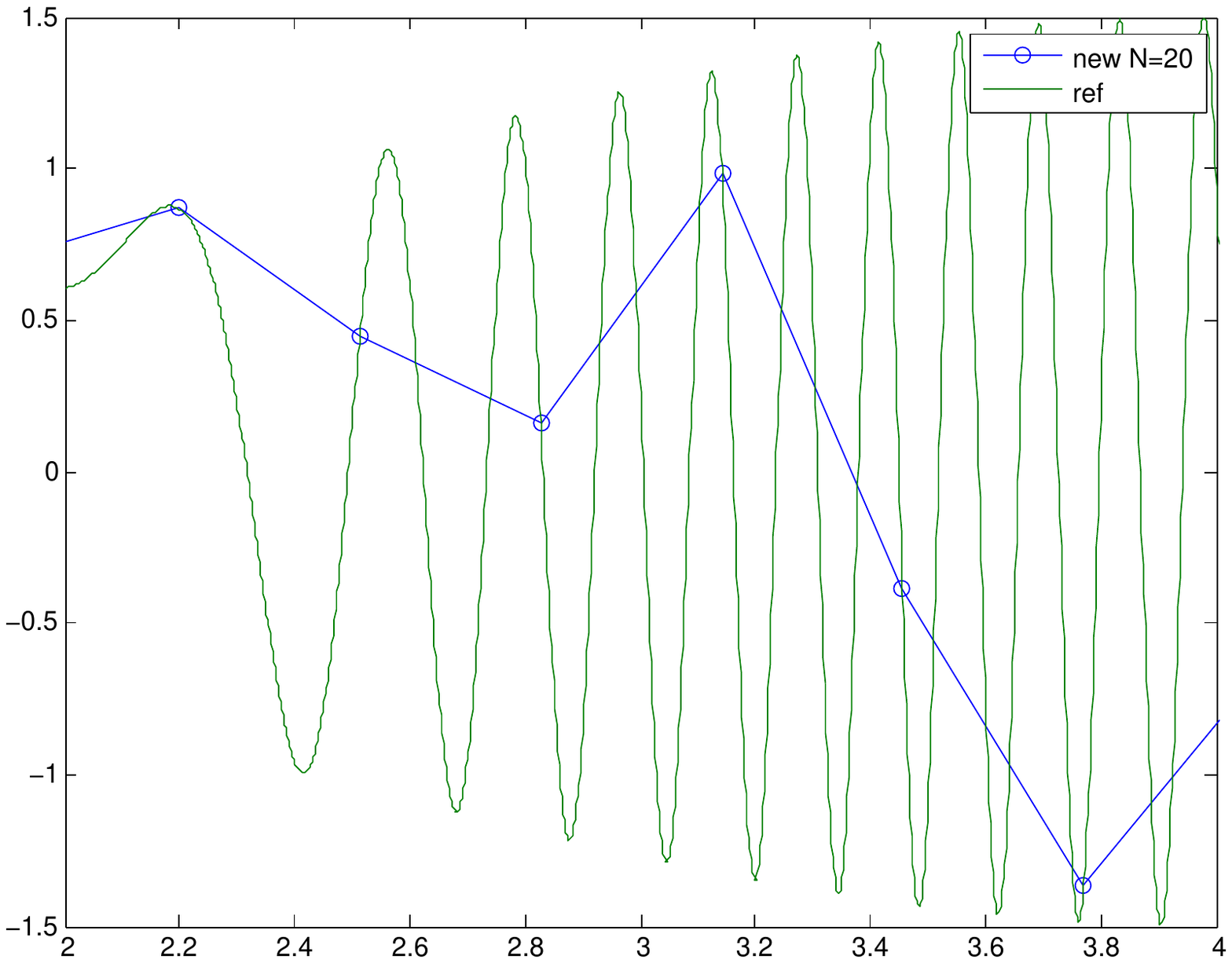}
\end{tabular}
\end{center}
\caption{The real part of the second component of the solution as a function of $x$ at time $t_f=1$, $\eps=0.01$.
Comparison between the reference solution (with $N_d=4000, \Delta t=10^{-4}$)
and the new method with $N_{ts}=100, 40, 20$ (from top to bottom). }
\label{fig7syst}
\end{figure}

\subsection{An application to a semiclassical surface hopping model}

We now show that the general approach  described above can be applied  to efficiently solve  the following semiclasscial surface hopping model, introduced
in \cite{jin-morandi}:

\beq
\label{modelBC-3}
\begin{array}{l} \ds
\partial_t f^+ +  p\cdot \nabla_x f^+ - \nabla_x (U+E)\cdot \nabla_pf^+= \overline b^i f^i + b^i \overline f^i, \\
\ds
\partial_t f^- +p\cdot \nabla_x f^- - \nabla_x (U-E)\cdot \nabla_pf^-= -\overline b^i f^i - b^i \overline f^i, \\ \ds
\partial_t f^i +p\cdot \nabla_x f^i + \nabla_x U\cdot \nabla_pf^i= -i \frac{2E}{\eps}f^i +
 b^i (f^- - f^+)  + (b^+-b^-)f^i,
\end{array}
\eeq
where $(f^+(t, x, p),f^-(t, x, p),f^i(t, x, p)) \in \mathbb{R}\times \mathbb{R}\times \mathbb{C}$, \  $(t,x,p)\in \mathbb{R}_+\times \mathbb{R}^d\times \mathbb{R}^d$, \ and $b^\pm\in \mathbb{C}, b^i\in \mathbb{C}, U\in \mathbb{R}, E\in \mathbb{R}$ are given functions  depending only on the space variable $x$.  We denote also by
$$(f^+(0, x, p),f^-(0, x, p),f^i(0, x, p)) = (f^+_{in}(x, p),f^-_{in}(x, p),f^i_{in}(x, p))$$
the initial conditions.

This model approximates semiclassically the nucleaonic Schr\"odinger system arising
from the Born-Oppenheimer approximation with non-adiabatic corrections. The right hand side describes the
interband transition between different potential energy surfaces
($2E$ is the {\it band gap} between two energy surfaces), and the
coefficients are related to Berry connection.
We refer to  \cite{jin-morandi} for more details.

As explained above, the general idea is to introduce a
phase $S(t,x,p)$ designed to follow the main oscillations in this model.
We then consider the phase $S(t,x,p)$, solution to
\beq
\label{eq_S_bc}
\partial_t S+p\cdot \nabla_x S + \nabla_x U\cdot \nabla_p S= 2E, \quad S(0,x,p)=0,
\eeq
and introduce the augmented unknowns $(F^\pm,F^i)(t,x,p,\tau)$ satisfying
$$ f^\pm(t,x,p)= F^\pm(t,x,p,S(t,x,p)/\eps), \quad f^i(t,x,p)=F^i(t,x,p,S(t,x,p)/\eps).$$
One then has:
\begin{eqnarray}
\ds &&\partial_t F^+ +  p\cdot \nabla_x F^+ - \nabla_x (U+E)\cdot \nabla_pF^+=\nonumber\\
&&\hspace{2cm}\frac{-1}{\eps} \left (2E- \nabla_x(2U+E)\cdot \nabla_pS\right) \partial_\tau F^+ + \overline b^i F^i + b^i \overline F^i, \nonumber\\
\ds
&&\partial_t F^- +p\cdot \nabla_x F^- - \nabla_x (U-E)\cdot \nabla_pF^-=\nonumber\\
&&\hspace{2cm}\frac{-1}{\eps} \left (2E- \nabla_x(2U-E)\cdot \nabla_pS\right) \partial_\tau F^- -\overline b^i F^i - b^i \overline F^i, \nonumber\\ \ds
&&\partial_t F^i +p\cdot \nabla_x F^i + \nabla_x U\cdot \nabla_pF^i=\nonumber\\
\label{modelBC-FFi}
&&\hspace{2cm}- \frac{2E}{\eps}\left(\partial_\tau F^i+ iF^i\right) +
 b^i (F^- - F^+)  + (b^+-b^-)F^i.
\end{eqnarray}
Let $G^i= e^{i\tau} F^i$, then
\beq
\label{modelBC-FGi}
\begin{array}{l} \ds
\ds \partial_t F^+ +  p\cdot \nabla_x F^+ - \nabla_x (U+E)\cdot \nabla_pF^+=
-\frac{{\cal E}^{+}}{\eps}  \partial_\tau F^+ + \overline b^i e^{-i\tau}G^i + b^i e^{i\tau}\overline G^i, \\
\ds
\partial_t F^- +p\cdot \nabla_x F^- - \nabla_x (U-E)\cdot \nabla_p F^-= -\frac{{\cal E}^{-}}{\eps}
 \partial_\tau F^- -\overline b^i e^{-i\tau}G^i - b^i e^{i\tau}\overline G^i, \\ \ds
\partial_t G^i +p\cdot \nabla_x G^i + \nabla_x U\cdot \nabla_pG^i=- \frac{2E}{\eps}\partial_\tau G^i +
 b^i e^{i\tau}(F^- - F^+)  + (b^+-b^-)G^i,
\end{array}
\eeq
where ${\cal E}^{\pm}= 2E- \nabla_x(2U\pm E)\cdot \nabla_pS$. This system needs initial data $F(0,x,p,\tau)$ and $G^i(0,x,p,\tau)$, which will be determined in a such way  that the  corresponding solution  is smooth with respect to $\eps$.  We proceed as in Section \ref{syst_initcond} for the $2 \times 2$ model. Let
$$F_0^{\pm}= \Pi F^{\pm}, \quad F_1^{\pm}= ({\cal I}- \Pi) F^{\pm},\quad G_0^i= \Pi G^i, \quad G_1^i= (I- \Pi) G^i.$$
We have
\beq
\begin{array}{l}
\label{Cet}
 \ds G^i=G^i_0-i \frac{\eps}{2E}b^i e^{i\tau}(F^-_0 - F^+_0)   + O(\eps^2), \\
  F^+=F^+_0+i \frac{\eps}{\cal E^+}\left(\overline b^i e^{-i\tau}G^i_0 -b^i e^{i\tau}\overline G^i _0  \right) + O(\eps^2), \\
\ds F^-=F^-_0-i \frac{\eps}{\cal E^-}\left(\overline b^i e^{-i\tau}G^i_0 - b^i e^{i\tau}\overline G^i _0  \right) + O(\eps^2).
\end{array}
\eeq
To fit with the initial data $(f_{in}^+,f_{in}^-, f_{in}^i )$, we set
$$ G^i_0-i \frac{\eps}{2E}b^i (F^-_0 - F^+_0) =f_{in}^i , $$
$$ F^+_0+i \frac{\eps}{\cal E^+}\left(\overline b^i G^i_0 - b^i \overline G^i _0  \right) = f_{in}^+, $$
$$ F^-_0-i \frac{\eps}{\cal E^-}\left(\overline b^i G^i_0 - b^i \overline G^i _0  \right)= f_{in}^-.$$
This gives

$$ F^+_0= f_{in}^+- i \frac{\eps}{\cal E^+}\left(\overline b^i  f_{in}^i  - b^i \overline f_{in}^i   \right), $$
$$ F^-_0=f_{in}^-+ i \frac{\eps}{\cal E^-}\left(\overline b^i  f_{in}^i  - b^i \overline f_{in}^i   \right), $$
$$ G^i_0=f_{in}^i- i \frac{\eps}{2E}b^i (f_{in}^+ - f_{in}^-).$$
Reporting these expressions  in (\ref{Cet}) yields
\beq
\label{CE0}
\begin{array}{l}
 F^+(0,x,p, \tau)= f_{in}^+-i \frac{\eps}{\cal E^+}\left(\overline b^i  f_{in}^i \left(1-e^{-i\tau} \right) - b^i \overline f_{in}^i \left(1-e^{-i\tau }\right)   \right), \\
 F^-(0,x,p,\tau)=f_{in}^-+ i \frac{\eps}{\cal E^-}\left(\overline b^i  f_{in}^i \left(1-e^{-i\tau }\right) - b^i \overline f_{in}^i \left(1-e^{-i\tau}\right)   \right), \\
 G^i(0,x,p,\tau)=f_{in}^i+ i \frac{\eps}{2E}b^i  \left(e^{i\tau}-1\right) (f_{in}^+ - f_{in}^-).
\end{array}
\eeq

\subsection{Numerical results}
We consider the following initial conditions for \eqref{modelBC-3} with $x, v\in [-2\pi, 2\pi]$
\begin{eqnarray*}
f_+(t=0, x, p)&=& f_-(t=0, x, p) = (1+0.5\cos(x))\frac{e^{-p^2/2}}{\sqrt{2\pi}},  \nonumber\\
f_i(t=0, x, p)&=& \Big[(1+0.5\sin(x)) + i (1+0.5\cos(x))\Big] \frac{e^{-p^2/2}}{\sqrt{2\pi}},
\end{eqnarray*}
and the following expression for $E$, $b_i$ and $b_\pm$
$$
E(x)=1-\cos(x/2)+\varepsilon, \;\;\; b_i(x, p) = -\frac{1}{2}\sin(p+1), \;\;\; b_\pm=0.
$$
Notice that with this choice of $E$, the narrowest band gap $2E=2\eps$ which describes the so-called
"avoided-crossing" case (see \cite{jin-morandi}).
We will compare a direct simulation of the model \eqref{modelBC-3} (using time splitting and
pseudo-spectral methods in space)
with our new approach \eqref{modelBC-FGi} (using time splitting,
pseudo-spectral methods in space also and the well-prepared initial condition \eqref{CE0}).
Moreover, periodic boundary conditions are considered in both $x$ and $p$.

In the sequel, we detail the steps of the two methods (direct and new).
First, we introduce the following notations:
${\cal A}=(- \nabla_x (U+E), - \nabla_x (U-E), \nabla_x U, \nabla_x U)$,
and ${\cal E}=(-{\cal E}^+, -{\cal E}^-, -2E/\eps, -2E/\eps)$
whereas the matrix $B_\tau$ is given by
$$
B_\tau =
\left(
\begin{array}{llllcccc}
0 &0 & 2 b^i\cos\tau & 2 b^i\sin\tau\\
0 &0 & -2 b^i\cos\tau & -2 b^i\sin\tau\\
-b^i\cos\tau & b^i\cos\tau & 0 & 0 \\
-b^i\sin\tau & b^i\sin\tau & 0 & 0
\end{array}
\right), \;\;
$$
and $B$ by
$$
B =
\left(
\begin{array}{llllcccc}
0 &0 & 2 b^i & 0\\
0 &0 & -2 b^i & 0\\
-b^i & b^i & 0 & 2E/\eps \\
0 & 0 & -2E/\eps &0
\end{array}
\right).
$$
Then, the direct numerical scheme for \eqref{modelBC-3} writes  (with $f=(f^+, f^-$, Re$(f^i)$, Im$(f^i))\in \R^4$)
\begin{itemize}
\item solve $\partial_t f + p\partial_x f = 0$ with spectral method in space and exact integration in time,
\item solve $\partial_t f + {\cal A}\partial_p f = 0$ with spectral method in space and exact integration in time,
\item solve $\partial_t f = Bf$ (with $B$ a $4$x$4$ matrix given above) exactly in time.
\end{itemize}
The numerical  scheme for \eqref{modelBC-FGi} is (with $F=(F^+, F^-$, Re$(G^i)$, Im$(G^i))\in \R^4$)
\begin{itemize}
\item solve $\partial_t F+ p\partial_x F = 0$ with spectral method in space and exact integration in time,
\item solve $\partial_t F + {\cal A}\partial_p F = 0$ with spectral method in space and exact integration in time,
\item solve $\partial_t F = B_\tau F$ (with $B_\tau$ a $4$x$4$ matrix given above) exactly in time,
\item solve $\partial_t F = \frac{1}{\varepsilon}{\cal E}\partial_\tau F$ with a pseudo-spectral method in $\tau$ and 
an implicit Euler scheme in time (exact time integration in the Fourier space can also be done).
\end{itemize}
The equation \eqref{eq_S_bc} on $S$ is solved using a time splitting method (between transport and right hand side)
and spectral methods are used in $(x,p)$.

In Figure \ref{fig1_eps1}, we plot the space dependence of the solution
$f(t_f=2, x, p=0)$ and of the densities $\rho(t_f=2, x)=\int_{\mathbb{R}} f(t_f=2, x, p)dp$,
for $\eps=1$ for the direct and the new methods.
The reference solution uses $\Delta t=0.05$, $N_x=256$, $N_p=64$ whereas for the new method,
we choose $\Delta t=0.05$, $N_x=32$, $N_p=64$ and $N_\tau=8$.
First, we observe that the new method captures well the solution for both diagnostics.
Second, the CPU time is about $15$ s for
the reference method whereas for the new method, it is about  $1$ min.

In Figures \ref{fig1_eps1s32} and \ref{fig2_eps1s32}, we consider the same diagnostics as before,
but with $\eps=1/32$. The reference solution uses now $\Delta t=0.02$, $N_x=512$, $N_p=64$
whereas we still choose $\Delta t=0.05$, $N_x=32$, $N_p=64$ and $N_\tau=8$
for the new method. Then, the CPU time for the reference method is now $75$ s
and is still $1$ min for the new method. Even for this value of $\eps$, the solution is
highly oscillatory (the $f^i$ part in particular) and the new method behaves very well
even its mesh is coarser than the spatial oscillations.

Finally,  in Figures \ref{fig1_eps1s256} and \ref{fig2_eps1s256}, we consider $\eps=1/256$ and $t_f=0.2$.
The numerical parameters for the reference method have been chosen to resolve the space-time oscillations
($\Delta t=0.0005$, $N_x=4096$, $N_p=64$) so that the CPU time is $1420$ s.
The numerical parameters of the new method are still fixed (so as  its CPU time).
The same conclusions as before arise.

\newpage

\begin{figure}
\vspace{-4cm}
\begin{center}
\begin{tabular}{ccllll}
\includegraphics[width=0.5\linewidth]{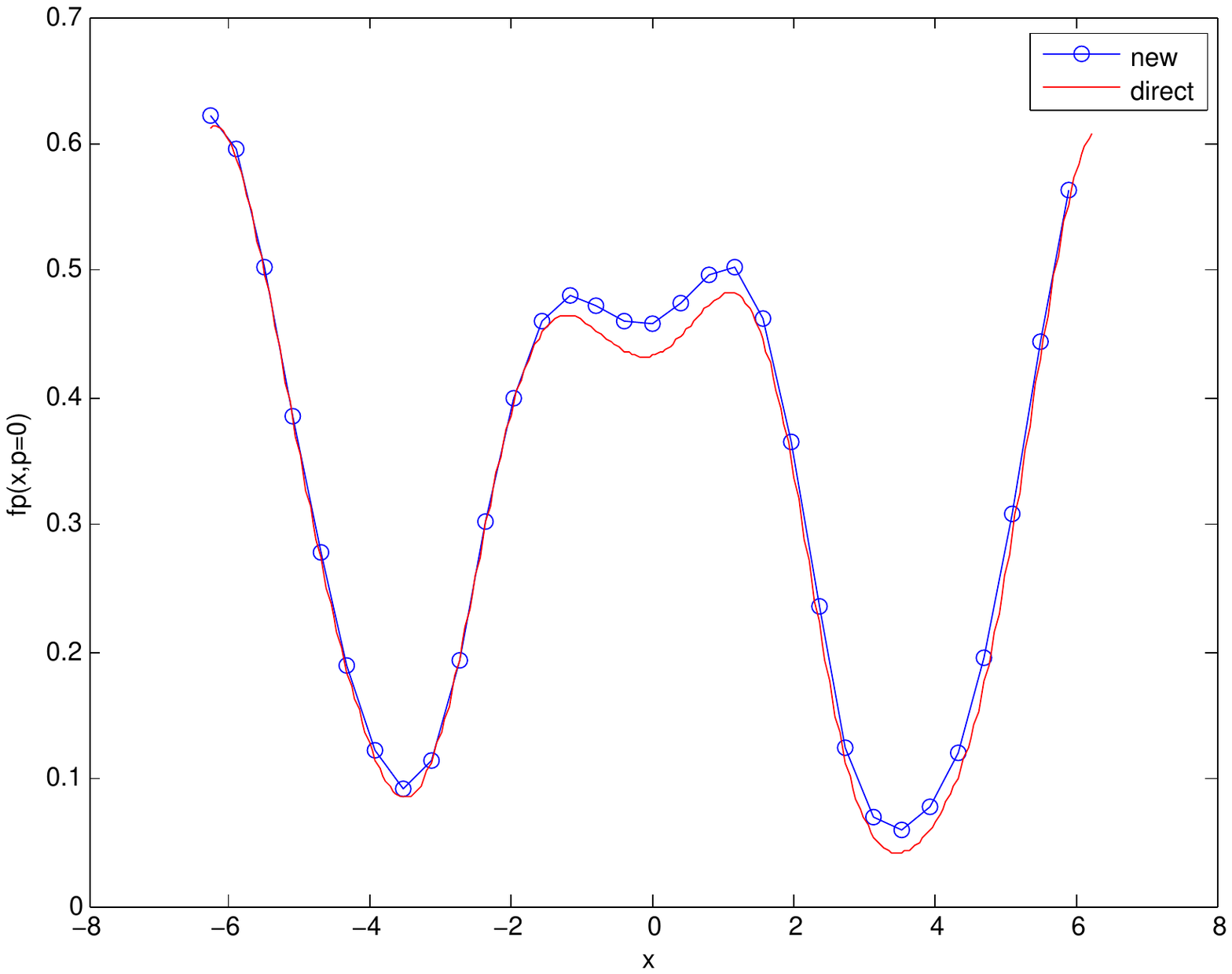}&
\hspace{-1.2cm}\vspace{-5cm}
\includegraphics[width=0.5\linewidth]{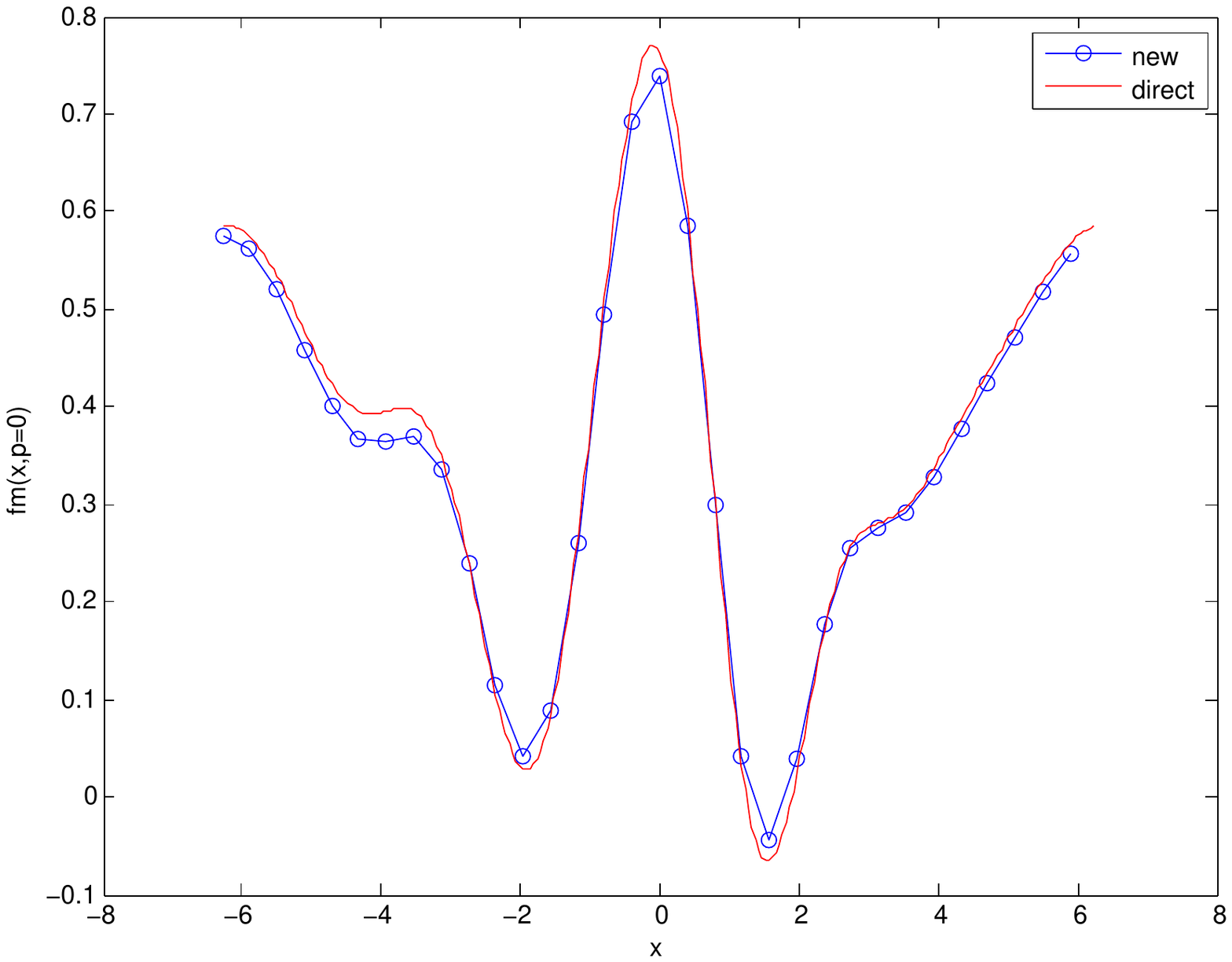}\\
\includegraphics[width=0.5\linewidth]{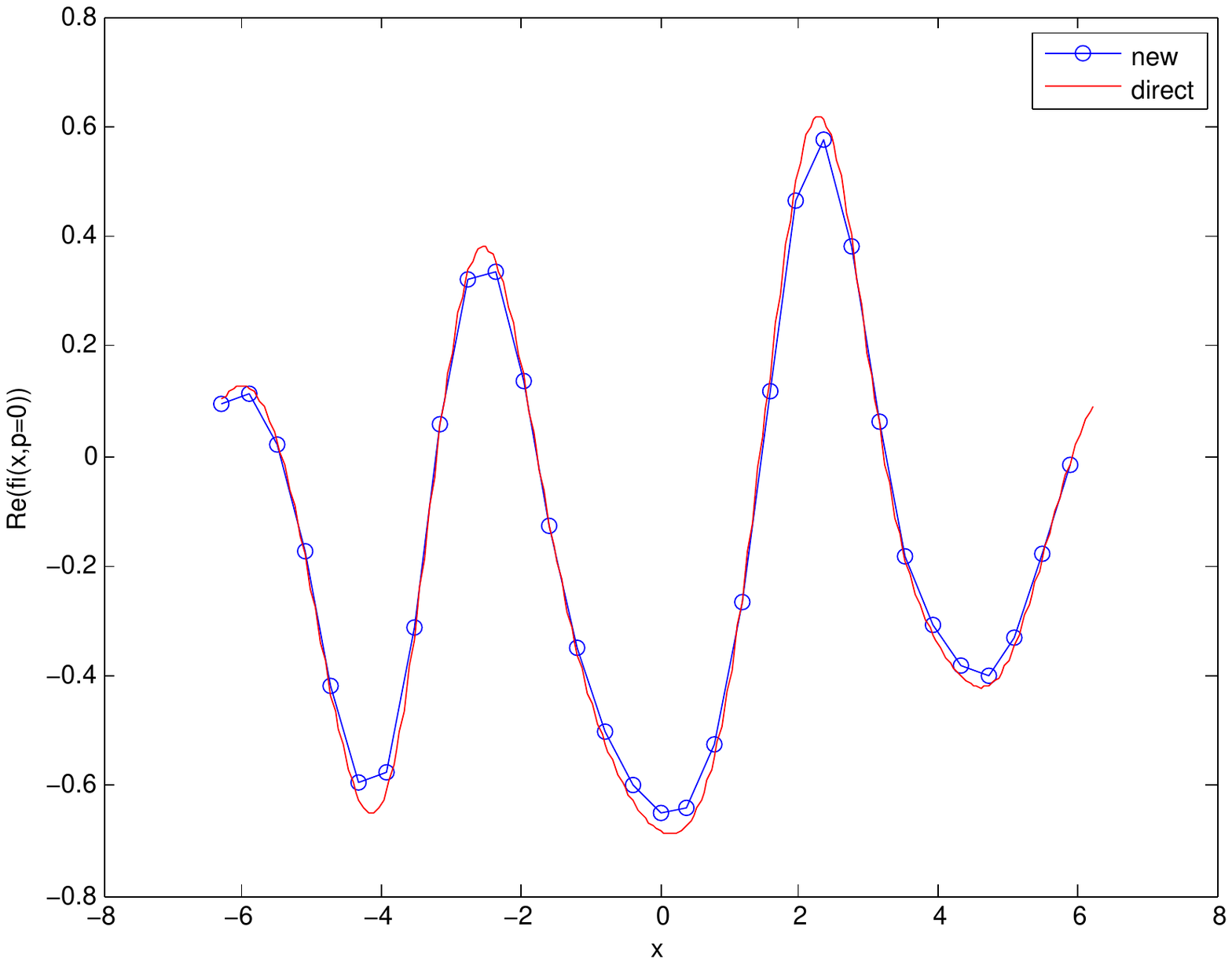}&
\hspace{-1.2cm}\vspace{-5cm}
\includegraphics[width=0.5\linewidth]{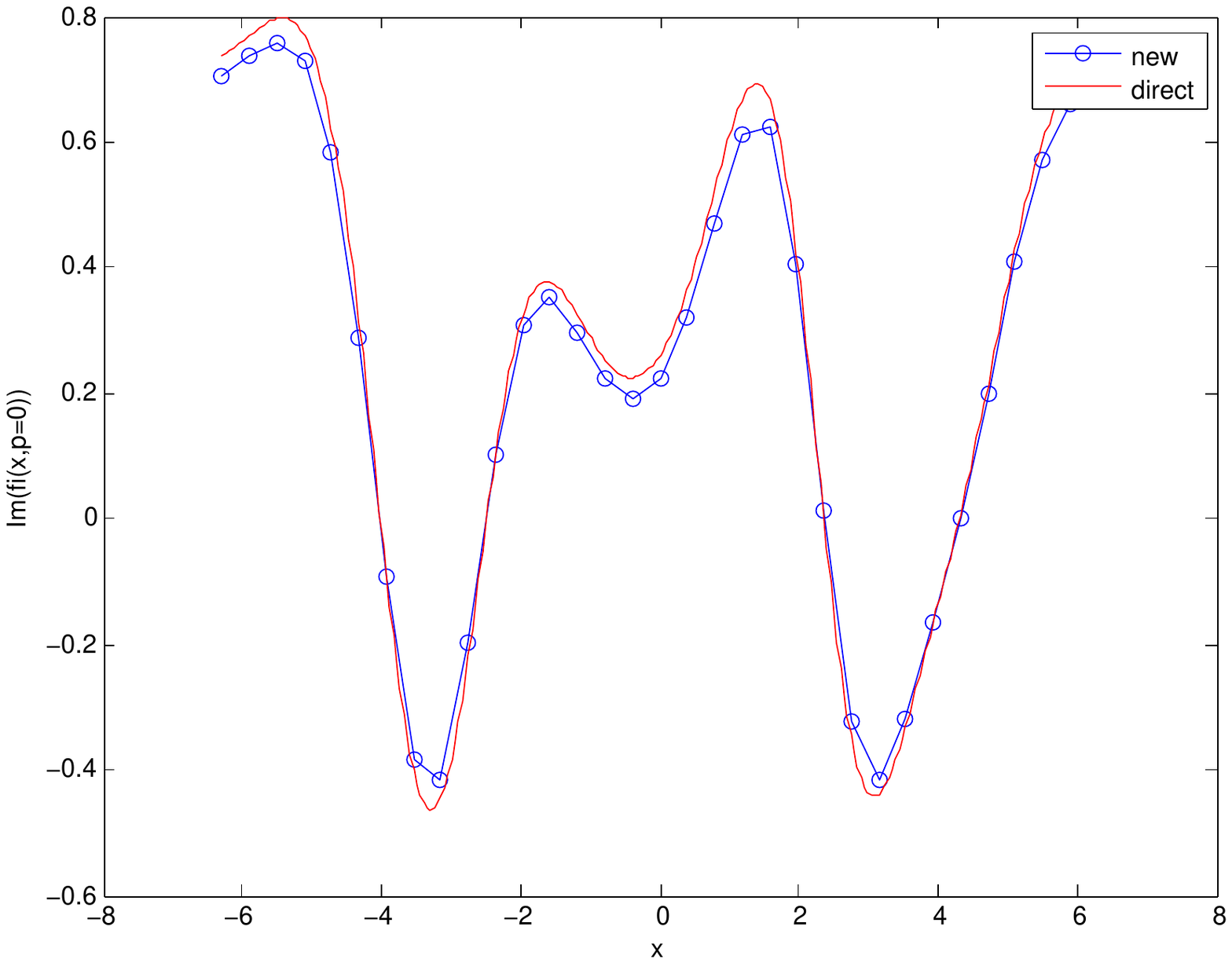}\\
\includegraphics[width=0.5\linewidth]{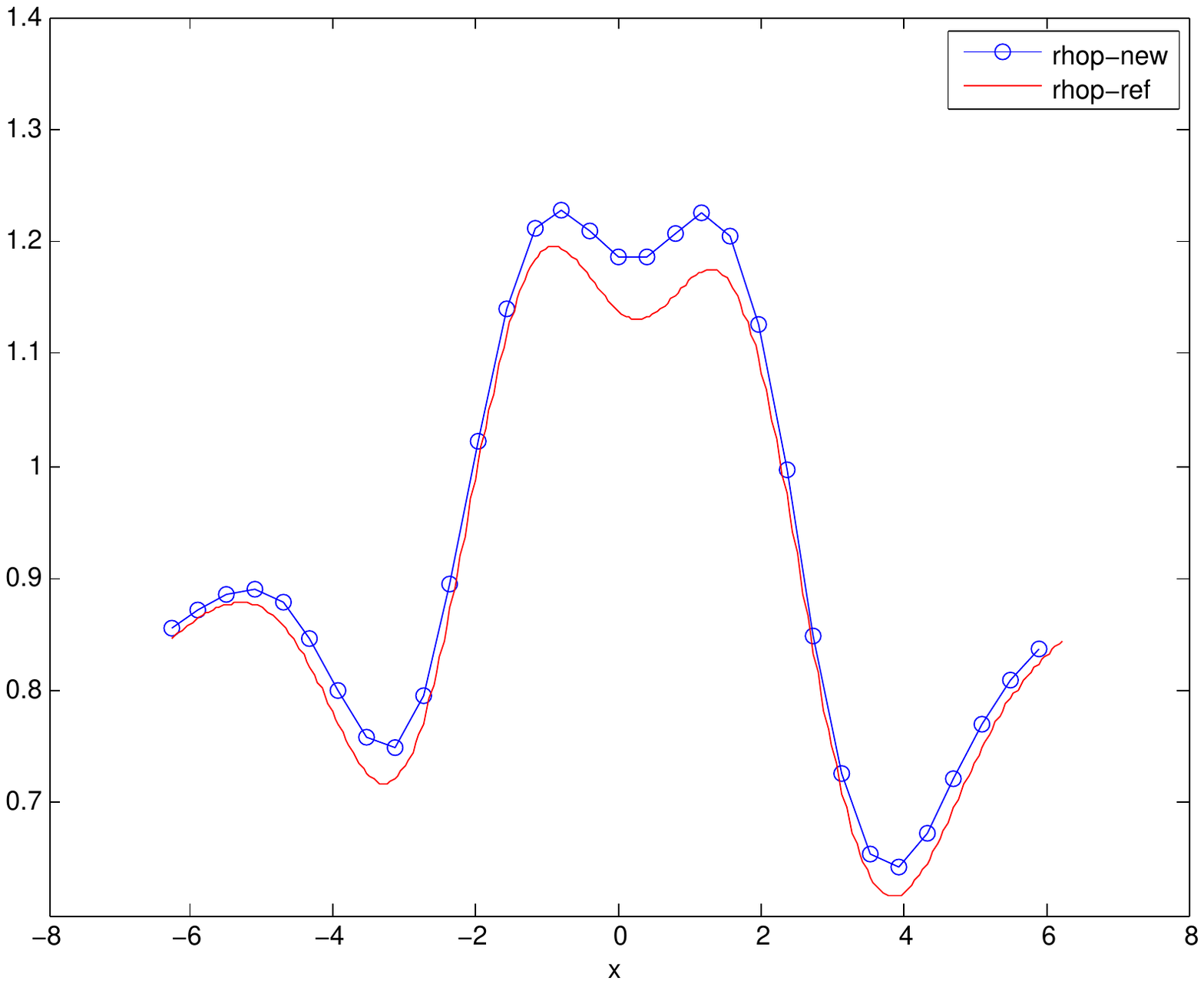}&
\hspace{-1.2cm}\vspace{-5cm}
\includegraphics[width=0.5\linewidth]{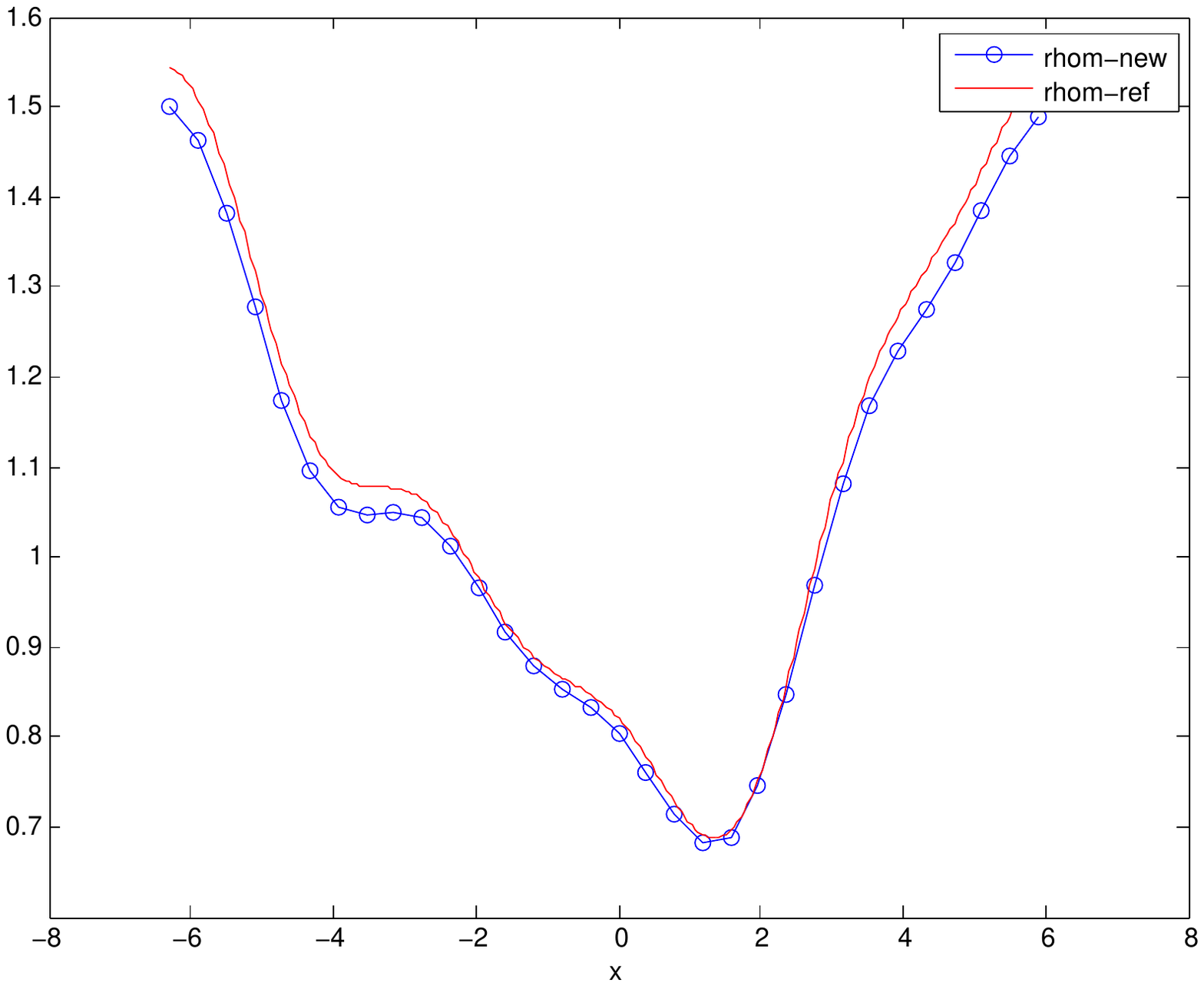}\\
\includegraphics[width=0.5\linewidth]{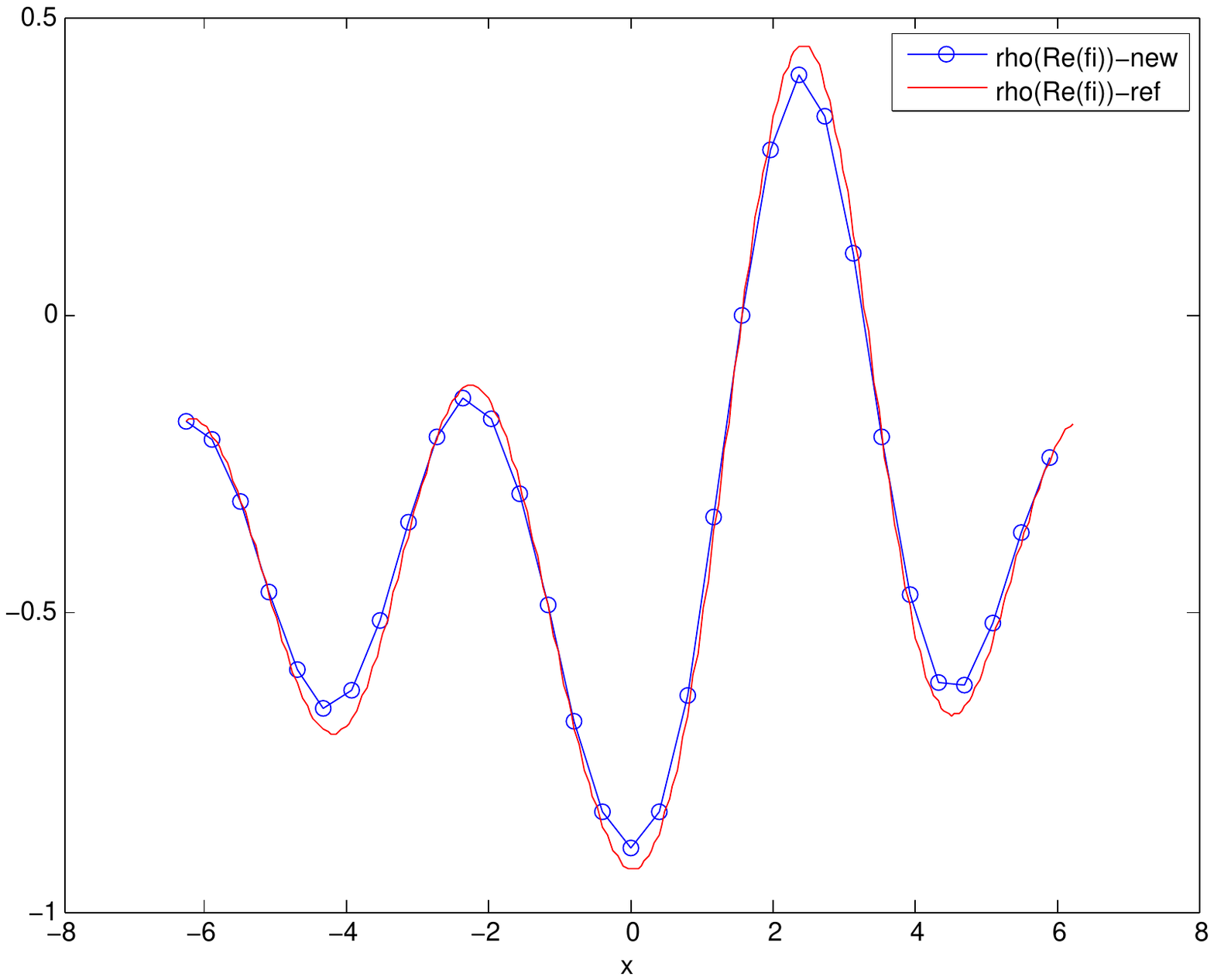}&
\hspace{-1.2cm}\vspace{-3cm}
\includegraphics[width=0.5\linewidth]{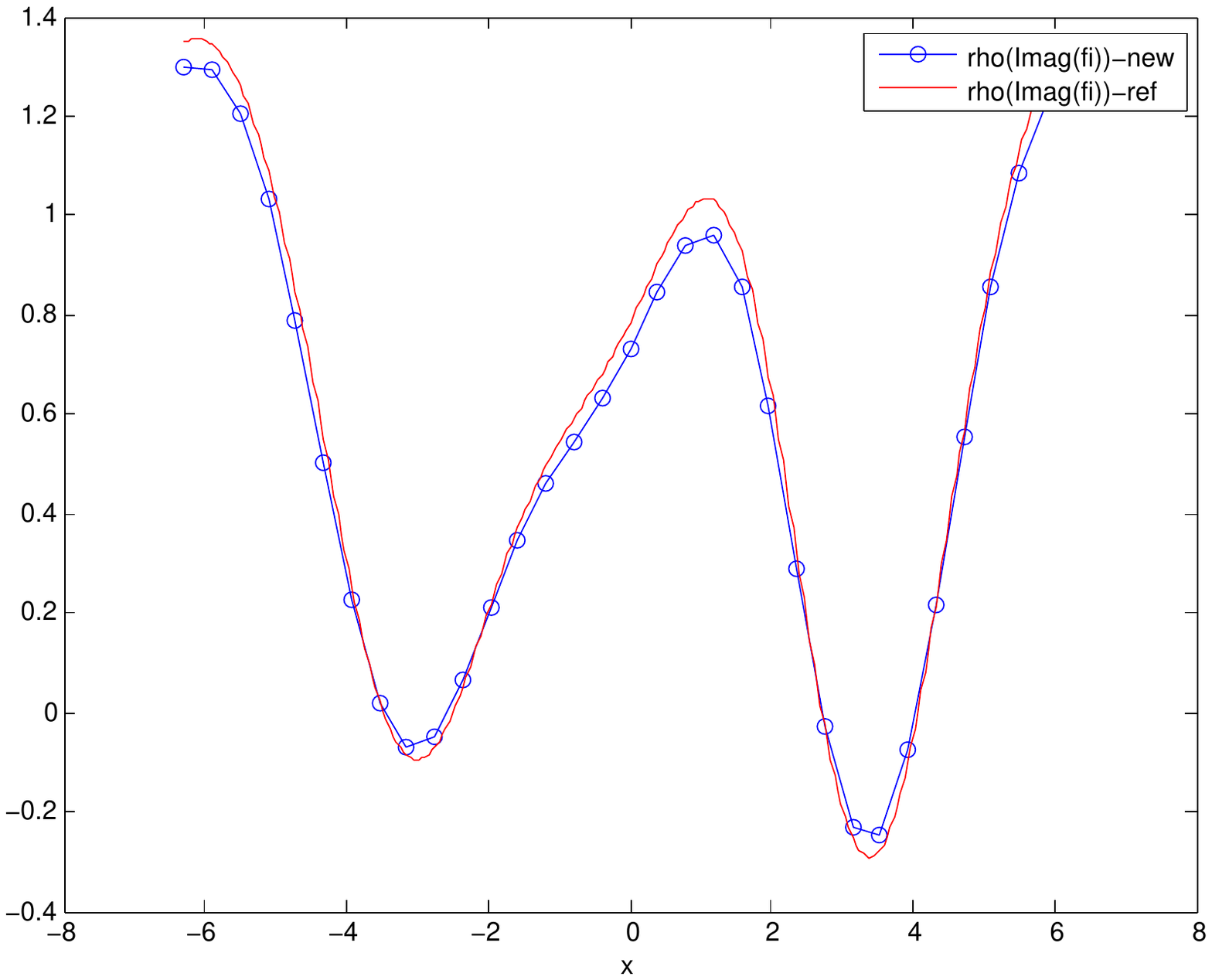}
\end{tabular}
\end{center}
\caption{$\varepsilon=1$. From top left to bottom right: space dependence of
$f^+$, $f^-$,  Re($f^i$) and Im$(f^i)$ at $p=0$, and space dependence of the densities
$\rho^+$, $\rho^-$, Re($\rho^i$) and Im$(\rho^i)$. }
\label{fig1_eps1}
\end{figure}

\newpage

\begin{figure}
\vspace{-2cm}
\begin{center}
\begin{tabular}{ccll}
\includegraphics[width=0.4\linewidth]{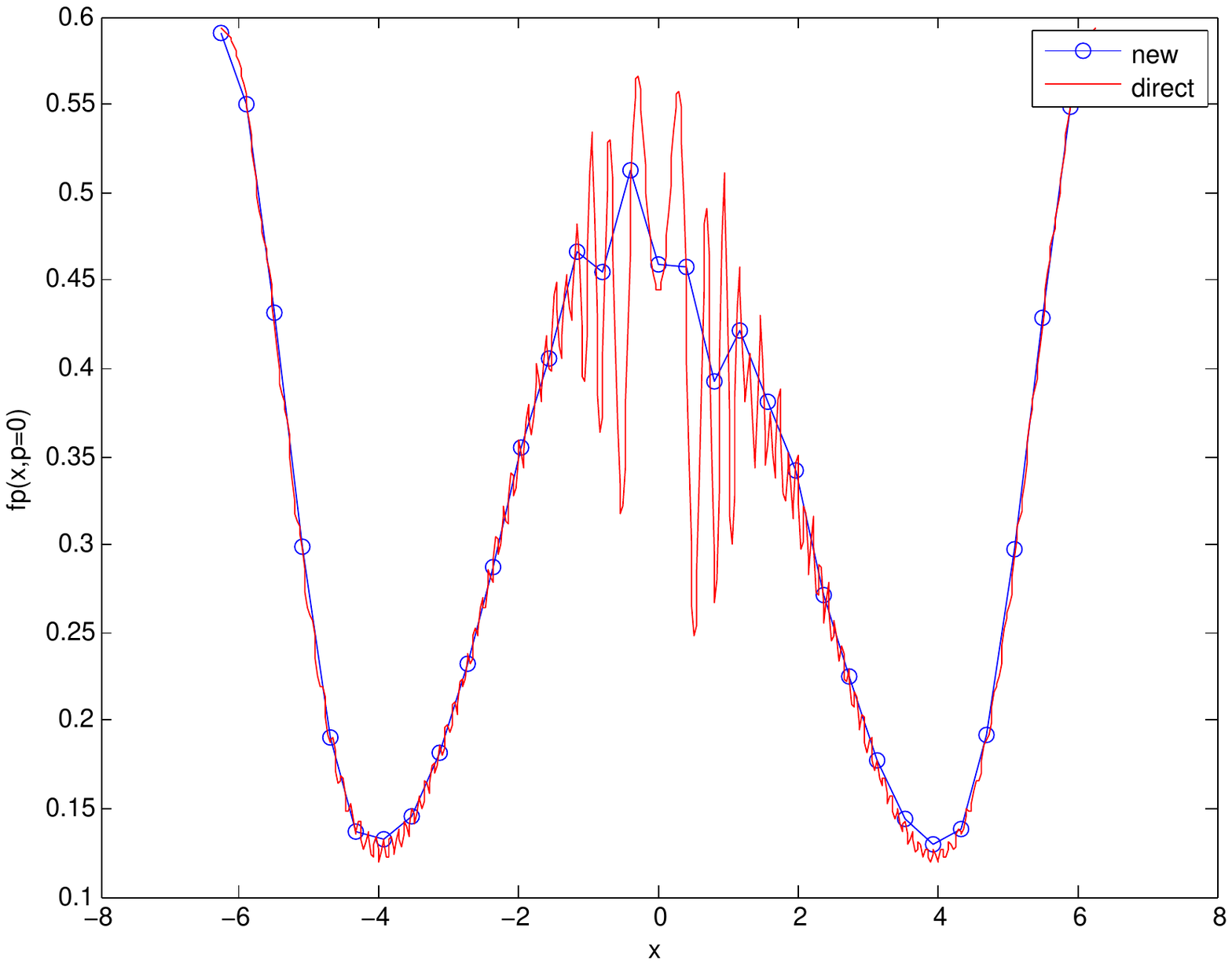}&
\hspace{-1.2cm}
\vspace{-3.4cm}
\includegraphics[width=0.4\linewidth]{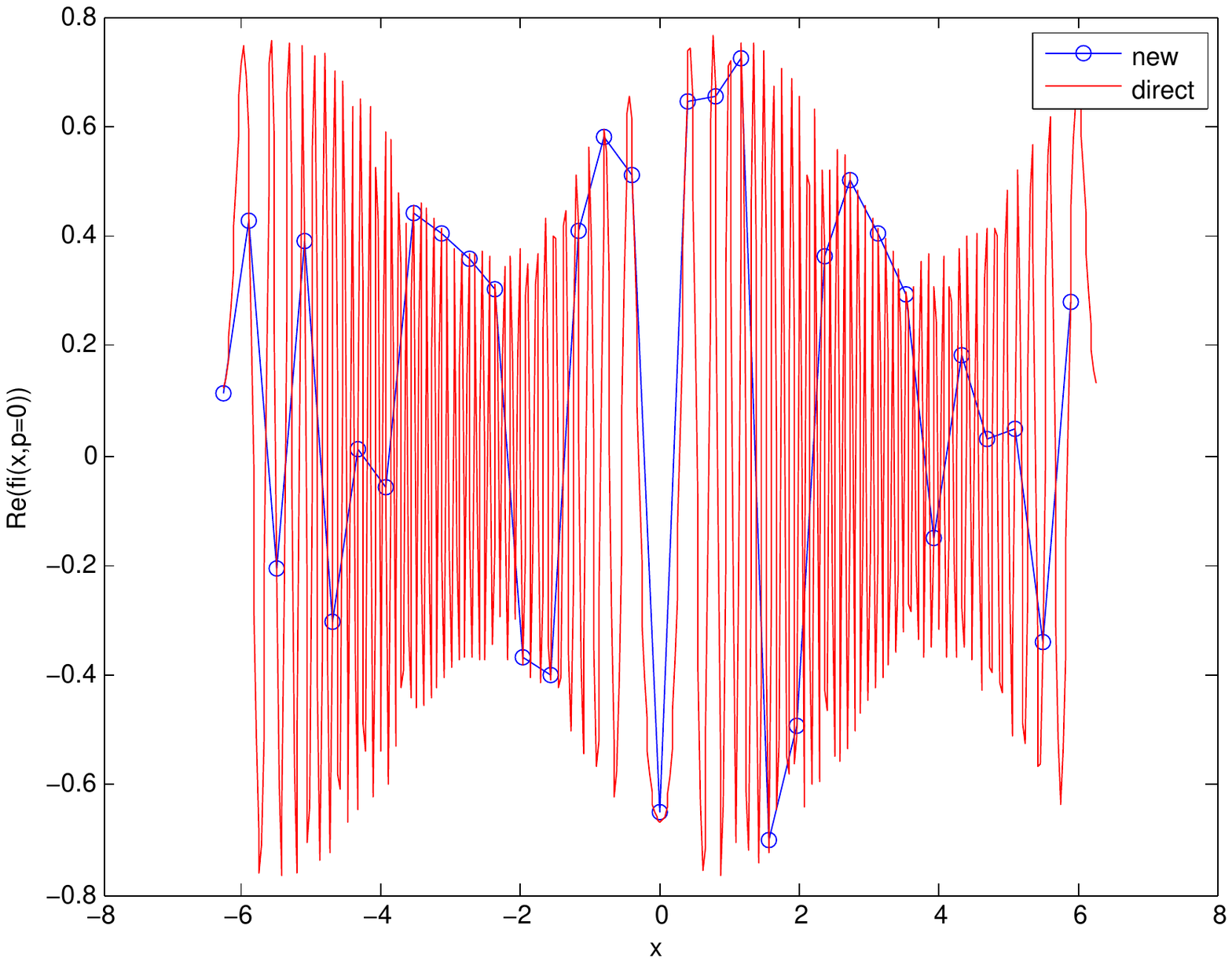}&
\hspace{-1.2cm}
\includegraphics[width=0.4\linewidth]{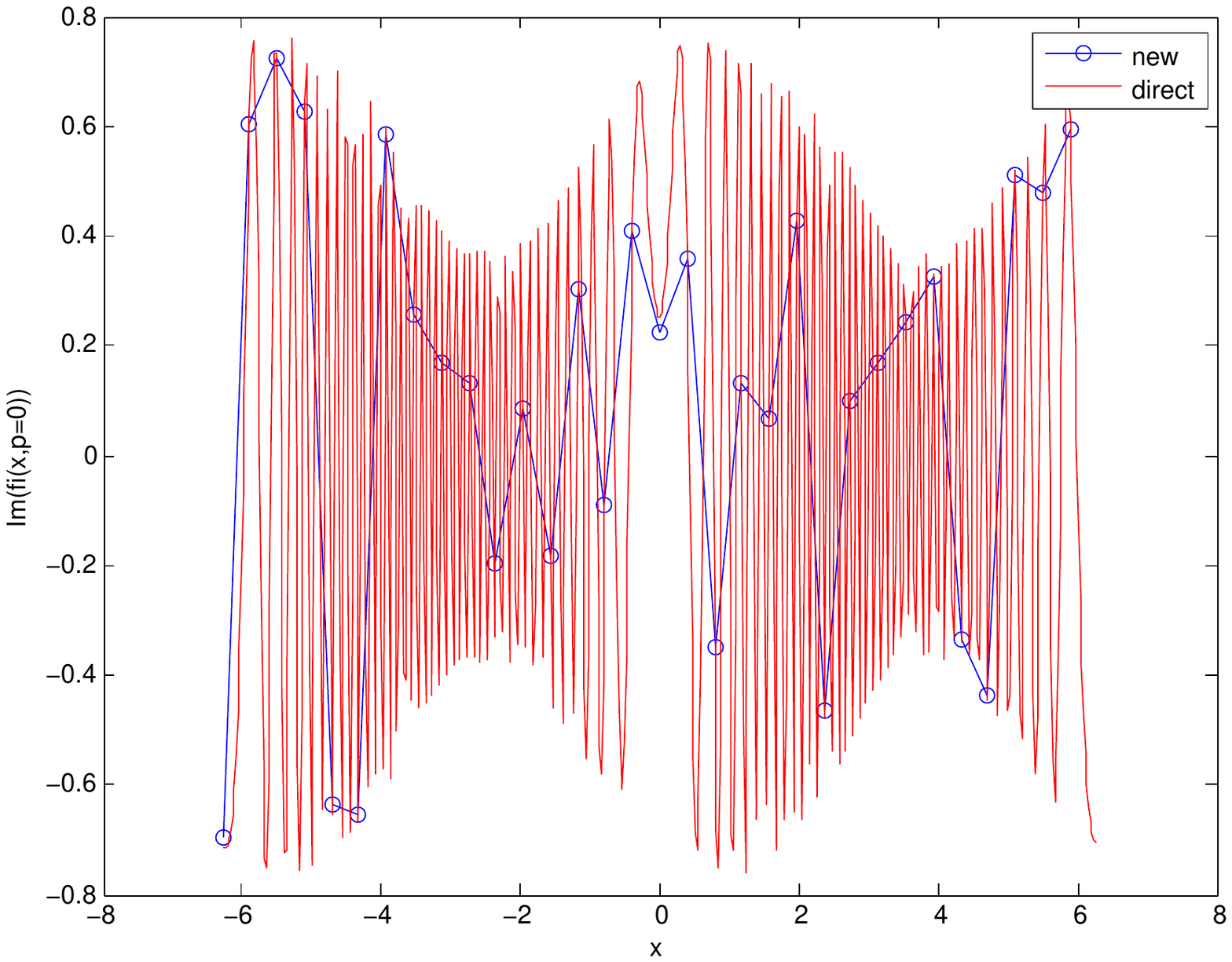}\\
\includegraphics[width=0.4\linewidth]{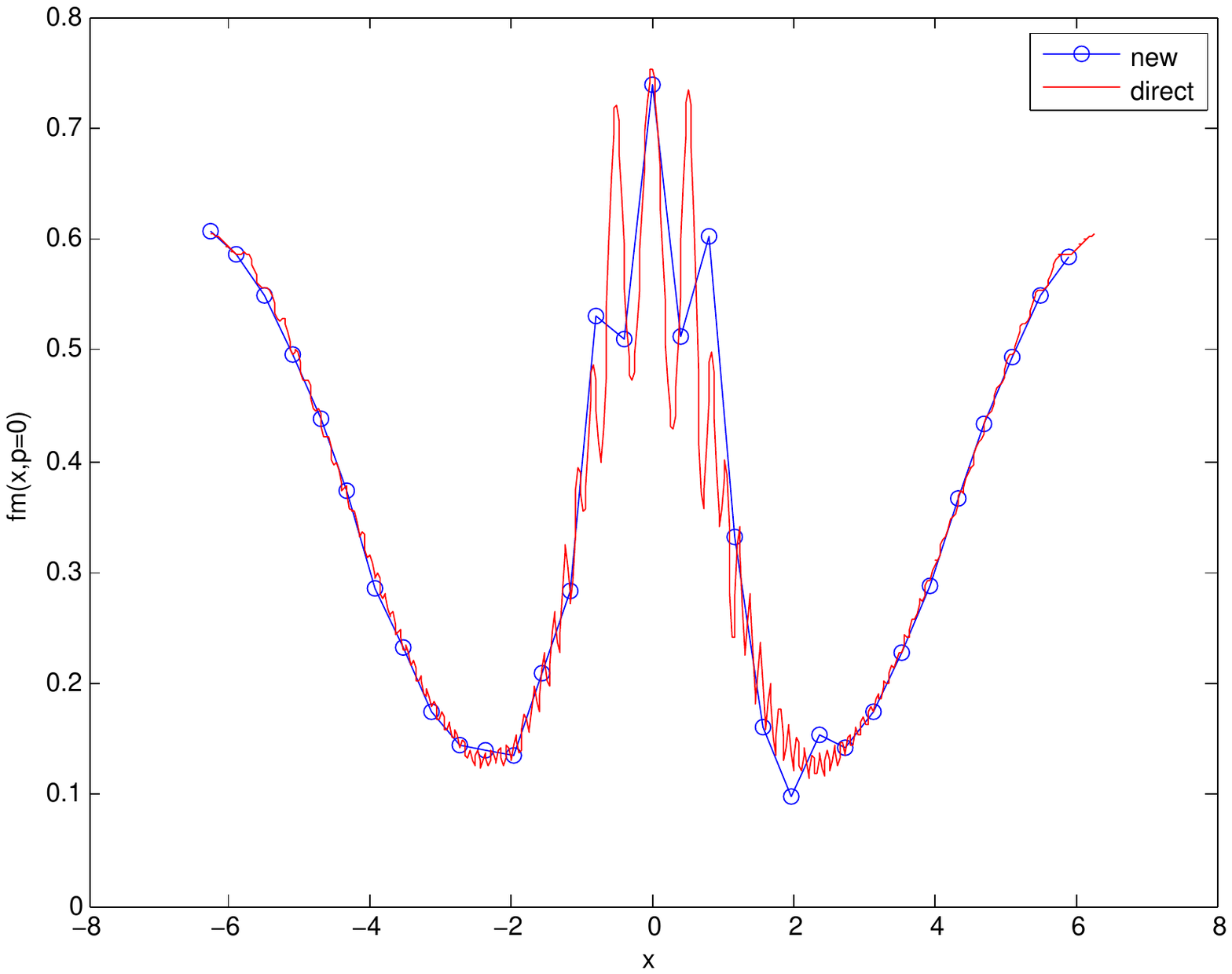}&
\hspace{-1.2cm}
\includegraphics[width=0.4\linewidth]{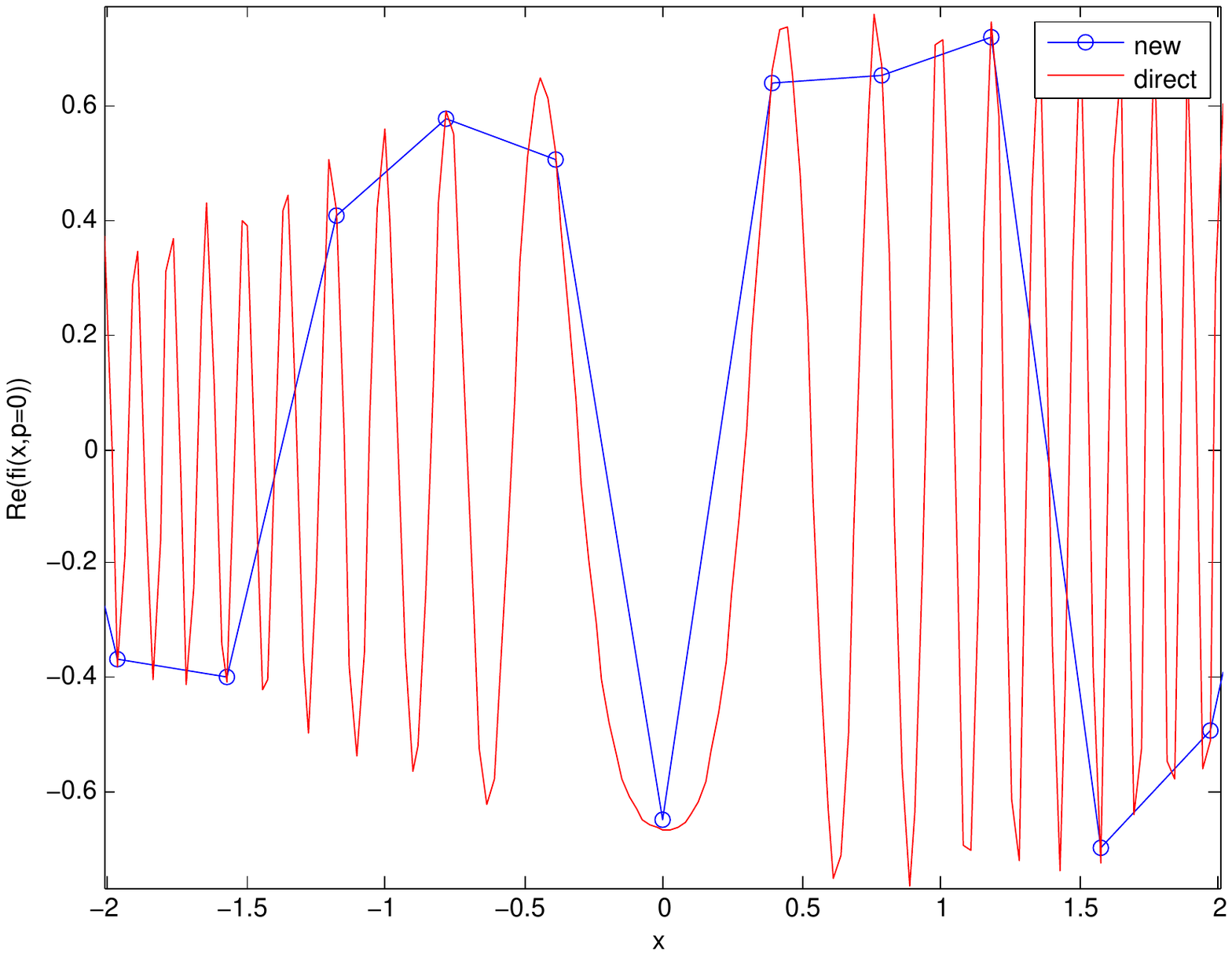}&
\hspace{-1.2cm}
\vspace{-2.3cm}
\includegraphics[width=0.4\linewidth]{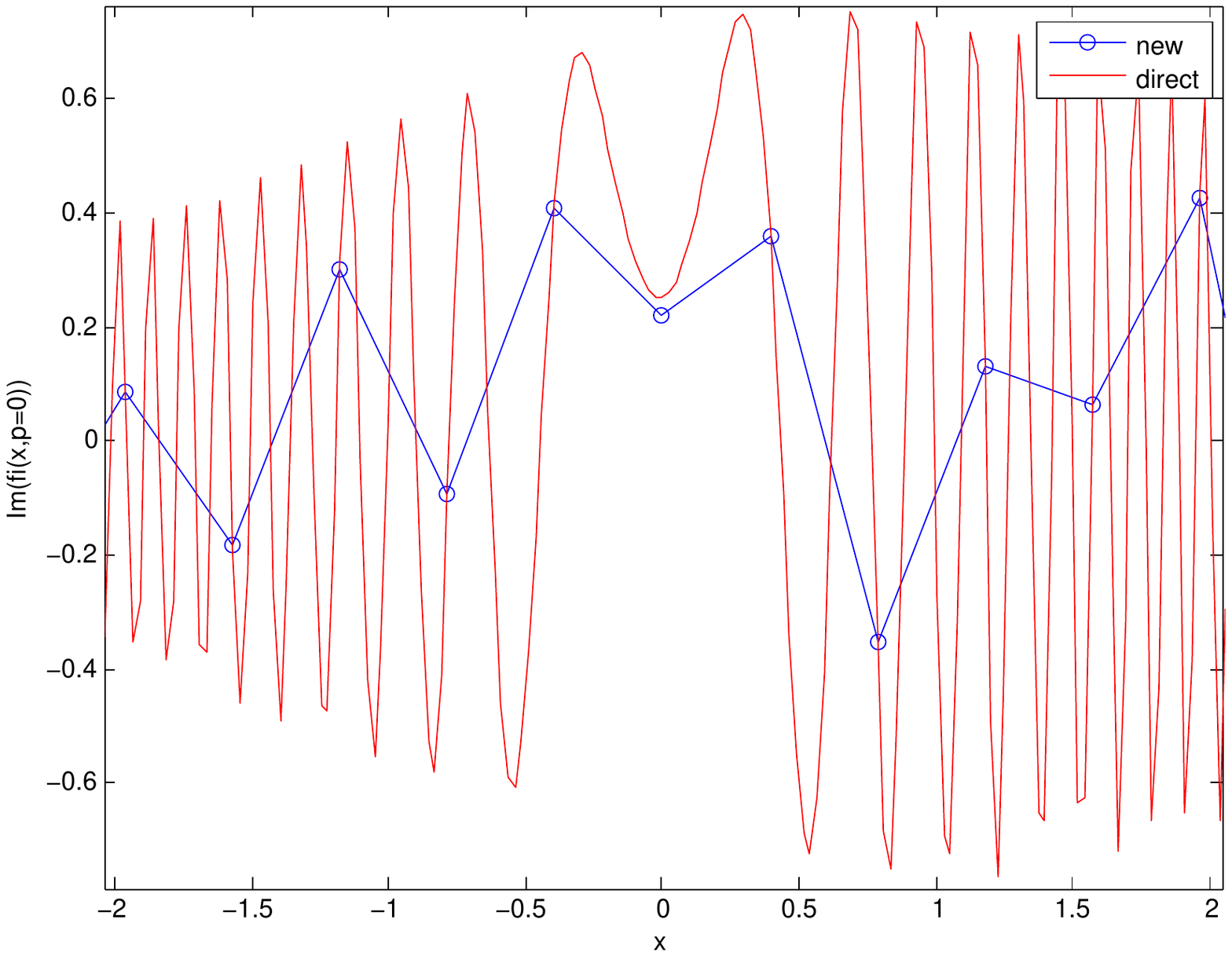}
\end{tabular}
\end{center}
\caption{$\varepsilon=1/32$. First line: space dependence of $f^+$, Re($f^i$) and Im$(f^i)$ at $p=0$.
Second line: space dependence of $f^-$, Re($f^i$) (zoom) and Im$(f^i)$ (zoom) at $p=0$.  }
\label{fig1_eps1s32}
\end{figure}

\begin{figure}
\vspace{-4cm}
\begin{center}
\begin{tabular}{ccll}
\includegraphics[width=0.4\linewidth]{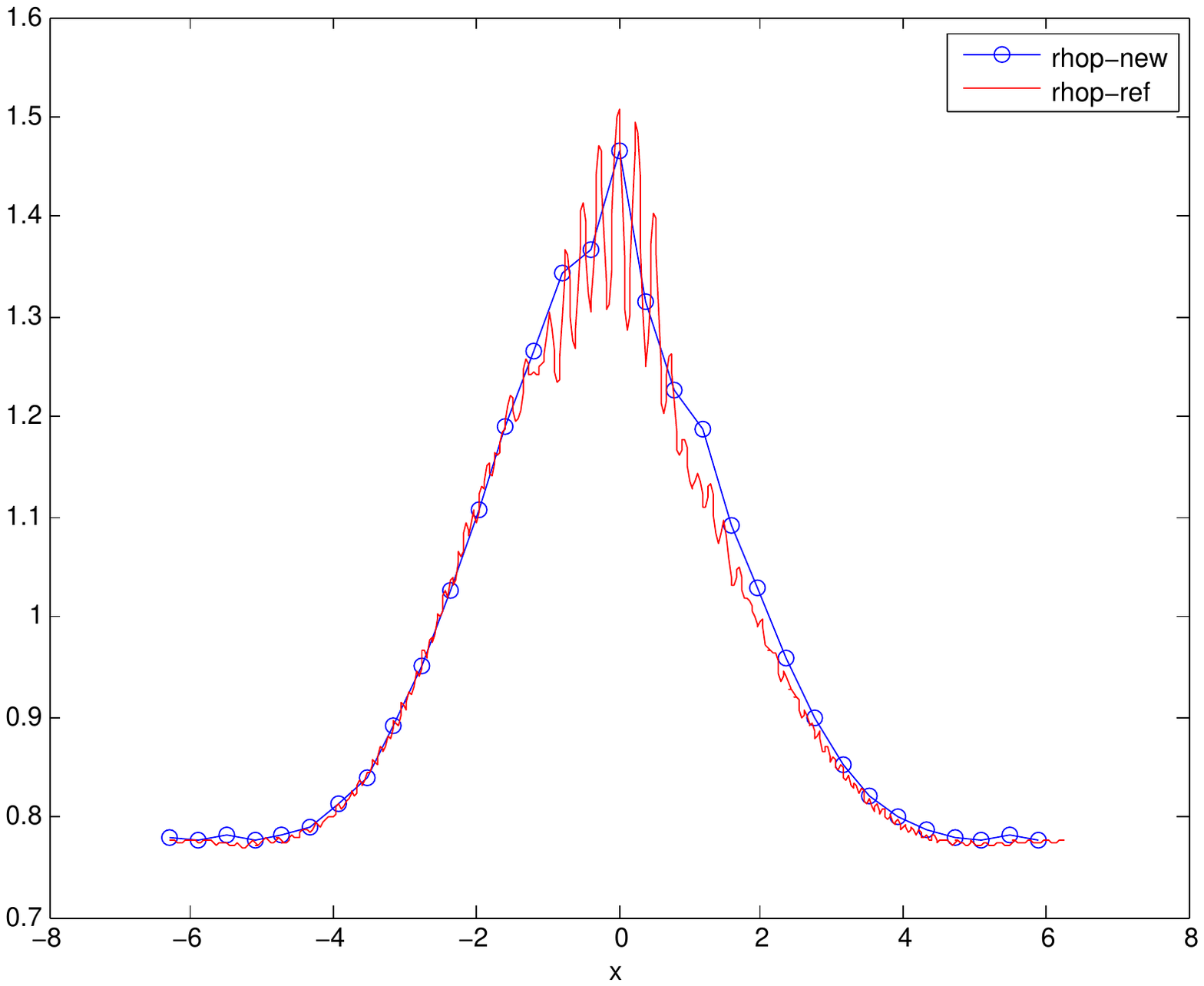}&
\hspace{-1.2cm}
\vspace{-3.4cm}
\includegraphics[width=0.4\linewidth]{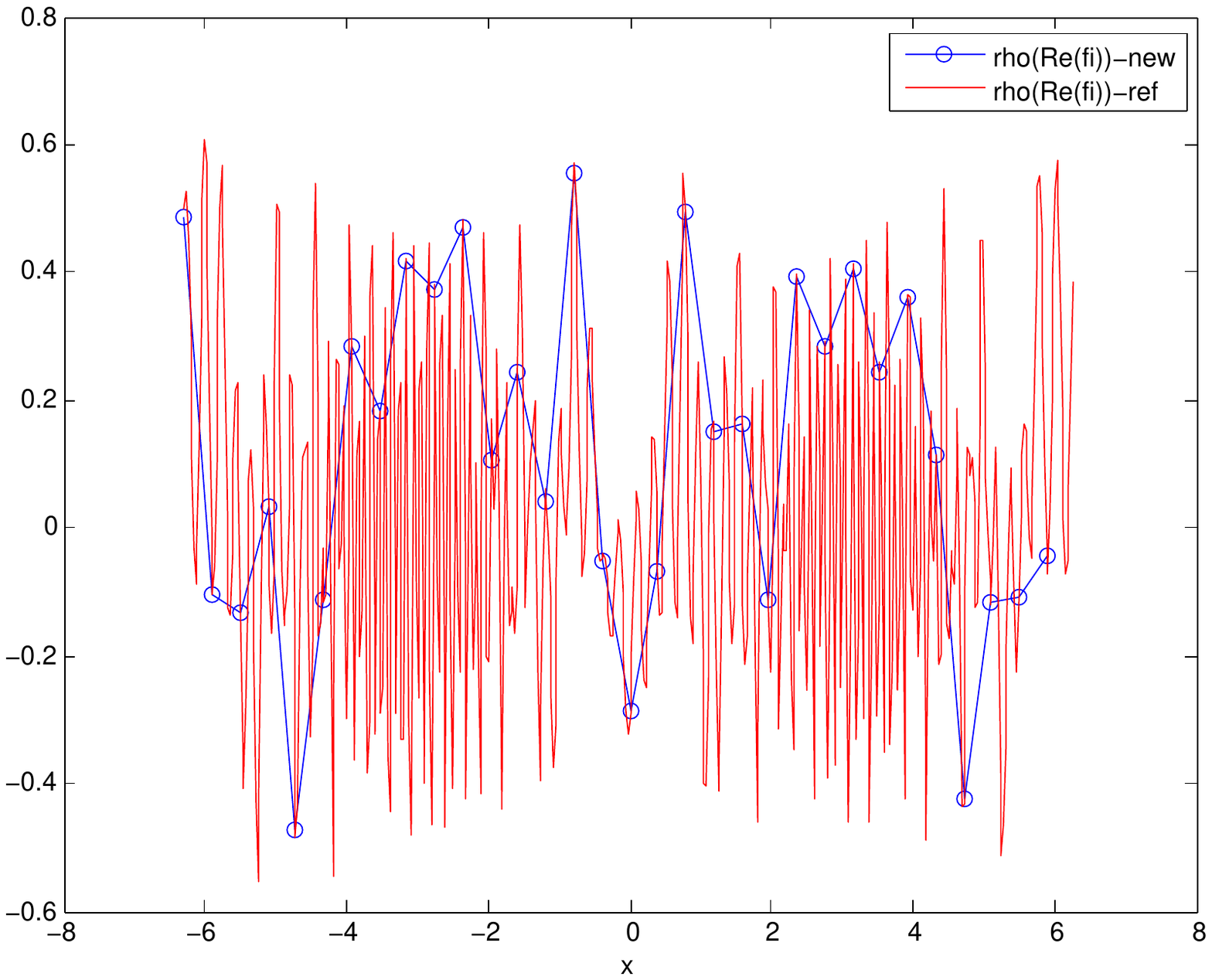}&
\hspace{-1.2cm}
\includegraphics[width=0.4\linewidth]{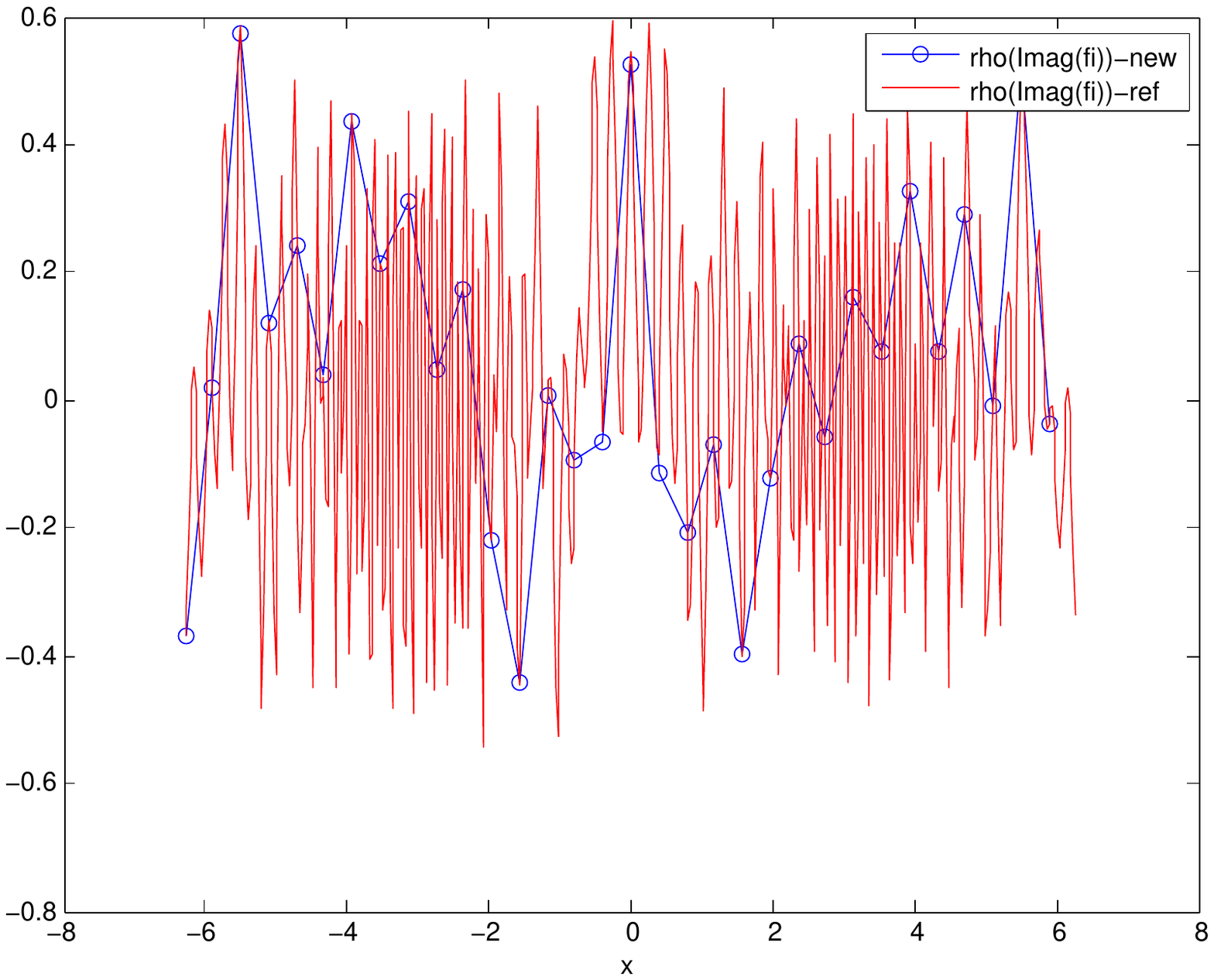}\\
\includegraphics[width=0.4\linewidth]{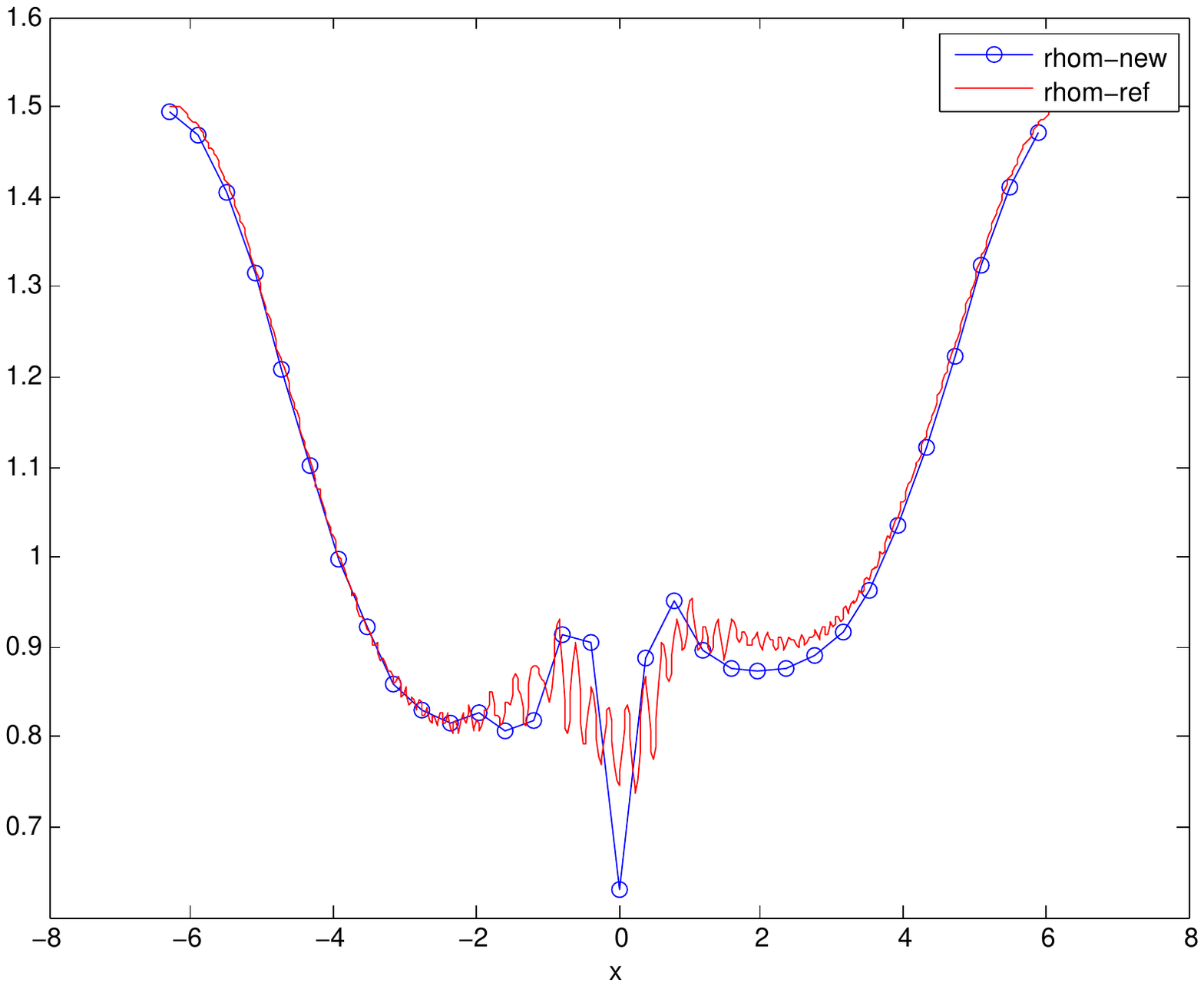}&
\hspace{-1.2cm}
\includegraphics[width=0.4\linewidth]{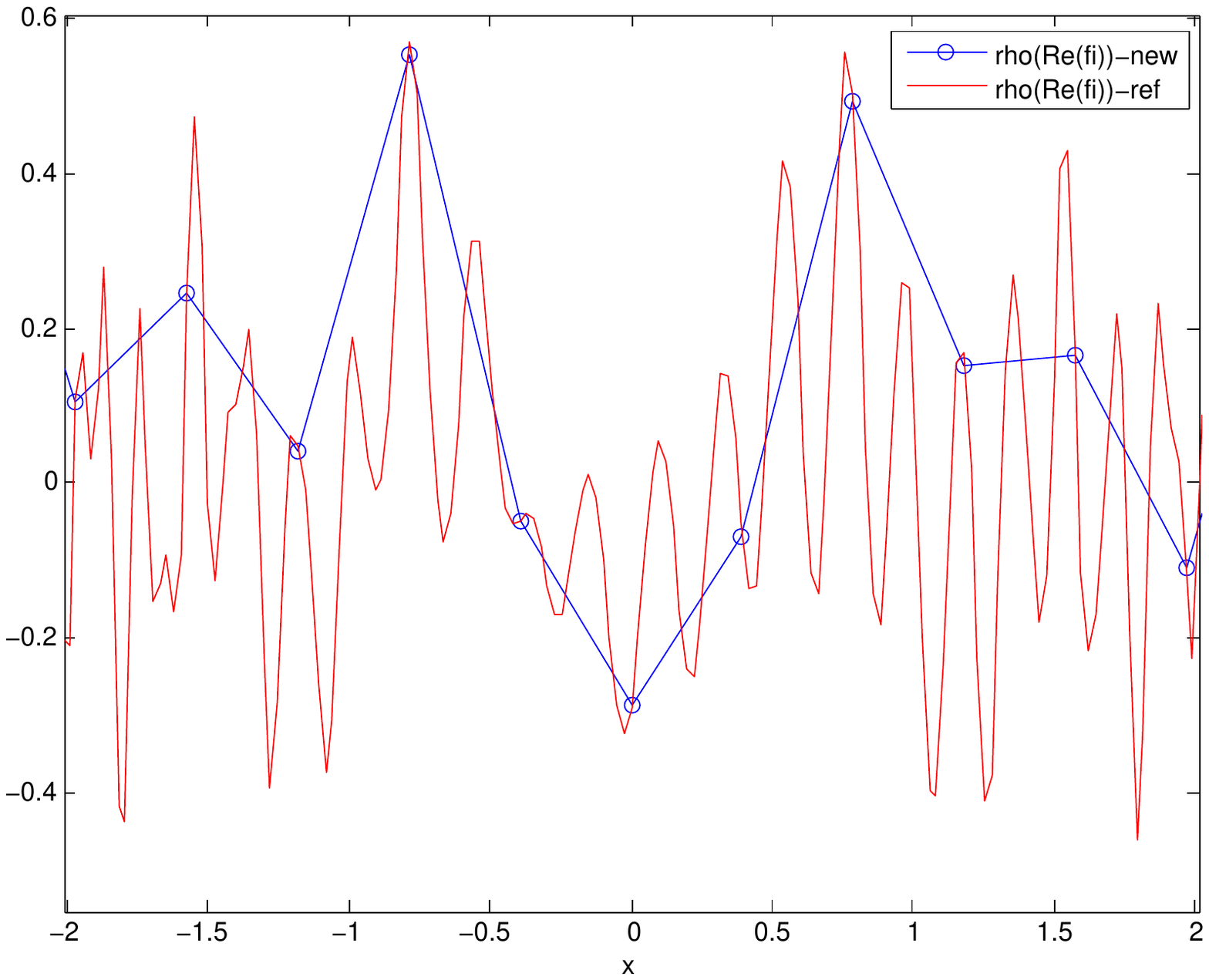}&
\hspace{-1.2cm}
\vspace{-2.3cm}
\includegraphics[width=0.4\linewidth]{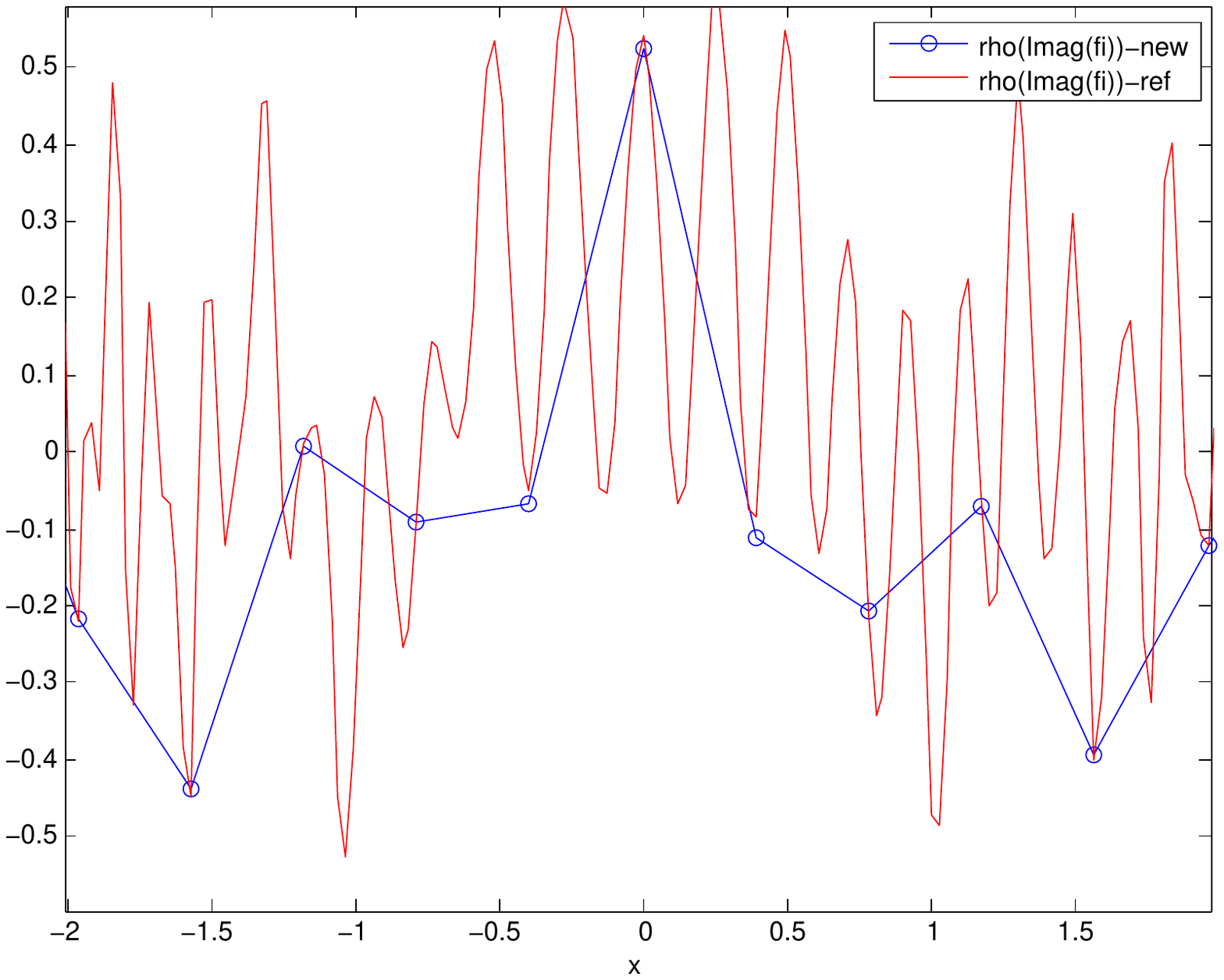}
\end{tabular}
\end{center}
\caption{$\varepsilon=1/32$. First line: space dependence of the densities $\rho^+$, Re($\rho^i$) and Im$(\rho^i)$.
Second line: space dependence of $\rho^-$, Re($\rho^i$) (zoom) and Im$(\rho^i)$ (zoom).  }
\label{fig2_eps1s32}
\end{figure}

\newpage

\begin{figure}
\vspace{-2cm}
\begin{center}
\begin{tabular}{ccll}
\includegraphics[width=0.4\linewidth]{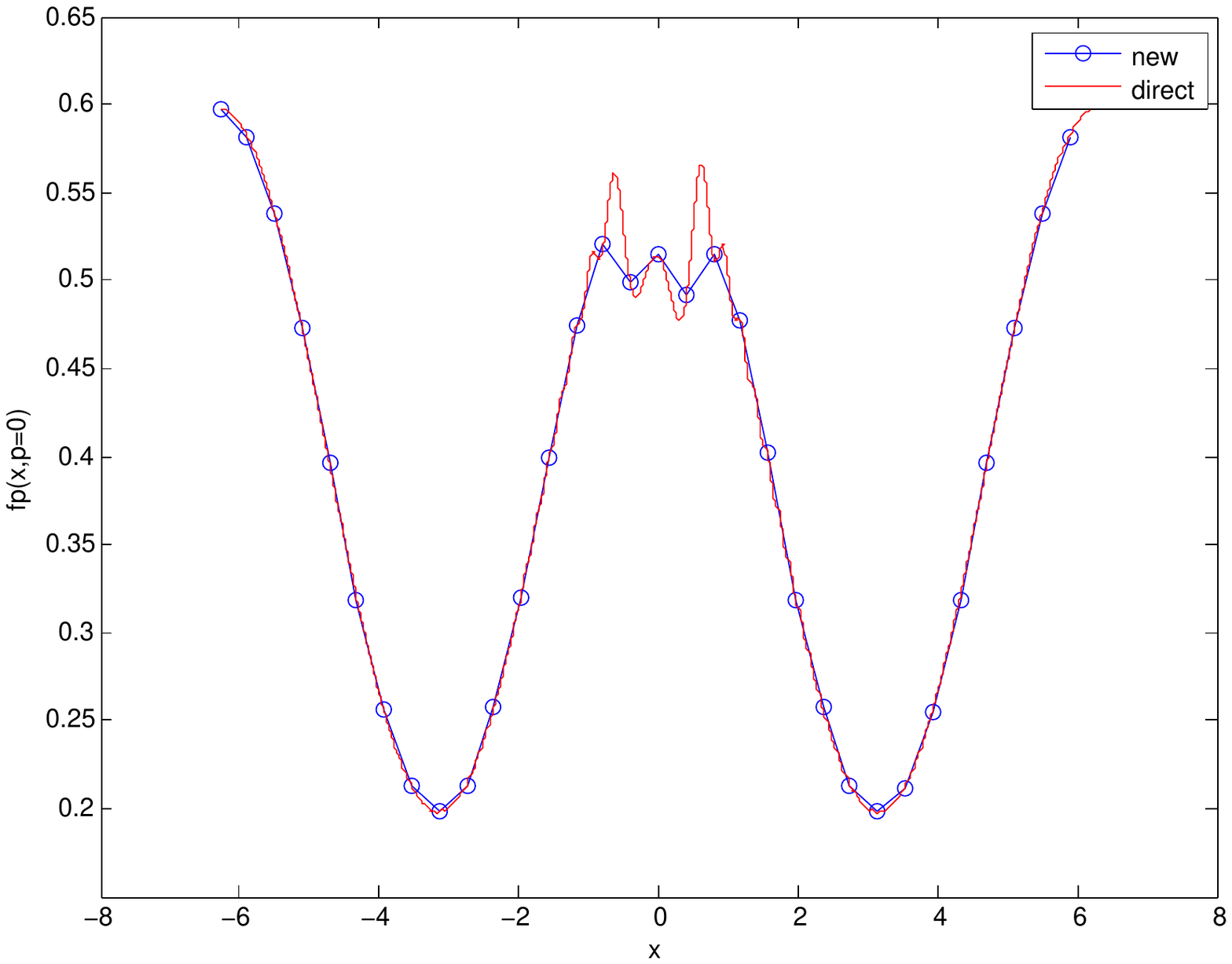}&
\hspace{-1.2cm}
\vspace{-3.4cm}
\includegraphics[width=0.4\linewidth]{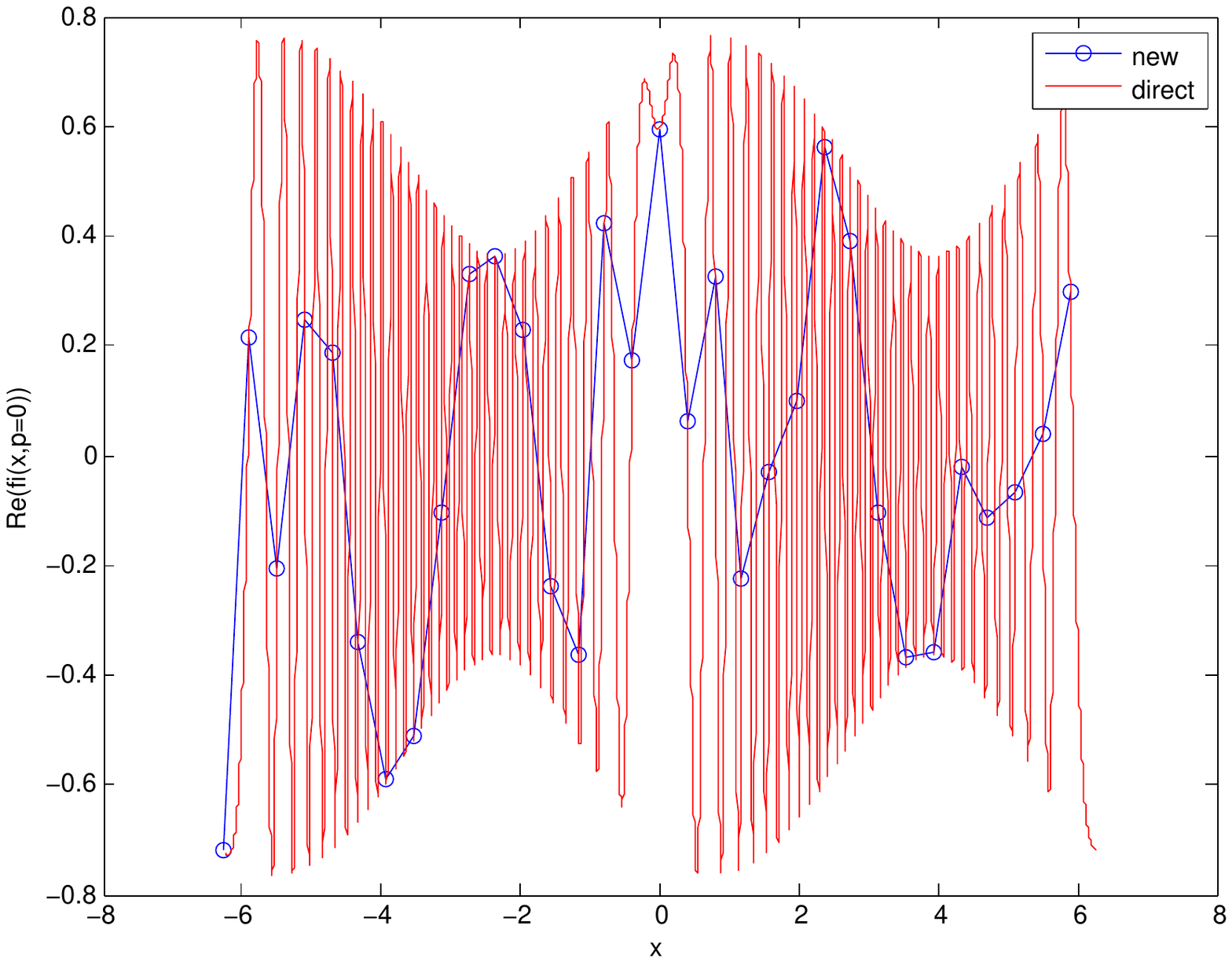}&
\hspace{-1.2cm}
\includegraphics[width=0.4\linewidth]{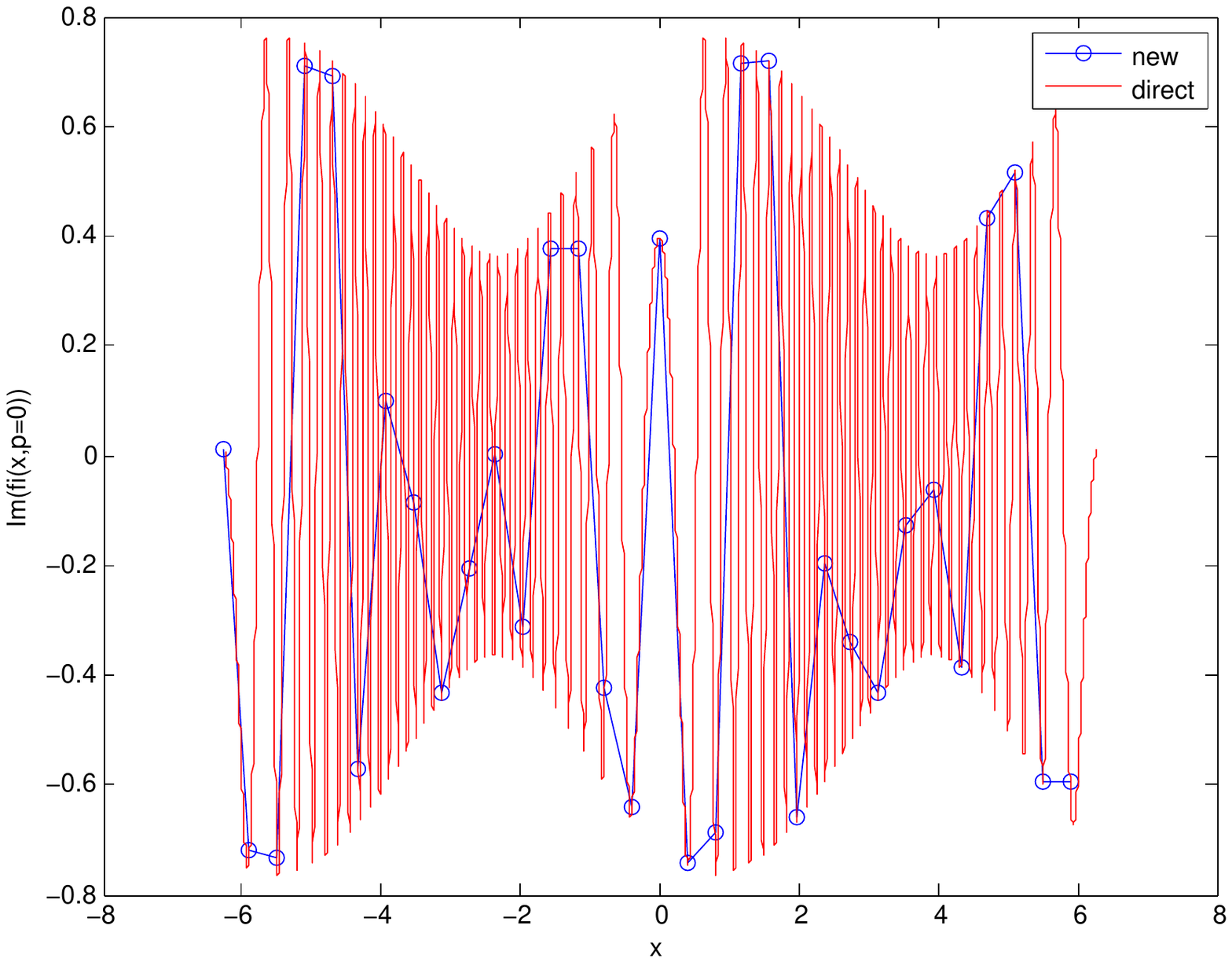}\\
\includegraphics[width=0.4\linewidth]{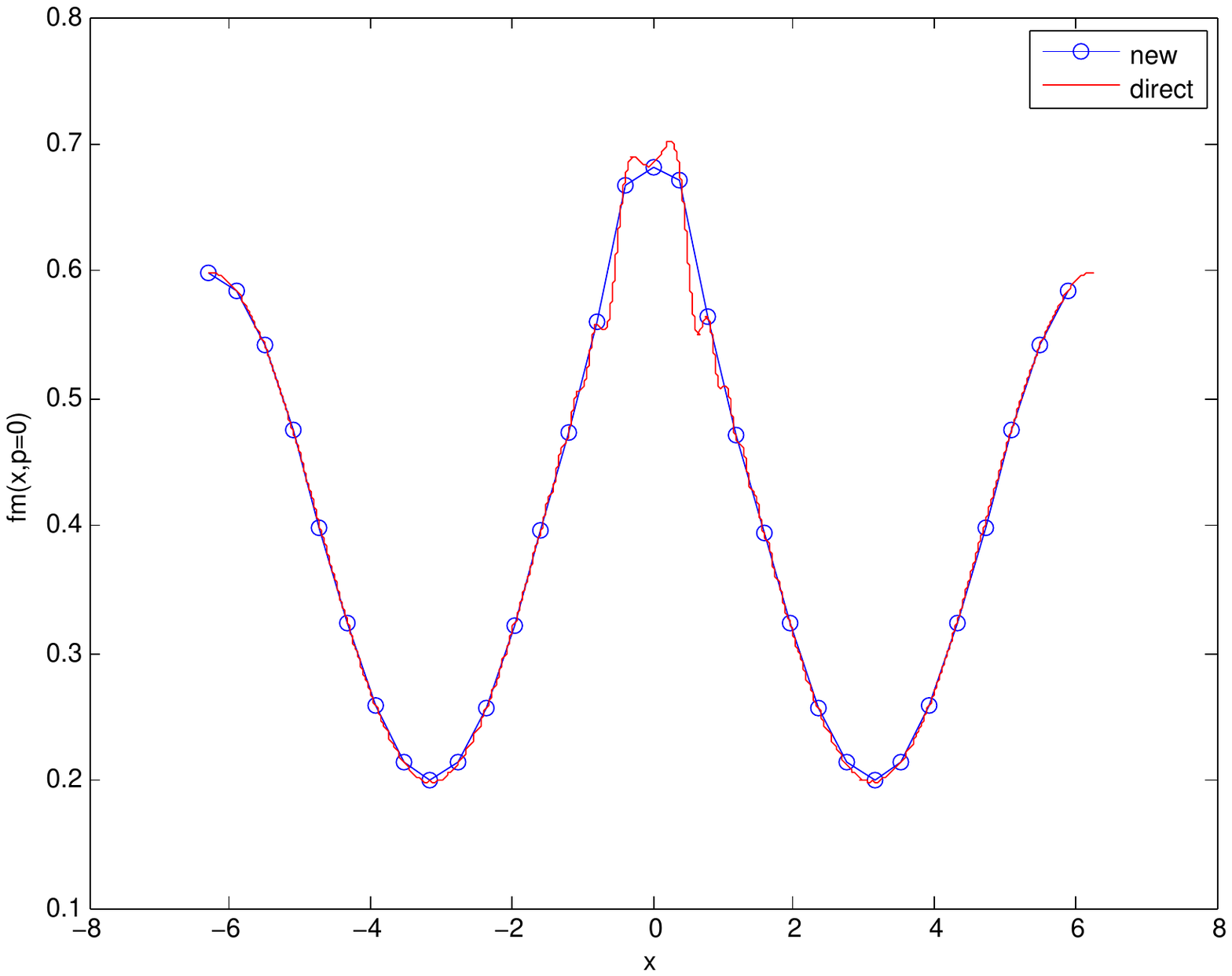}&
\hspace{-1.2cm}
\includegraphics[width=0.4\linewidth]{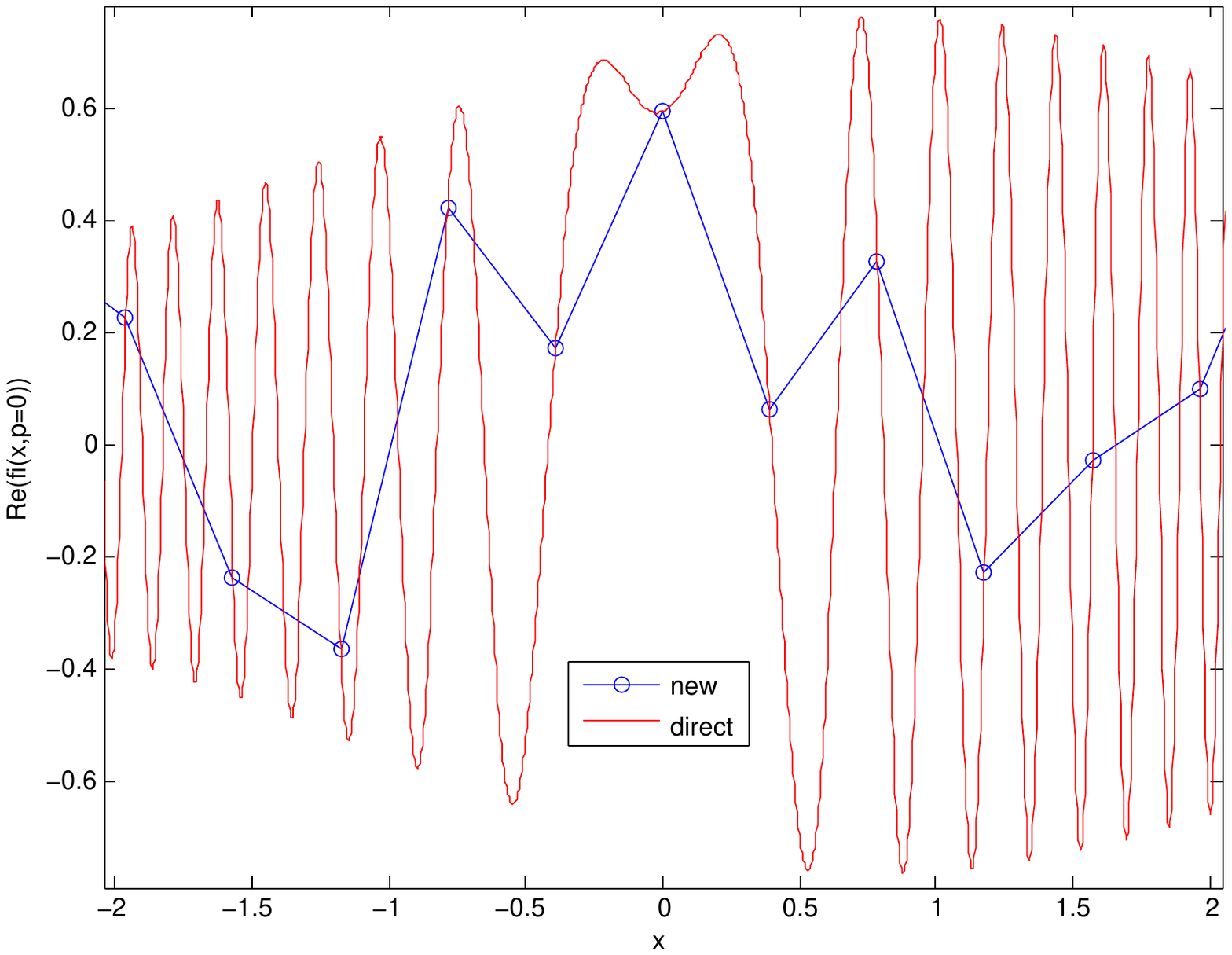}&
\hspace{-1.2cm}
\vspace{-2.3cm}
\includegraphics[width=0.4\linewidth]{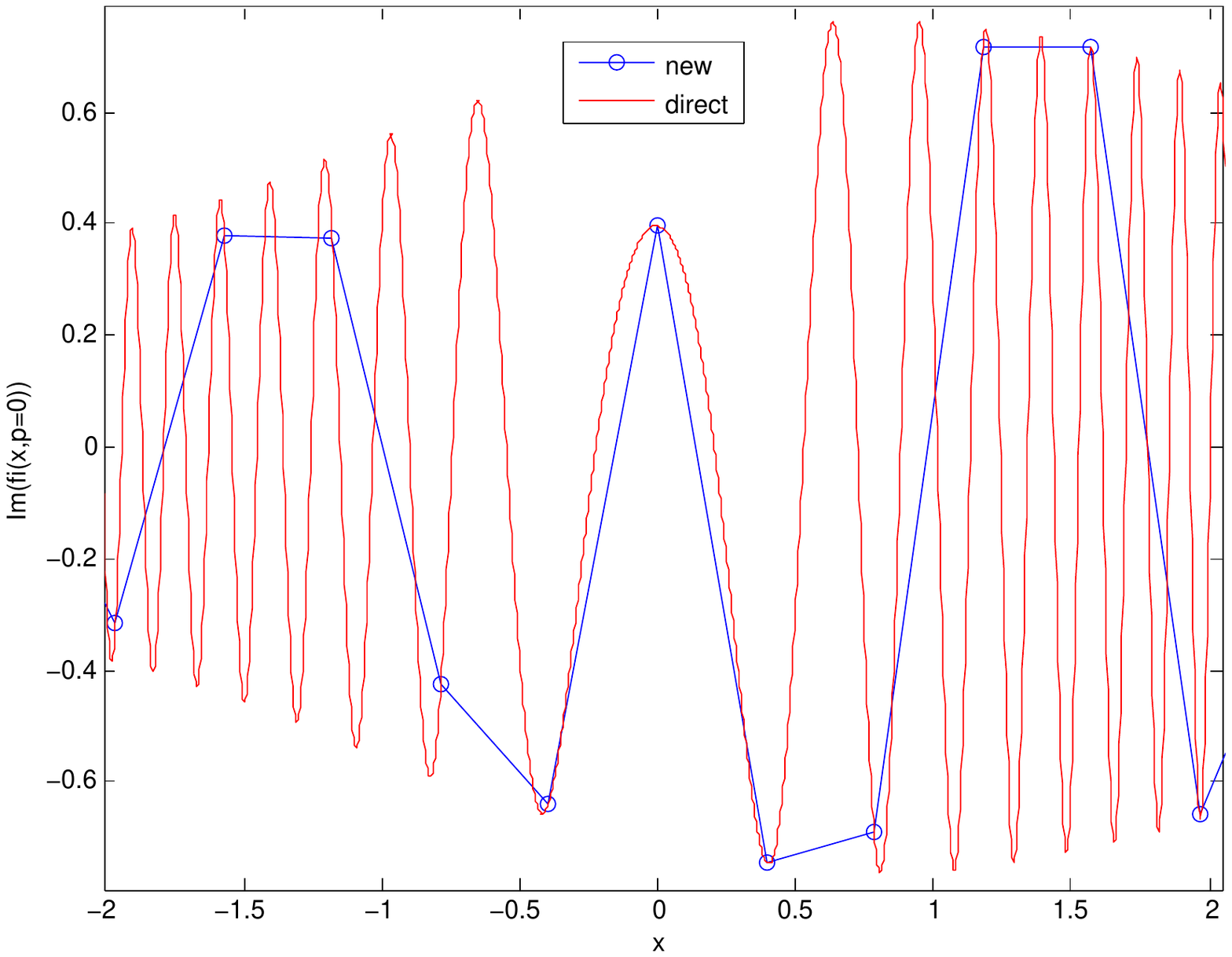}
\end{tabular}
\end{center}
\caption{$\varepsilon=1/256$. First line: space dependence of $f^+$, Re($f^i$) and Im$(f^i)$ at $p=0$.
Second line: space dependence of $f^-$, Re($f^i$) (zoom) and Im$(f^i)$ (zoom) at $p=0$.  }
\label{fig1_eps1s256}
\end{figure}

\begin{figure}
\vspace{-4cm}
\begin{center}
\begin{tabular}{ccll}
\includegraphics[width=0.4\linewidth]{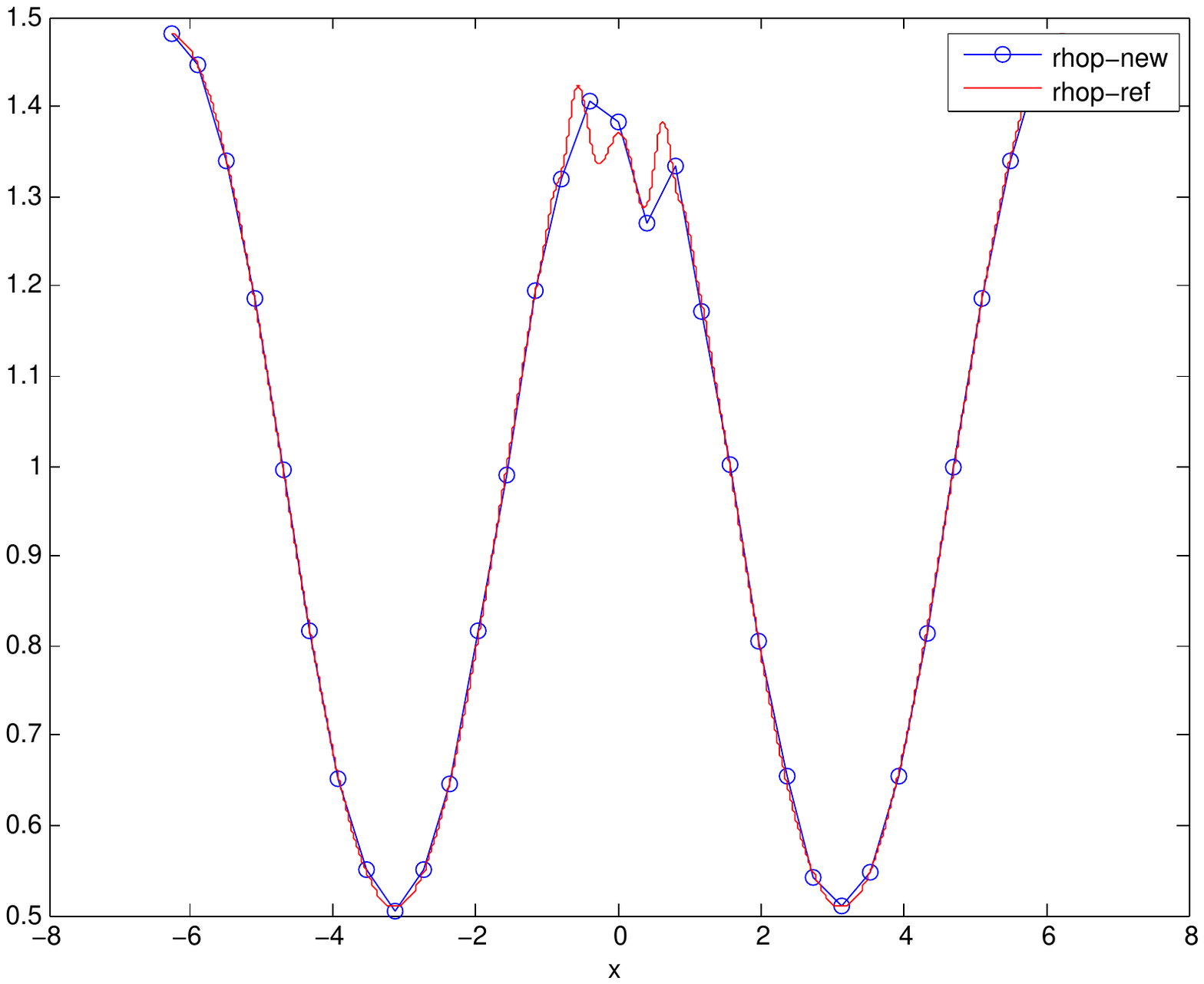}&
\hspace{-1.2cm}
\vspace{-3.4cm}
\includegraphics[width=0.4\linewidth]{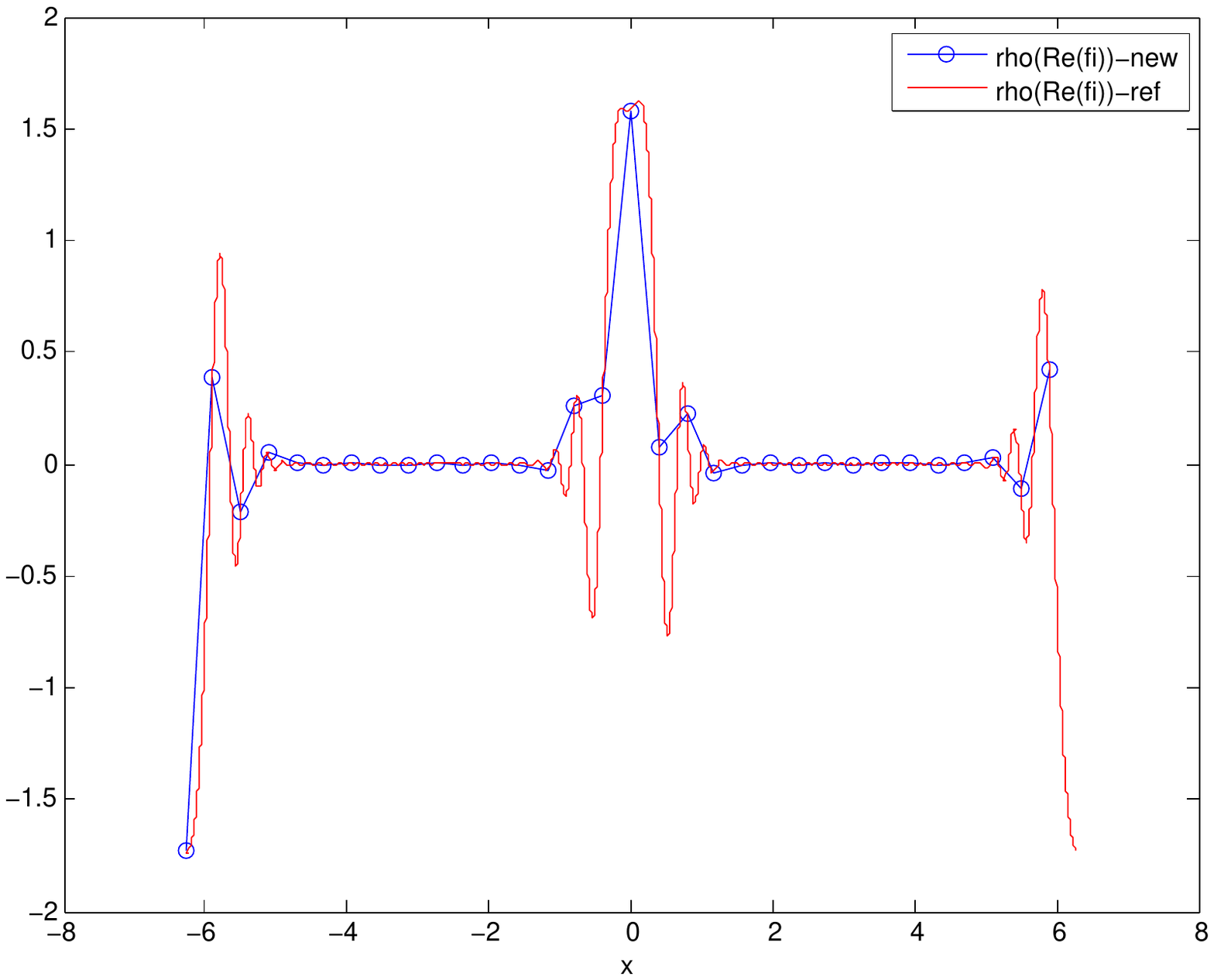}&
\hspace{-1.2cm}
\includegraphics[width=0.4\linewidth]{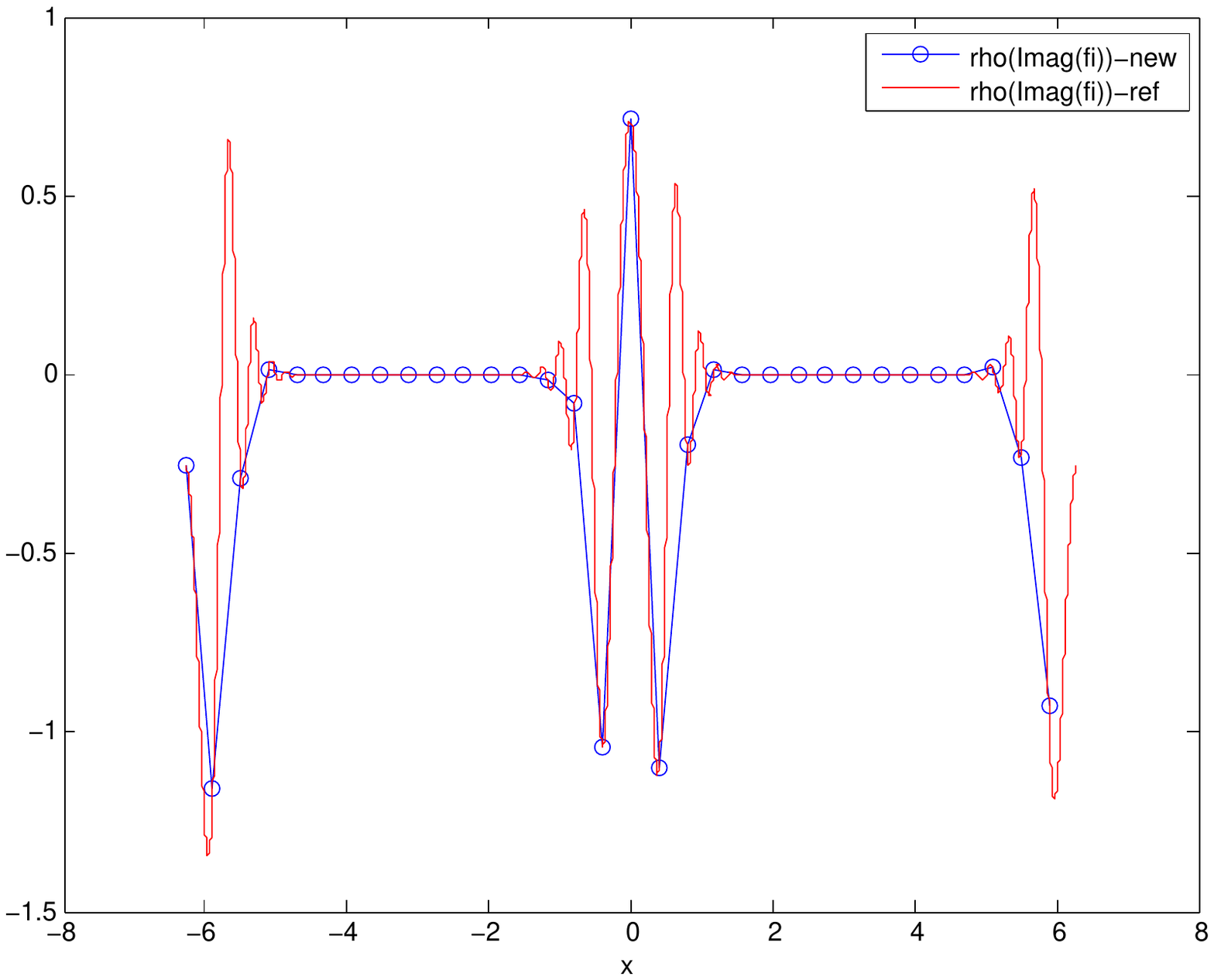}\\
\includegraphics[width=0.4\linewidth]{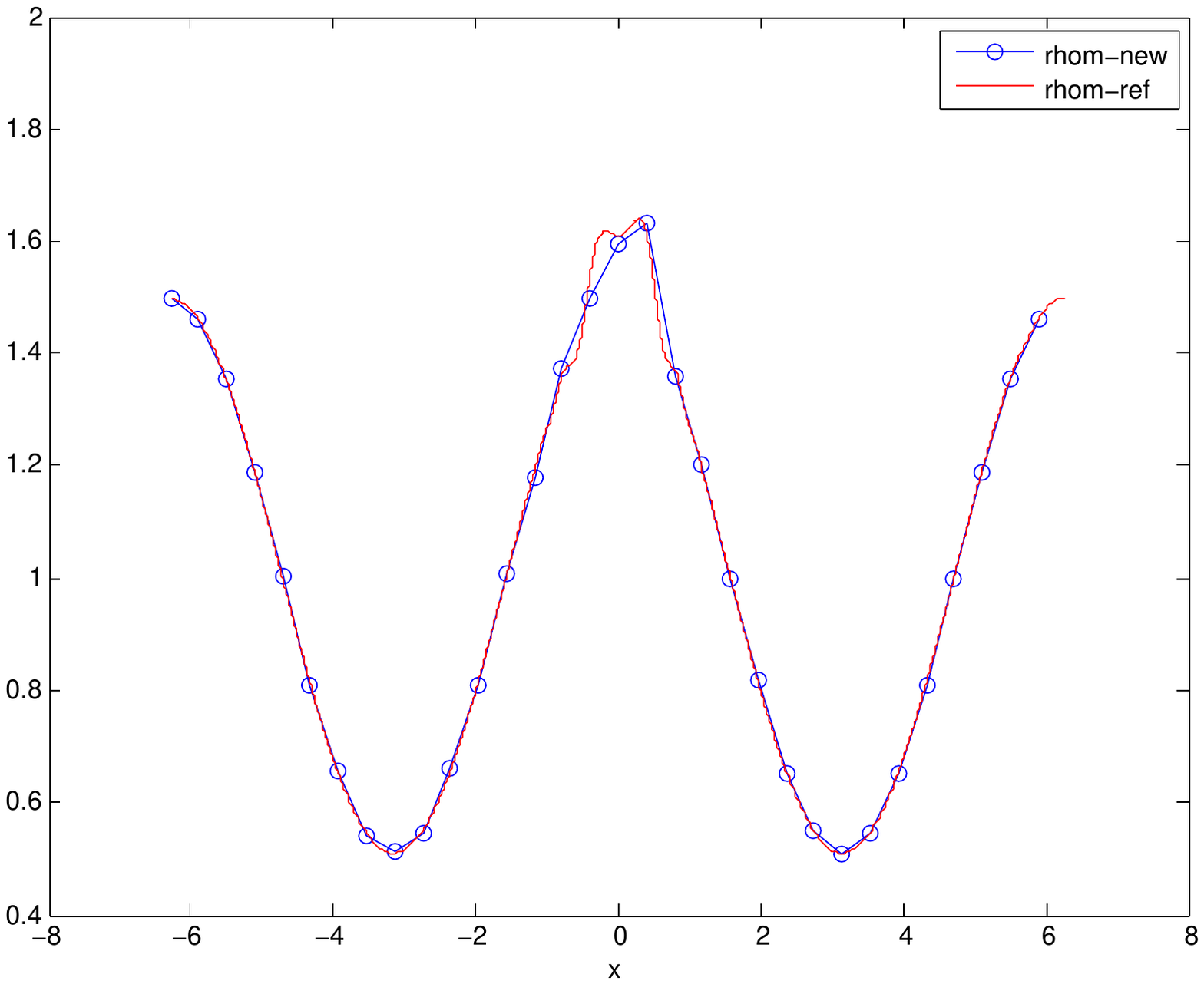}&
\hspace{-1.2cm}
\includegraphics[width=0.4\linewidth]{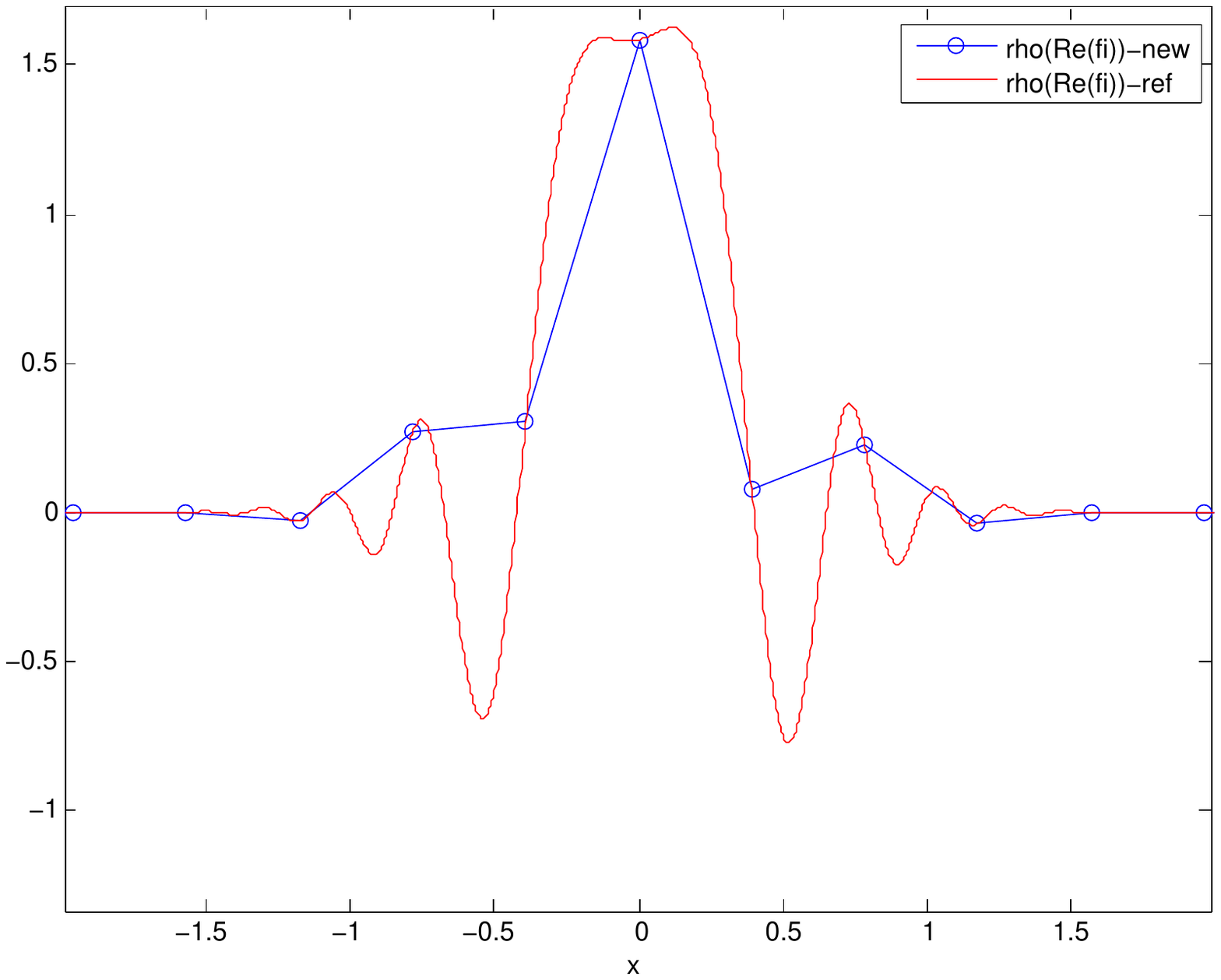}&
\hspace{-1.2cm}
\vspace{-2.3cm}
\includegraphics[width=0.4\linewidth]{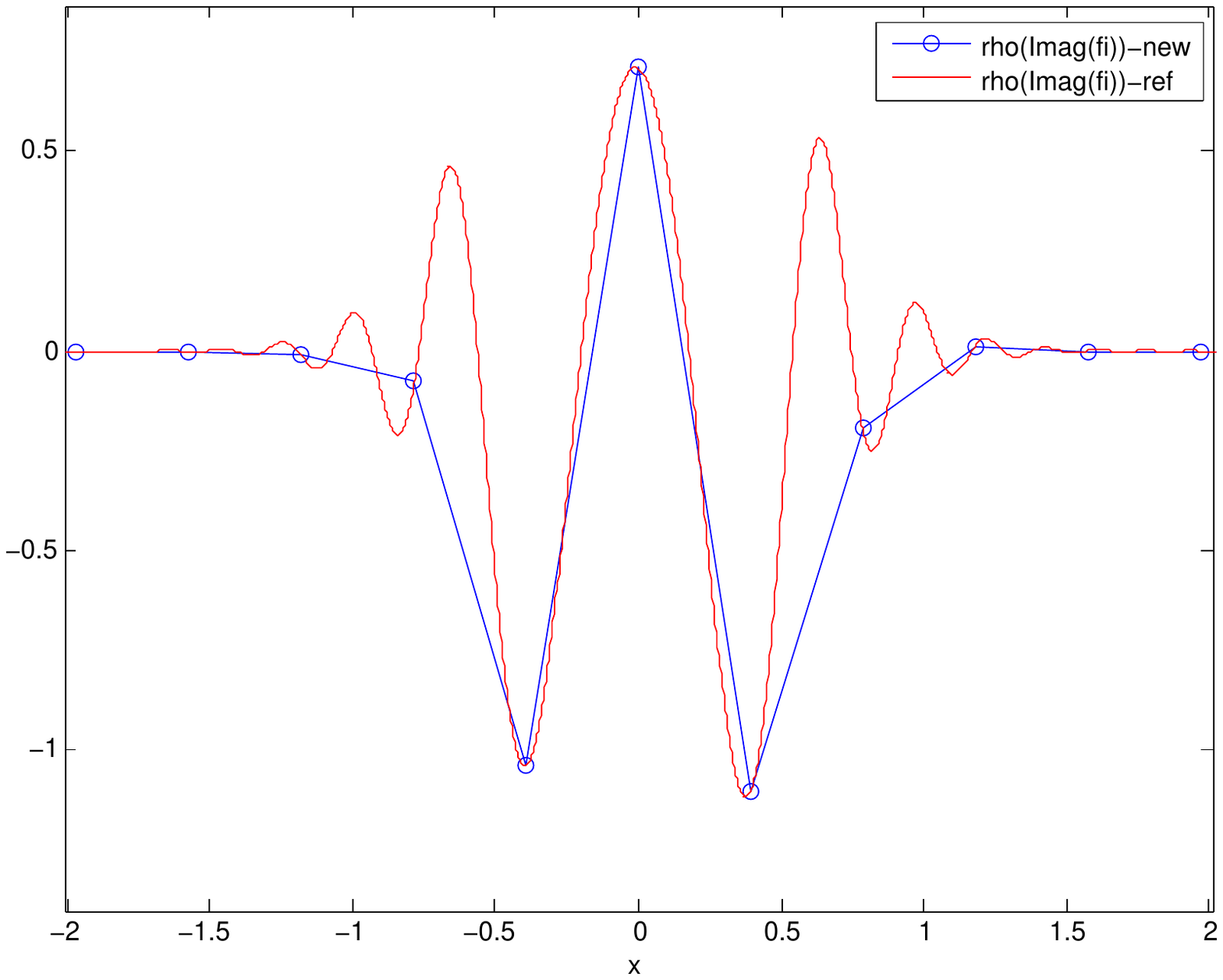}
\end{tabular}
\end{center}
\caption{$\varepsilon=1/256$. First line: space dependence of the densities $\rho^+$, Re($\rho^i$) and Im$(\rho^i)$.
Second line: space dependence of $\rho^-$, Re($\rho^i$) (zoom) and Im$(\rho^i)$ (zoom).  }
\label{fig2_eps1s256}
\end{figure}

\newpage

\section{Conclusion}
In this work, for a class of highly oscillatory hyperbolic systems of transport equations, we introduced a new numerical method which allows one to obtain accurate numerical solutions with mesh size and time step independent of the (possibly very small) wave length.  The central ideas include a geometric optics based
 ansatz, which builds the oscillatory phase into an independent variable, and a suitably chosen initial
 data derived from the Chapman-Enskog expansion. For a scalar model we prove that a first order
 approximation the  converges with a first order accuracy
 uniformly in the wave length, and the method is also extended for a system that arises in semiclassical modeling of surface hopping, which deals with quantum transition between different energy bands. Numerous numerical examples demonstrate that the method has
 the desired property of capturing the point-wise solutions of  highly oscillatory waves with mesh sizes much larger than the wave length. .

In the future, we will extend the method to higher dimensions, conduct more theoretical investigation on the
method for systems, and study other interesting non-adiabatic quantum dynamics problems.

\bibliographystyle{amsplain}
\bibliography{bc}

\end{document}